\documentclass[reqno,12pt,letterpaper]{amsart}

\RequirePackage{amsmath,amssymb,amsthm,graphicx,mathrsfs,url}
\RequirePackage[usenames,dvipsnames]{color}
\RequirePackage[colorlinks=true,linkcolor=ForestGreen,citecolor=BrickRed]{hyperref}
\RequirePackage{amsxtra}
\usepackage{comment}
\usepackage{enumitem}
\usepackage{cases}

\relax

\numberwithin{equation}{section}

\newtheorem{theo}{Theorem}
\newtheorem{prop}{Proposition}[section]
\newtheorem{defi}[prop]{Definition}
\newtheorem{lemm}[prop]{Lemma}

\def\Remark{\noindent\textbf{Remark.}\ }
\def\Remarks{\noindent\textbf{Remarks.}\ }

\newtheorem{cor}[prop]{Corollary}
\newtheorem{lem}[prop]{Lemma}
\newtheorem{defn}[prop]{Definition}
\newtheorem{ass}[prop]{Assumption}

\DeclareMathOperator{\diag}{diag}

\usepackage{url}

\newcommand{\eps}{{\varepsilon}}

\newcommand{\RR}{{\mathbb R}}

\newcommand\cD{{\mathcal  D}}
\newcommand\cU{{\mathcal  U}}

\newcommand\cJ{{\mathcal  J}}

\newcommand\cG{{\mathcal  G}}

\newcommand\cO{{\mathcal O}}

\newcommand{\os}{\overline{s}}
\newcommand{\ox}{\overline{x}}
\newcommand{\oy}{\overline{y}}
\newcommand{\otheta}{\overline{\theta}}
\newcommand{\oxi}{\overline{\xi}}

\newcommand{\urho}{\underline{\rho}}

\newcommand{\usigma}{\underline{\sigma}}

\newcommand{\utau}{\underline \tau}

\newcommand{\ueta}{\underline \eta}

\newcommand{\unu}{{\underline \nu}}

\def\eps{\epsilon }

\newcommand\adots{\mathinner{\mkern2mu\raise1pt\hbox{.}
\mkern3mu\raise4pt\hbox{.}\mkern1mu\raise7pt\hbox{.}}}

\let\oldtocsection=\tocsection
\let\oldtocsubsection=\tocsubsection
\renewcommand{\tocsection}[2]{\hspace{0em}\oldtocsection{#1}{#2}}
\renewcommand{\tocsubsection}[2]{\hspace{1em}\oldtocsubsection{#1}{#2}}

\numberwithin{equation}{section}

\begin{document}

\title[Transport of nonlinear oscillations]{Transport of nonlinear oscillations along rays that graze a convex obstacle to any order}

\author{Jian Wang}
\email{wangjian@email.unc.edu}
\address{Department of Mathematics, University of North Carolina, Chapel Hill, NC 27599}
\author{Mark Williams}
\email{williams@math.unc.edu}
\address{Department of Mathematics, University of North Carolina, Chapel Hill, NC 27599}


\begin{abstract}
 We provide a geometric optics description in spaces of low regularity, $L^2$ and $H^1$, of the transport of  oscillations in solutions to 
 linear and some semilinear second-order hyperbolic boundary problems
along rays that graze the boundary of a convex obstacle to arbitrarily high finite or infinite order.    The fundamental motivating example  is the case where the spacetime manifold is $M=(\mathbb{R}^n\setminus \mathcal{O})\times \mathbb{R}_t$, 
where $\mathcal{O}\subset \mathbb{R}^n$ is an open convex  
obstacle with $C^\infty$ boundary, and the governing hyperbolic operator is the wave operator $\Box:=\Delta-\partial_t^2$.

\end{abstract}

\maketitle

{
\hypersetup{linkcolor=NavyBlue}
\tableofcontents
}

\newpage

\section{Introduction}\label{intro}

In this paper we provide a description in spaces of low regularity, $L^2$ and $H^1$, of the transport of  oscillations in solutions to 
 linear and some semilinear second-order hyperbolic boundary problems
along rays that graze the boundary of a convex obstacle to arbitrarily high finite or infinite order.    The fundamental motivating example  is the case where the spacetime manifold is $M=(\mathbb{R}^n\setminus \mathcal{O})\times \mathbb{R}_t$, 
where $\mathcal{O}\subset \mathbb{R}^n$ is an open convex  
obstacle with $C^\infty$ boundary, and the governing hyperbolic operator is the wave operator $\Box:=\Delta-\partial_t^2$.   Our main theorem, Theorem \ref{mt2}, is proved in greater generality than this, but it involves two assumptions that can be difficult to verify.   In \S \ref{convexobstacle} we show that the theorem applies to describe the diffraction of oscillatory plane waves by a variety of convex obstacles for which those assumptions can be verified.  

We approach this problem from the point of view of \emph{geometric optics} in the sense of \cite{jmr1995asens,jmr1996cpam}.\footnote{We use ``geometric optics" roughly to refer to an approach where approximate solutions to problems with highly oscillatory boundary data or initial data are constructed by solving eikonal equations to obtain phases and transport equations to obtain profiles, and where a rigorous error analysis is done to show that high frequency approximate solutions are close to exact solutions in some appropriate norm on a fixed time interval independent of wavelength.}   
The papers most closely related to this paper appear to be those of Cheverry \cite{cheverry1996} and  Dumas \cite{dumas2002}, which applied geometric optics to obtain results similar to the ones studied here, but in problems where only first-order grazing is allowed.    In particular, each of those papers describes the behavior of solutions in spaces of low regularity.  

With regard to linear hyperbolic boundary problems where only first-order grazing is allowed, we recall the papers of Melrose \cite{melrose1975duke} and Taylor \cite{taylor1976cpam}, which construct microlocal parametrices to describe the propagation of $C^\infty$ singularities (wavefront sets) near grazing points, and the book of H\"ormander \cite{hormander3}, which gives such a description based just on energy estimates.  The papers of Melrose and Sj\"ostrand \cite{melrosesjostrand1978cpam, melrosesjostrand1982cpam},  study the propagation of $C^\infty$ singularities along ``generalized bicharacteristics"  which can reflect off the boundary, graze the boundary to any order, or glide along the boundary.   
    
The diffraction of conormal waves in semilinear problems where only first-order grazing is allowed is studied in the paper of Melrose, S\'a Barreto, and Zworski \cite{msz1996aster} in conormal spaces of high regularity.    In both linear and nonlinear problems where higher-order grazing is allowed,  
it seems out of reach at present to describe diffraction using geometric optics in spaces of high regularity.   Roughly speaking, working with spaces of low regularity is more feasible, since much of the complicated (and interesting) behavior that is now too hard to describe is invisible in such spaces.    The papers \cite{jmr1996cpam, jmr2000mams} use spaces of low regularity to describe the behavior of nonlinear oscillations beyond caustics.

In order to describe and state our main result with a minimum of preparation,  we work now in coordinates $(x,y,t)\in\mathbb{R}^{n+1}$ and dual coordinates $(\lambda,\eta,\tau)$ where $t$ is the time variable and $x=0$ defines the (noncharacteristic)  boundary.  
In \S \ref{DandA} we state definitions, assumptions, and the main theorem, Theorem \ref{mt2},  more precisely  and in a coordinate-free way.  

Consider a second-order operator $P(x,y,t,\partial_{x,y,t})$ with $C^\infty$ coefficients, strictly hyperbolic with respect to $t$,
whose principal symbol has the form 
\begin{align}\label{d1}
p(x,y,t,\lambda,\eta,\tau)=\lambda^2+q(x,y,t,\eta,\tau),
\end{align}
where $q(x,y,t,\cdot,\cdot)$ has signature $(n-1,1)$.    On a domain 
\begin{align*}
\Omega_T=\{(x,y,t)\in\mathbb{R}^{n+1} \ | \ x\geq 0, -T\leq t\leq T\}, \;T>0,
\end{align*}
we study the continuation problem 
\begin{subnumcases}{\label{e0a}}
    Pu^\eps=f(x,y,t,u^\eps,\nabla_{x,y,t} u^\eps) & in $\Omega_T$, \label{e0aa}\\
    u^\eps(0,y,t)=0 & on $\Omega_T\cap\{x=0\}$, \label{e0ab}\\
    u^\eps=v_\eps\sim_{H^1} u^1(x,y,t)+\eps U_1(x,y,t,\phi_i/\eps) & on $\Omega_{[-T,-T+\delta]}$ \label{e0ac}
\end{subnumcases}
where $\Omega_{[-T,-T+\delta]}:=\{(x,y,t) \ | \ x\geq 0, -T\leq t\leq -T+\delta\}$ for some small $\delta>0$, and 
the meaning of $\sim_{H^1}$ is explained in Definition  \ref{meaning}.  We assume given
\begin{align*}
\begin{gathered}
v_\eps(x,y,t) \text{ and } u^1(x,y,t)\in H^1(\Omega_{[-T,-T+\delta]}),
\text{ and }U_1(x,y,t,\theta_i) \in L^2(\Omega_{[-T,-T+\delta]}\times\mathbb{T}), 
\end{gathered}
\end{align*}
where each of $v_\eps$, $u$, $U_1$ has compact $(x,y,t)$-support strictly away from $x=0$,  $U_1$ is periodic in $\theta_i$ of mean zero, and
$$\partial_{\theta_i} U_1\in L^2(\Omega_{[-T,-T+\delta]}\times\mathbb{T}).$$
The function $f$ is assumed uniformly Lipschitzean in its last arguments (Definition \ref{A0z}) and satisfies $f(x,y,t,0,0)=0$.

\Remarks 
1. The problem \eqref{e0a}, where $P$ has principal symbol \eqref{d1},  is a local model or \emph{standard form}  to which the problem considered in  Theorem \ref{mt2} can be reduced by a local change of variables near $(0,0,0)$; see Definition \ref{sform} and \S \ref{reduce}.   

\noindent
2. The uniformly Lipschitzean assumption on $f$, Assumption \ref{A0z}, allows one to prove the existence of a unique solution 
$u^\eps\in H^1(\Omega_T)$ by a simple Picard iteration.   The result of Kreiss \cite{kreiss1970cpam} provides the estimate \eqref{kreiss} needed to obtain both existence and convergence of the iterates on $\Omega_T$ for some sufficiently small $T>0$ independent of $\eps$.  The definition of $\sim_{H^1}$ plays no role in this proof.

The function $U_1(x,y,t,\phi_i/\eps)$ in \eqref{e0a}  supplies the incoming oscillations. The surfaces of constant phase are the spacetime surfaces $\phi_i(x,y,t)=c$, where the function $\phi_i$, called the \emph{incoming phase}, is a $C^\infty$ function that satisfies the \emph{eikonal} equation 
\begin{align*}
p(x,t,y,\nabla_{x,y,t}\phi_i)=0 \text{ on }U, 
\end{align*}
where $U$ is some $\mathbb{R}^{n+1}$-neighborhood of $0$ that we take to be an open ball centered at $0$.    The phase $\phi_i$ is constructed to satisfy
\begin{align*}
\nabla_{x,y,t}\phi_i(x,y,t)\neq 0, \text{ for all } (x,y,t)\in U.
\end{align*}
Let 
\begin{align*}
  U_{\det}\subset U\cap \{x\geq 0\} \text{ with }0\in U_{\det}
\end{align*}
be a domain of determinacy for continuation problems in $\mathbb{R}^{n+1}_+=\{x\geq 0\}$ determined by $P$ and the Dirichlet boundary condition \eqref{e0ac}.  
We assume that $U_1$ in \eqref{e0ac} satisfies
\begin{align*}
\mathrm{supp}_{(x,y,t)}U_1\subset \mathring{U}_{\det}.
\end{align*}

\begin{figure}[t]
\begin{center}
    \includegraphics[scale=0.5]{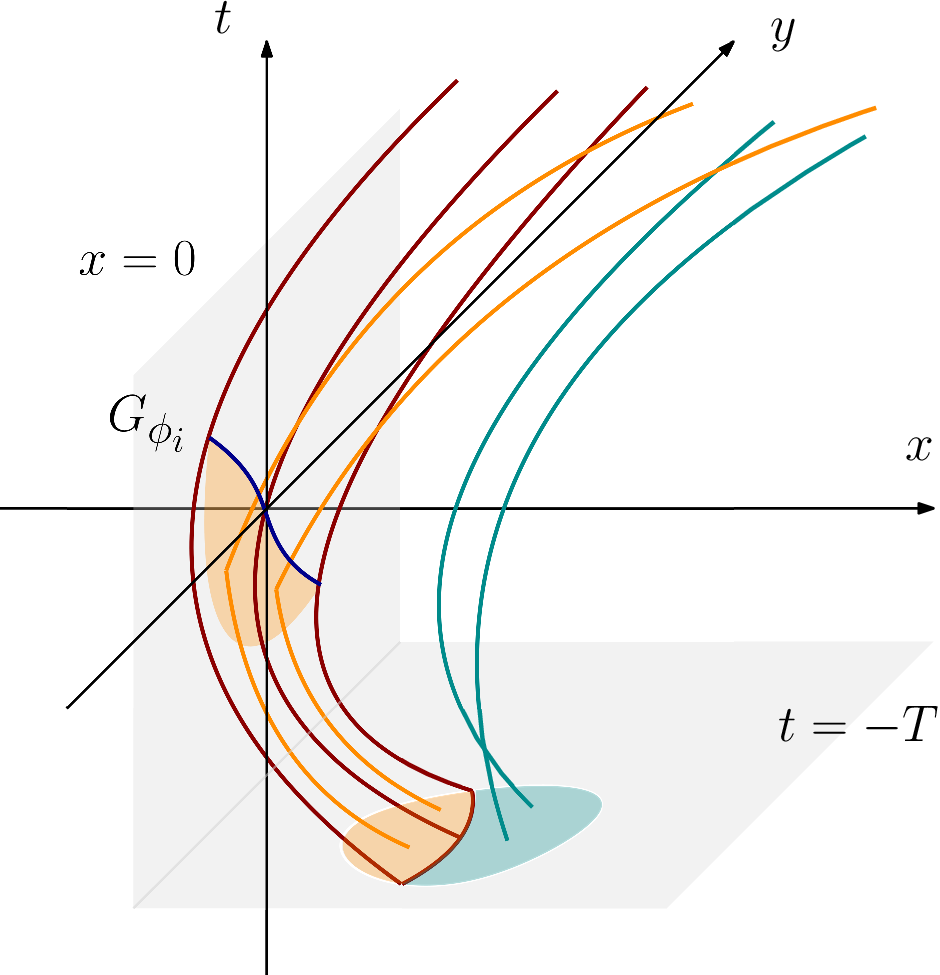}
    \caption{Characteristics associated to $\phi_i$ and $\phi_r$. The yellow curves reflect off the boundary, the red curves graze the boundary, and the green curves do not touch the boundary. The dark curve on $\{x=0\}$ is the grazing set $G_{\phi_i}$.}
    \label{rays}
\end{center}
\end{figure}

We will see that the oscillations are transported along characteristics of $p$ associated not only to $\phi_i$ but also to an associated \emph{reflected phase} $\phi_r$.
The characteristics associated to $\phi_k$, $k=i,r$, are integral curves of the \emph{characteristic vector field of $\phi_k$}:
\begin{align}\label{ezb}
T_{\phi_k}:=\left(2\lambda\partial_x+\partial_{\eta,\tau}q(x,y,t,\eta,\tau)\partial_{y,t}\right)|_{(\lambda,\eta,\tau)=\nabla_{x,y,t}\phi_k}.
\end{align}
These curves are projections onto spacetime of null bicharacteristics of $p$ associated to $\phi_k$; see Definition \ref{q6d} and the Remark after Definition \ref{q6d}.
The operator $P$ and the incoming phase $\phi_i$ are chosen so that {some} of the characteristics of $\phi_i$ emerging from points in the $(x,y,t)$-support of  $U_1$ as in \eqref{e0ac}  graze the boundary $x=0$ to some finite or possibly infinite order.  Each such grazing characteristic is tangent to $x=0$ at a single spacetime point, and nearby points on the characteristic lie in $x>0$.   The order of tangency is what we mean by the order of ``grazing".  We arrange so that the origin $0\in\Omega_T$ is such a point of tangency.   Near each grazing characteristic there are transversal incoming characteristics that reflect off the boundary; see Figure \ref{rays}. These definitions are made precise in \S \ref{decomposition} and \S \ref{reflectedphase}.

The main theorem is stated in terms of incoming and reflected profiles, $U_i(x,y,t,\theta_i)$ and $U_r(x,y,t,\theta_r)$, that describe the transport of oscillations.   Each function $U_k$ for $k=r,i$  is the unique mean zero periodic primitive in $\theta_k$ of a function $W_k(x,y,t,\theta_k)\in L^2(\Omega_T\times \mathbb{T})$ that is constructed to satisfy the transport equations \eqref{e6}--\eqref{e8} of \S \ref{pe}.

We proceed to define particular subsets of $\Omega_T$, $J_r$ and $J_i$, that {contain} the supports of $W_r$ and $W_i$. 
From \eqref{ezb} we know that characteristics of $\phi_i$ are tangent to $x=0$ precisely at points of the \emph{grazing set}
\begin{align*}
G_{\phi_i}:=\{(x,y,t)\in U \ | \ \partial_x\phi_i(0,y,t)=0\}.
\end{align*}

\begin{ass}[Regularity of the grazing set]\label{A02}
The set $G_{\phi_i}$ is a codimension two $C^1$ submanifold of $\mathbb{R}^{n+1}$ near $0\in G_{\phi_i}$.   That is, there exists a $C^1$ function $\zeta(x,y,t)$ defined near $0$ such that $\nabla\zeta(0,0,0)\neq 0$ and
\begin{align*}
G_{\phi_i}=\{(x,y,t)\in U \ | \ x=0 \text{ and }\zeta=0\}.
\end{align*}
Moreover, the vector field $T_{\phi_i}$ is transverse to the $n$-dimensional hypersurface $\{\zeta=0\}$ at $0$.
\end{ass}

\Remark 
When the origin is a point of \emph{first}-order tangency,  it was shown in \cite{cheverry1996} that Assumption \ref{A02} always holds and that $\zeta$ can be taken to be a $C^\infty$ function.   When the origin is a point of higher than first-order tangency, verifying this assumption can be difficult.  It is not clear that Assumption \ref{A02} always holds even when $P$ is the wave operator acting in the exterior of a convex obstacle and the incoming phase $\phi_i$ is linear.  We verify this assumption in \S \ref{va2} for a number of examples in all dimensions involving all orders of tangency.

Let $\mathrm{SB}=\mathrm{SB}_+\cup \mathrm{SB}_-$ be the $C^1$ hypersurface in $\mathbb{R}^{n+1}$  which is the flowout  of $G_{\phi_i}$ along characteristics of $\phi_i$.   More precisely, $\mathrm{SB}$ is the union of the forward and backward flowouts of $G_{\phi_i}$, $\mathrm{SB}_\pm$ respectively,  along integral curves of $T_{\phi_i}$.\footnote{By the ``forward flowout" we mean the flowout along integral curves for which $t$ increases as the curve parameter increases.}  We call $\mathrm{SB}_+$ the \emph{shadow boundary}; see Definition \ref{SB}.

\begin{figure}[t]
\begin{center}
    \includegraphics[scale=0.45]{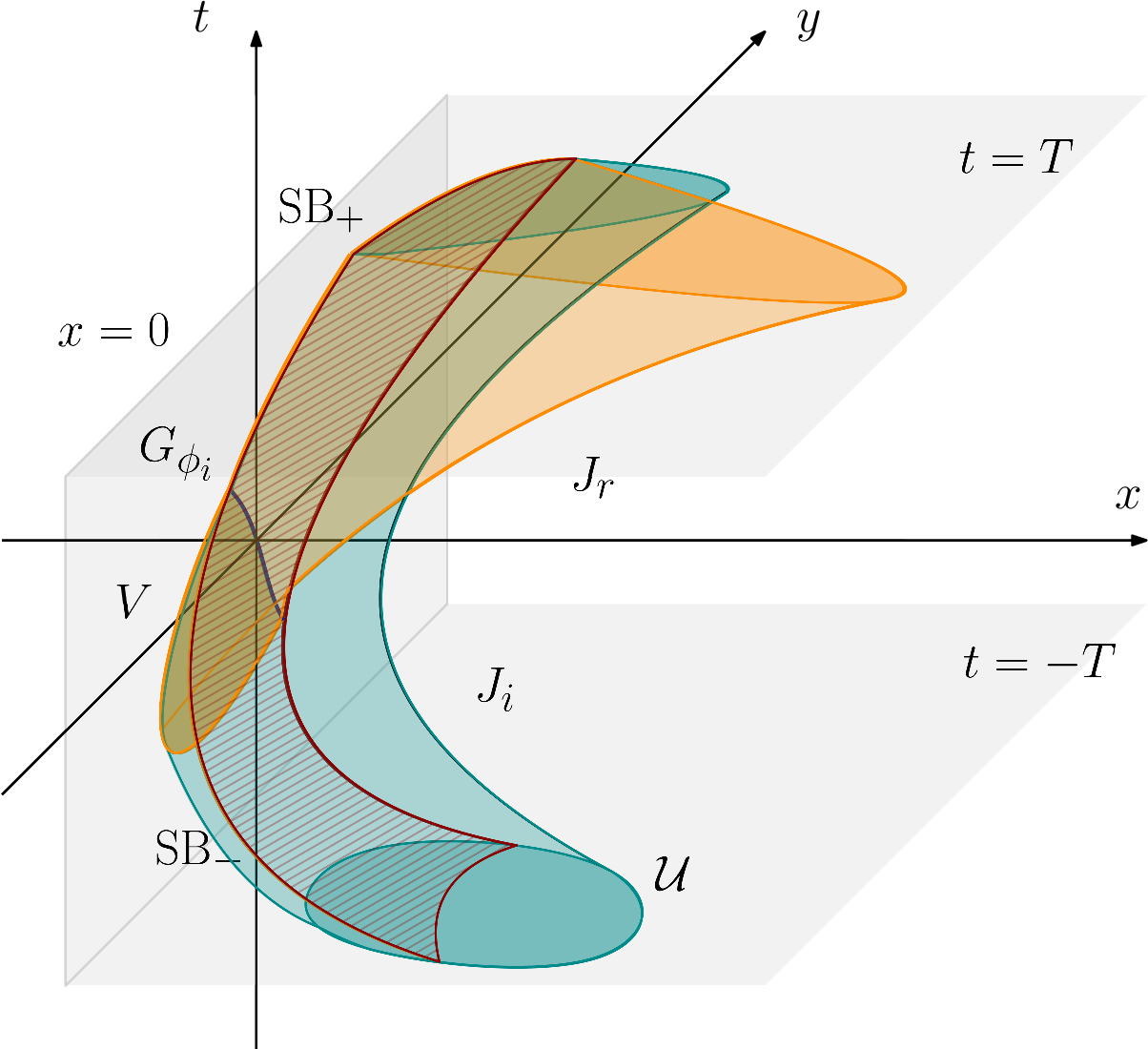}
    \caption{Green domain: forward flowout in $\{x\geq 0\}$ of the characteristic vector field $T_{\phi_i}$ associated to the incoming phase $\phi_i$. Yellow domain: forward flowout of the characteristic vector field $T_{\phi_r}$ associated to the reflected phase $\phi_r$. Dark curve on the boundary $\{x=0\}$: the grazing set $G_{\phi_i}$. Red surfaces: $\mathrm{SB}_{\pm}$, forward and backward flowouts of the grazing sets along characteristics of $T_{\phi_i}$. }
    \label{shadow}
\end{center}
\end{figure}

Set 
\begin{align*}
W_1(x,y,-T,\theta_i):=\partial_{\theta_i}U_1(x,y,-T,\theta_i)\text{ for }U_1\text{ as in }\eqref{e0ac}.
\end{align*}
We are interested in the behavior of oscillations transported by rays that reflect off and graze $\partial M$ near $0$, so it is no restriction to assume that $\mathrm{supp}_{(x,y,-T)}W_1$ is small and located near $\mathrm{SB}_-\cap \{t=-T\}.$ 
For $T>0$ small this allows us to choose an $n$-dimensional closed ball $\cU$ such that
\begin{align}\label{ezg}
\begin{split}
\cU\subset {U}_{\det}\cap \{t=-T\}\text{ and }\mathrm{supp}_{(x,y,-T)}W_1\subset \mathring{\cU};
\end{split}
\end{align}
\noindent see Figure \ref{shadow}.
For points $(x',y',-T)\in \cU$ and for $s\geq 0$ let 
\begin{align*}
(x,y,t)=Z_i(s,x',y'), \text{ where }Z_i(0,x',y')=(x',y',-T)
\end{align*}
denote the forward flow map determined by $T_{\phi_i}$.   We refer to $Z_i$ as the 
\emph{the incoming flow map}; it is a $C^\infty$ diffeomorphism onto its range, since $T_{\phi_i}$ is transverse to surfaces $t=c$
for $|c|$ small.   Moreover the range of $Z_i$ contains an $\mathbb{R}^{n+1}$-neighborhood of $0$.

Now define the flowout of $\cU$ under $T_{\phi_i}$ in $\Omega_T$ to be
\begin{align}\label{e2a}
J_i=\{Z_i(s,x',y') \ | \ 0\leq s\leq s(x',y'), (x',y',-T)\in \cU\}:=Z_i(\cD^i)\subset \Omega_T,
\end{align}
where $s(x',y')$ is the value of $s$ for which the $x$-component of $Z_i(s,x',y')$ is $0$ when the integral curve leaves $\{x\geq 0\}$, and is the value of $s$ for which the $t$-component of $Z_i(s,x',y')$ is $T$ when the integral curve remains inside $\{x\geq 0\}$.

Let $V:=J_i\cap \{x=0\}$.  For points $(0,y',t')\in V$ and for $s\geq 0$ let 
\begin{align}\label{e3}
(x,y,t)=Z_r(s,y',t'), \text{ where }Z_r(0,y',t')=(0,y',t')
\end{align}
denote the forward flow map determined by $T_{\phi_r}$.\footnote{The reflected phase $\phi_r$ and reflected flow map $Z_r$ are defined precisely in \S \ref{reflectedphase}.}   Parallel to $J_i$  we define the flowout of $V$
\begin{align}\label{e4}
\begin{split}
&J_r=\{(x,y,t)=Z_r(s,y',t') \ | \  0\leq s\leq s(y',t'), (0,y',t')\in V \}:=Z_r(\cD^r)\subset \Omega_T,
\end{split}
\end{align}
where $s(y',t')$ is the value of $s$ for which the $t$-component of $Z_r(s,x',y')$ is $T$;  see Figure \ref{shadow}.

The mapping properties of $Z_r$ are much more difficult to assess than those of $Z_i$, because the set $V=J_i\cap\{x=0\}$ contains points of the grazing set $G_{\phi_i}$ and $T_{\phi_r}$ is tangent to the initial surface $\{x=0\}$ for $Z_r$ on $G_{\phi_i}$.   
It was noticed in \cite{cheverry1996} in the case of first-order grazing that the inverse of $Z_r$ becomes singular nearly the grazing set; the Jacobian determinant of $Z_r^{-1}$ blows up roughly like $1/(\text{distance to } G_{\phi_i})$.  In cases of higher-order grazing we observe that the singularity of this determinant worsens and becomes more complicated as the order of grazing increases.\footnote{See the Remark at the end of section \ref{2d}, along with \eqref{ja} and the subsequent analysis of $\det(A)$.} 
This singularity  of $Z_r^{-1}$ has to be taken into account in our study of diffraction, 
since the formula that constructs the reflected phase $\phi_r$ by the method of characteristics involves $Z_r^{-1}$; see \eqref{q13}--\eqref{q15}.   This leads to 

\begin{ass}[Reflected flow map $Z_r$]\label{A03}
Let $V_r:=\{(y',t') \ | \ (0,y',t')\in V\}$ and $\mathring{V}_r=\{(y',t') \ | \ (0,y',t')\in V\setminus G_{\phi_i}\}$ for $V$ as above.    The sets $\cU$ and $V$  as well as $s_0>0$ can be chosen so that the map
\begin{align*}
Z_r:[0,s_0)\times V_r\to \Omega_T
\end{align*}
is a homeomorphism onto its range $J_r$,  and so that 
\begin{align*}
Z_r:[0,s_0)\times \mathring{V}_r\to \Omega_T
\end{align*}
is a  $C^\infty$ diffeomorphism onto its range.
\end{ass}

\Remark 
In Proposition \ref{i3} of Appendix \ref{flowmap} we show that Assumption \ref{A03} is always satisfied, even for nonlinear incoming phases $\phi_i$,  when the origin is a point of first-order tangency.\footnote{A proposition close to Proposition \ref{i3} was formulated in \cite{cheverry1996}, but the proof there applied to a modified map obtained by truncating the Taylor series of $Z_r$ at order two.}   As with Assumption \ref{A02}, when the origin is a point of higher than first-order tangency, verifying this assumption can be difficult.  In \S\S \ref{2d}--\ref{nd} we show that Assumption \ref{A03} always holds when $P$ is the wave operator acting in the exterior of a strictly convex obstacle (Definition \ref{r1}) and the incoming phase $\phi_i$ is linear.  The proof there applies to all orders of tangency and, in fact, does not depend on Assumption \ref{A02}.

Here is our main result stated in standard form coordinates.   See Theorem \ref{mt2} of \S \ref{DandA} for a more precise and coordinate-free statement.
\begin{theo}\label{mt}
Consider the problem \eqref{e0a} under Assumptions \ref{A02} and \ref{A03}, where $\phi_i$ is a given incoming phase and the origin $0$ belongs to the grazing set $G_{\phi_i}$.  
Suppose that $W_1=\partial_{\theta_i} U_1$ satisfies the support condition \eqref{ezg}.
Then if $T>0$ is small enough, the solution $u^\eps\in H^1(\Omega_T)$ to \eqref{e0a} satisfies 
 \begin{align}\label{e4z}
u^\eps(x,y,t)|_{\Omega_T}\sim_{H^1} u(x,y,t)+\eps U_r(x,y,t,\phi_r/\eps)+\eps U_i(x,y,t,\phi_i/\eps).
\end{align}
Here $U_k(x,y,t,\theta_k)$ for $k=r,i$  is the unique mean zero periodic primitive in $\theta_k$ of $W_k(x,y,t,\theta_k)$,
and the functions 
$$u\in H^1(\Omega_T),\; W_r\in L^2(\Omega_T\times \mathbb{T}),\;  W_i\in  L^2(\Omega_T\times \mathbb{T})$$
  are constructed to satisfy the profile equations \eqref{e6}--\eqref{e8}.  In particular, $W_k$ has support in $J_k$ for $k=r,i$.   The meaning of $\sim_{H^1}$ in \eqref{e0ac} and \eqref{e4z} is given in Definition \ref{meaning}.
\end{theo}
  
The reader may have noticed that an expression like $W_i(x,y,t,\phi_i/\eps)$ has no direct meaning since $W_i$ is only in 
$L^2(\Omega_T\times \mathbb{T})$.    As in \cite{cheverry1996} we therefore make the following definition.

\begin{defn}\label{meaning}
The condition $$u^\eps(x,y,t)|_{\Omega_T}\sim_{H^1} u(x,y,t)+\eps U_r(x,y,t,\phi_r/\eps)+\eps U_i(x,y,t,\phi_i/\eps)$$  means that for any sequence of positive reals $\delta_l\to 0$ as $l\to \infty$,  there exist sequences $W^l_k(x,y,t,\theta_k)$, $k=r,i$ of trigonometric polynomials of mean zero in $\theta_k$ with coefficients in $C^\infty_c(\mathring{\Omega}_T)$ and sequences of  positive reals $\eps_l$ such that 
\begin{subequations}\label{e4y}
    \begin{align}
        \|W_k-W^l_k\|_{L^2(\Omega_T\times\mathbb{T})}\leq \delta_l; \label{e4ya}
    \end{align}
    and for all $\eps\in (0,\eps_l]$,
    \begin{align}
        \left\|u^\eps(x,y,t)-\left(u(x,y,t)+\eps U^l_r(x,y,t,\phi_r/\eps)+\eps U^l_i(x,y,t,\phi_i/\eps)\right)\right\|_{H^1(\Omega_T)}\lesssim\delta_l. \label{e4yb}
    \end{align}
\end{subequations}
Here $U^l_k(x,y,t,\theta_k)$ is the unique mean zero primitive in $\theta_k$ of $W^l_k$.    Up to a change in $\eps_l$ the condition \eqref{e4yb}  is equivalent to the pair of conditions
\begin{align*}
\begin{gathered}
\text{for all } \eps\in (0,\eps_l],\ \left\|u^\eps-u\right\|_{L^2(\Omega_T)}\lesssim \delta_l\text{ and }\\
\left\|\nabla u^\eps-\left(\nabla u(x,y,t)+ W^l_r(x,y,t,\phi_r/\eps)\nabla\phi_r+W^l_i(x,y,t,\phi_i/\eps)\nabla\phi_i\right)\right\|_{L^2(\Omega_T)}\lesssim\delta_l.
\end{gathered}
\end{align*}
\end{defn}
\noindent
In fact, the trigonometric polynomials $W^l_k$ will be constructed to have coefficients in $C_c^\infty(\mathring{J_k})$.

\Remark 
Definition \ref{meaning} also gives the meaning of the symbol $\sim_{H^1}$ in \eqref{e0ac}, except that $\Omega_T$ should replaced by $\Omega_{[-T,-T+\delta]}$ and the terms $U_r$, $U^l_r$ are absent.

Since the profiles $W_r$, $W_i$ have support in $J_r\cup J_i$, Theorem \ref{mt} implies
\begin{cor}\label{inshadow}
The solution $u^\eps$ to problem \eqref{e0a} satisfies
\begin{align*}
\|u^\eps - u\|_{H^1(\Omega_{T}\setminus (J_r\cup J_i))}=o_\eps(1),
\end{align*}
for $u(x,y,t)$ as in \eqref{e4z}.
Although $u$ generally has some of its support in the set $\Omega_{T}\setminus (J_r\cup J_i)$, there are no high frequency oscillations in that set that are  detectable in the $H^1$ norm.    In particular the shadow region adjacent to $\mathrm{SB}_+$ contains no such oscillations.
\end{cor}

\Remark 
The Lipschitzean assumption on $f(x,y,t,\cdot,\cdot)$ includes, of course, the linear case.   We believe that the results of this paper that pertain to higher than first-order grazing are new even for the linear case.   The main new difficulties addressed in this paper are not associated with nonlinearity.

\noindent
{\bf Organization of the paper.}
In \S \ref{intro}, we state assumptions and the main result Theorem \ref{mt} in standard coordinates. In \S \ref{DandA}, we state the assumptions and the main result Theorem \ref{mt2} in a coordinate-free way.
\S\S \ref{sfchoice}-\ref{ea} carry out the proof of the main theorem.   \S \ref{convexobstacle}, which is rather geometric and can be read independently of \S\S \ref{sfchoice}-\ref{ea},  provides examples in all dimensions and involving grazing rays of any order where the main theorem applies.

We close this introduction with some comments on the relation between this paper and \cite{cheverry1996}.   

Recall that the inverse of the reflected flow map, $Z_r^{-1}$ has a singularity at the grazing set that worsens with the order of grazing.   This singularity produces a singularity in $\phi_r$, which is $C^1$ but not $C^2$ near the grazing set.  
The solution of the profile equations for $(u,W_r,W_i)$ in \cite{cheverry1996} for the case of first-order grazing made use of an explicit calculation of this singularity in the second derivatives of $\phi_r$.\footnote{See \cite[(6.1.9)]{cheverry1996} and the top of \cite[p.451]{cheverry1996}, for example.}  Second derivatives of $\phi_r$ occur in the term $(P_1\phi_r)W_r$ of the linearized profile equation \eqref{e7y}, and $P_1\phi_r$  is used in \cite{cheverry1996} to construct an integrating factor when the profile equation is solved by integrating along characteristics.\footnote{Here $P_1=P-B_0$, where $B_0$ is the zeroth order part of $P$; see \eqref{b6zz}}
Our solution of the profile equations 
does not depend on an explicit knowledge of the singularity in $P_1\phi_r$, and this is one reason we were able to solve the equations for any order of grazing.    Indeed, in the energy estimates \eqref{d4}--\eqref{d5} we were surprised to observe a \emph{cancellation} of the term involving $P_1\phi_r$, which blows up near the grazing set.\footnote{It is actually just the bad second-order part $((p(x,y,t,\partial)\phi_r)W_r,W_r)_{L^2}$ of $((P_1\phi_r)W_r,W_r)_{L^2}$ that cancels out.}    In \S \ref{rigorous} we use these estimates in an approximation argument involving approximants $(W^k_r,W^k_i)$ that \emph{vanish} near the grazing set to construct $(W_r,W_i)$.   The cancellation  of the term involving $P_1\phi_r$ allows us to pass to the  limit as $k\to \infty$ to obtain an $L^2$ estimate for $(W_r,W_i)$; see Remark in \S \ref{rigorous}.

The error analysis of \S \ref{ea} uses an essential idea of \cite{cheverry1996}; namely, to estimate the difference between the exact solution $u^\eps$ and an approximate solution obtained by truncating and regularizing $u+\eps U_r+\eps U_i$ in \eqref{e4z} in a careful way.    But there are substantial differences from \cite{cheverry1996} in the way we carry out this idea.   For example, except for Lemmas \ref{f9} and \ref{f13}, we use the profile equations in a quite different way; see \eqref{f16}, \eqref{f17}, and the proofs of Propositions \ref{f19} and \ref{f21}.  Moreover, we found it necessary, even in the case of first-order grazing, to incorporate an extra ``corrector" term of order $\eps^2$ and depending on both $\phi_r$ and $\phi_i$ into the definition of the truncated and regularized approximate solution $m^l_{\mu,\rho,M,\eps}$ in \eqref{f1}.  The corrector is the
 term $\eps^2 U^M_{{\mathrm{nc}}}(x,y,t,\frac{\phi_r}{\eps},\frac{\phi_i}{\eps})$ in \eqref{f1}, and it is needed to ``solve away" a term $f^*_{{\mathrm{nc}}}$ of order $O(1)$ in the expansion of $f(x,y,t,m^l_{\mu,\rho,M,\eps},\nabla m^l_{\mu,\rho,M,\eps})$; see \eqref{f1aa}. The terms $U^M_{{\mathrm{nc}}}$ and $f^*_{{\mathrm{nc}}}$ carry noncharacteristic oscillations that do not propagate.

\section{Definitions,  assumptions, and the main result}\label{DandA}

    In this section we give precise, coordinate-independent statements of our main definitions and assumptions as well as the main theorem, Theorem \ref{mt2}.  

\begin{ass}\label{A0}

For $m\in\mathbb{R}^{n+1}$, let $P(m,\partial_m)$ be a scalar second-order differential operator with real $C^\infty$ coefficients and principal symbol $p(m,\nu)$ a smooth function on $T^*\mathbb{R}^{n+1}$.
We are given a $C^\infty$ hypersurface 
$S=\{m \ | \ \alpha(m)=0\}$ that is spacelike at $m=0$, and a $C^\infty$ hypersurface $\partial{M}=\{m \ | \ \beta=0\}$  that is timelike at $0$.\footnote{Here $m$ denotes a general point and ``$0$'' denotes some distinguished point in the manifold $\mathbb{R}^{n+1}$.  Coordinates have not yet been chosen.}  Replacing $P$ by $-P$ if necessary, we may suppose $p(0,d\alpha(0))<0$, which implies $p(0,d\beta(0))>0$.   
\end{ass}

The surfaces $S$ and $\partial{M}$ are thus both noncharacteristic and intersect transversally at $m=0$.\footnote{The surface $S=\{\alpha=0\}$ is  \emph{spacelike} at $0$ if ${P}(0,\partial_m)$ is strictly hyperbolic in the direction $d\alpha(0)\neq 0$.  If $p(0,d\alpha(0))<0$ then the hypersurface $\beta=0$ is \emph{timelike} at $0$ when $p(0,d\beta(0))>0$.  The hypersurface $\psi=0$ is \emph{noncharacteristic} at $0$ if $p(0,d\psi(0))\neq 0$.   See \cite[pp.416--417]{hormander3} for more discussion of these definitions.}  Define ${M}=\{m \ |\ \beta(m)\geq 0\}$ and $\partial M=\{m \ | \ \beta(m)= 0\}$ for $m$ near $0\in \partial M\cap S$.

The fundamental motivating example to keep in mind is the case  $M=(\mathbb{R}^n\setminus \mathcal{O})\times \mathbb{R}_t$, 
where $\mathcal{O}$ is an open convex  
obstacle with $C^\infty$ boundary, and where 
$P$ is the wave operator $\Box:=\partial_{x_1}^2+\dots+\partial_{x_n}^2-\partial_t^2$.  

In order to work in spaces of low regularity like $L^2$ and $H^1$ we assume that $f$ is \emph{uniformly Lipschitzean} in its last arguments.

\begin{ass}\label{A0z}
For some $R>0$, let $B(0,R)=\{m\in \mathbb{R}^{n+1} \ | \ |m|\leq R\}$.   We assume that 
$f(m,p,q) : B(0,R)\times \mathbb{R}\times \mathbb{R}^{n+1}\to \mathbb{R}$ is $C^\infty$ and there exists $K$ such that 
\begin{align*}
|f(m,p_1,q_1)-f(m,p_2,q_2)|\leq K|(p_1,q_1)-(p_2,q_2)|, \ \text{ for all } (m,p_i,q_i).
\end{align*}
Suppose also that $f(m,0,0)=0$.
\end{ass}

\subsection{Decomposition of \texorpdfstring{$T^*\partial M\setminus 0$}{TEXT} with respect to \texorpdfstring{$p$}{TEXT}.} \label{decomposition}

We recall from \cite{melrosesjostrand1978cpam} the decomposition 
\begin{align*}
T^*\partial M\setminus 0=E\cup H\cup G
\end{align*}
into \emph{elliptic}, \emph{hyperbolic}, and \emph{glancing} sets.   Let $i^*:\partial T^*M\to T^*\partial M$ be the pullback map induced by the inclusion
$i:\partial M\to M$.     Observe that the kernel of $i^*$ is the conormal bundle to $\partial M$, $N^*(\partial M)\subset T^*M$.

If $\sigma\in T^*\partial M\setminus 0$, we say that $\sigma$ belongs to $E$, $H$, or $G$ if the number of elements in  $(i^*)^{-1}(\sigma)\cap p^{-1}(0)$ is zero, two, or one respectively.   The sets $E$ and $H$ are conic open subsets of $T^*\partial M\setminus 0$, and $G$ is a closed conic hypersurface in  $T^*\partial M\setminus 0$.

\begin{defn}\label{q2}
Let $\sigma=(m,\nu)\in G$ and suppose $(i^*)^{-1}(\sigma)\cap p^{-1}(0)=\{\rho\}$, where $\rho\in T^*_mM$.   
We say $\sigma\in G^l$, the glancing set of order at least $l\geq 2$, if \footnote{Here $H_p$ is the Hamilton vector field of $p$, which is defined using the standard symplectic form on $T^*\mathbb{R}^{n+1}$.  A formula  for $H_p$ in coordinates is given by \eqref{hp}.}
\begin{align*}
p(\rho)=0 \text{ and }H_p^j \beta(\rho)=0\text{ for }0\leq j<l.   
\end{align*}
Thus, $G=G^2\supset G^3\supset \dots\supset G^\infty$.  

We say $\sigma\in G^{l}\setminus G^{l+1}$, the set of glancing points of exact order $l$, if $\sigma\in G^l$ and $H^l_p\beta(\rho)\neq 0$.     We will study the transport of oscillations near points $\sigma\in G^{2k}\setminus G^{2k+1}$, $k\geq 1$,  such that $H^{2k}_p\beta(\rho)>0$.    When $k=1$, such a point $\sigma$ is a  classical \emph{diffractive point} as studied in \cite{melrose1975duke} or \cite{cheverry1996}.   When $k\geq 1$ we refer to $\sigma$ as a diffractive point of order $2k$, and we write 
\begin{align}\label{q3}
\sigma\in G^{2k}_d\setminus G^{2k+1}\Leftrightarrow p(\rho)=0, \ H_p^j \beta(\rho)=0\text{ for }0\leq j<2k,\text{ and }H^{2k}_p\beta(\rho)>0.
\end{align}  
\end{defn}

\Remarks
1. If $\sigma\in G^{2k}_d\setminus G^{2k+1}$, let $\gamma(s)$ denote the  bicharacteristic of $p$ such that $\gamma(0)=\rho$.    Then $\gamma$ is tangent to $\partial T^*M$ at $\rho$ and lies $T^*\mathring{M}$ for small $s\neq 0$.

\noindent
2. \emph{Gliding points} of order $2k$, $\sigma\in G^{2k}_g\setminus G^{2k+1}$, are defined as in \eqref{q3} with the single change $H^{2k}_p\beta(\rho)<0$.   If $\sigma\in G^l\setminus G^{l+1}$ for some odd $l$, we call $\sigma$ an \emph{inflection point} of order $l$.

\begin{defn}[Diffractive points of order $\infty$]\label{q2a}
Let $\sigma\in G^\infty$ and suppose $(i^*)^{-1}(\sigma)\cap p^{-1}(0)=\{\rho\}$.
We say that $\sigma$ is a \emph{diffractive point of order $\infty$} and write $\sigma\in G^\infty_d$ if the 
bicharacteristic $\gamma(s)$ of $p$ such that $\gamma(0)=\rho$ lies $T^*\mathring{M}$ for small $s\neq 0$.
\end{defn}

\begin{defn}[Glancing points of diffractive type]\label{q2b}
We denote by 
\begin{align*}
\mathcal{G}_d:=\cup_{k=1}^\infty \left(G^{2k}_d\setminus G^{2k+1}\right)\cup G^\infty_d
\end{align*}
the set of \emph{glancing points of diffractive type}.
\end{defn}

\subsection{The incoming phase \texorpdfstring{$\phi_i$}{TEXT}. }
 For a function $f$ as in Assumption \ref{A0z} and small $\eps>0$,  we study a semilinear problem of the form
\begin{align}\label{q4}
\begin{cases}
P(m,\partial_m) u^\eps=f(m,u^\eps,\partial_m u^\eps) & \text{near $m=0$ in }M, \\
u^\eps=0 &\text{on }\partial M,\\
u^\eps=v^\eps\sim_{H^1} u^1(m)+U_1(m,\phi_i(m)/\eps) &\text{in }\alpha < -T_0 \text{ for some }T_0>0, 
\end{cases}
\end{align}
where $v^\eps$, $u^1$ and $U_1$ are given,  the initial profile $U_1(m,\theta)$ is periodic with mean zero in $\theta$,  and the meaning of $\sim_{H_1}$ is explained in Definition \ref{meaning}.   Here $\phi_i$ is a $C^\infty$  \emph{incoming phase} such that:

\begin{ass}\label{A00}
The function $\phi_i\in C^\infty(U)$ satisfies the 
 the eikonal equation 
\begin{align}\label{q5}
p(m,d\phi_i(m))=0
\end{align}
on some open $\mathbb{R}^{n+1}$-ball $U$ centered at $m=0$.  Here  $U\subset B(0,R)$ for $B(0,R)$ as in Assumption \ref{A0z}.
\end{ass}

We assume that a choice of $\phi_i$ is given satisfying additional properties described below.    We are interested in describing the behavior of oscillations in solutions to \eqref{q4} in an $M$-neighborhood of $m=0$, when a characteristic of $\phi_i$ emerging from the ``past", $\{m\in M \ | \ \alpha(m)<-T_0\}$,  grazes $\partial M$ at $m=0$ to either finite or infinite order.

Let $\phi_0\in C^\infty(\partial M\cap U)$ be defined by 
\begin{align*}
\phi_0=i^*\phi_i=\phi_i|_{\partial M\cap U}.
\end{align*}

The following assumption means that a characteristic of $\phi_i$ grazes $\partial M$  at $0$ to some finite or possibly infinite order:
\begin{ass}\label{A1}
With $\mathcal{G}_d$ as in Definition \ref{q2b}, we have $\underline{\sigma}:=(0,d\phi_0(0))\in \mathcal{G}_d$.
\end{ass}

Let $\urho$ be the point in $\partial T^*M$ such that $(i^*)^{-1}(\usigma)\cap p^{-1}(0)=\{\urho\}$.
We show in \S \ref{coords} that strict hyperbolicity of $P$ with respect to $\alpha$ and the fact that $\{\beta=0\}$ is timelike imply that we can modify $\alpha$ if necessary so that
\begin{align}\label{q6a}
H_p\alpha(\urho)>0.
\end{align}
Thus, $\alpha$ increases along the bicharacteristic through $\urho$ as the bicharacteristic parameter, say $s$, increases, and $\usigma$ is  \emph{nondegenerate} in the sense of \cite{melrose1975duke}.

\begin{defn}\label{nondeg}
The point $\usigma\in \mathcal{G}_d$ is \emph{nondegenerate} if $p$ restricted to the fiber of $T^*M$ over $\pi\usigma$ is nonstationary at 
$\urho$.  
\end{defn}
\noindent In standard form coordinates this is the condition $\partial_{\lambda,\eta,\tau}p(\urho)\neq 0$.  
This condition implies that the $\pi$-projection to $M$ of the bicharacteristic of $p$ through $\urho$ is nonsingular at $\pi\usigma$.\footnote{Here and below we use $\pi$ denote the natural projection from  $T^*M$,  $T^*\partial M$,  or $T^*\mathbb{R}^{n+1}$ to $M$, $\partial M$, or $\mathbb{R}^{n+1}$ respectively.   We denote the derivative of $\pi$ by $\pi_*$.}

\begin{figure}[t]
    \begin{center}
        \includegraphics[scale=0.8]{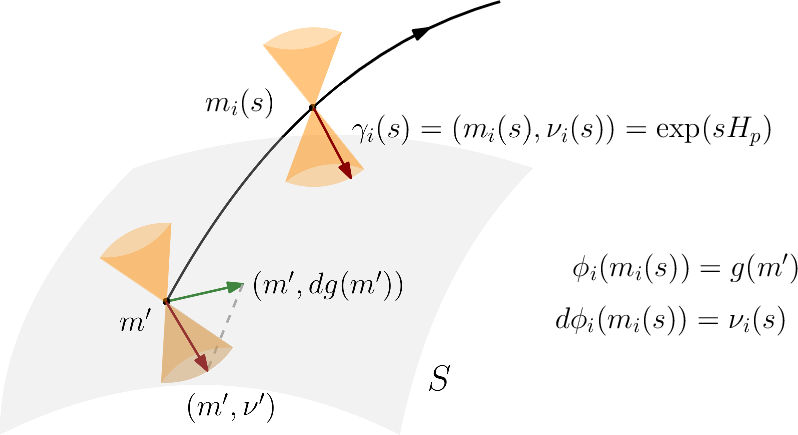}
        \caption{Solving the eikonal equation using the method of characteristics. The yellow cones are the characteristic variety $p^{-1}(0)$, i.e., the light cone. The red arrow on the characteristic cone indicates the choice of $\nu'$ or $\nu_i(s)$ for  which $\alpha$ increases with $s$. The dependence on $(m',\nu')$ is omitted in notations.}
        \label{eikonal}
    \end{center}
\end{figure}

To construct a phase $\phi_i$ as in Assumption \ref{A00}  on an $\mathbb{R}^{n+1}$-neighborhood of $m=0$ by the method of characteristics, one first solves the
bicharacteristic equations for $p$  with a prescribed value for $\phi_i|_S$, say $\phi_i|_S=g\in C^\infty(S)$,  on $S=\{m\in \mathbb{R}^{n+1} \ | \ \alpha(m)=0\}$.\footnote{See  Williams \cite{williams2022} or Evans \cite[Chapter 3]{evans2010ams} for a discussion of this method.} 
Let $i_S:S\to M$ be the inclusion map and $i_S^*:T^*\mathbb{R}^{n+1}|_S\to T^*S$ the natural pullback map.  Denote by
 $\gamma_i(s;(m',\nu'))$ the null bicharacteristic of $p$ 
such that 
\begin{align}\label{q6aa}
\dot\gamma_i(s;(m',\nu'))=H_p\left(\gamma_i(s;(m',\nu'))\right),\;\;     \gamma_i(0;(m',\nu'))=(m',\nu'), 
\end{align}
where $m'\in S$ and $\nu'\in T^*_m\mathbb{R}^{n+1}$ is
chosen so that
\begin{align}\label{q6b}
(m',\nu')\in (i_S^*)^{-1}\left(m',dg(m')\right)\cap p^{-1}(0).
\end{align}
Since $P$ is strictly hyperbolic there are two possible choices of $\nu'$ satisfying \eqref{q6b}, and we make the choice $\nu'=\nu'(m')$ so that 
$\alpha$ increases along $\gamma_i(s;(m',\nu'))$ as $s$ increases.  In particular, if $\unu'=\nu'(0)$ denotes the choice for $m'=0$, we have $\gamma_i(0;(0,\unu'))=\urho.$\footnote{By \eqref{q6b} $g$ must have been chosen so that $\urho\in (i_S^*)^{-1}\left(0,dg(0)\right)\cap p^{-1}(0)$ in order to be compatible with the condition $\usigma=(0,d\phi_0(0))$.}

Let us write $\gamma_i(s;(m',\nu'))=(m_i(s;(m',\nu')),\nu_i(s;(m',\nu'))).$  Then the method of characteristics yields a solution of the eikonal equation such that 
\begin{subequations}\label{q6c}
    \begin{align}
        \phi_i(m_i(s;(m',\nu'))) & =g(m'), \label{q6c1}\\
        d\phi_i(m_i(s;(m',\nu'))) & =\nu_i(s;(m',\nu')). \label{q6c2}
    \end{align}
\end{subequations}

\Remark 
For some open interval $(a,b)\ni 0$ and an open subset $O_S\subset S$, this construction determines an \emph{incoming flow map}
\begin{align}\label{q6cc}
Z_i:(a,b)\times O_S\to \mathbb{R}^{n+1}, \text{ where }Z_i(s,m')=m_i(s;(m',\nu'(m'))).
\end{align}
The transversality condition \eqref{q6a} implies that this map is a $C^\infty$ diffeomorphism.  For $m$ near $0$ in $\mathbb{R}^{n+1}$ let $(s,m')=Z_i^{-1}(m)$.   Then \eqref{q6c} gives 
\begin{align*}
\phi_i(m)=g(m'),
\end{align*}
showing that $\phi_i$ is a  $C^\infty$ function of $m$.

\begin{defn}\label{q6d}
1. The curve $C_i(s)$ in $\mathbb{R}^{n+1}$ given by $C_i(s):= m_i(s;(m',\nu'))$ is called the \emph{forward characteristic curve of} $\phi_i$ passing through $m'$ at $s=0$.\footnote{The word ``forward" indicates just that $\alpha$ increases along the curve as $s$ increases.}

\noindent
2. We call $\gamma_i(s;(m',\nu'))$ a \emph{forward null bicharacteristic associated to $\phi_i$}.
\end{defn}

It follows from \eqref{q6aa} and \eqref{q6c} that forward characteristics of $\phi_i$ satisfy the ODE
\begin{align}\label{q6dd}
\begin{gathered}
\dot m_i(s;(m',\nu'))=\pi_* H_p\left(m_i(s;(m',\nu')),d\phi(m_i(s;(m',\nu')))\right),\\ 
m_i(0,(m',\nu'))=m', \ \nu'=\nu'(m),
\end{gathered}
\end{align}
and the choice of $\nu'(m')$ implies that $\alpha(m_i(s;(m',\nu')))$ increases as $s$ increases.   This curve,  of course, is the $\pi$-projection of the  {forward null bicharacteristic} $\gamma_i(s;(m',\nu'))$.   By \eqref{q6a} $\dot m_i$ is nonvanishing for $|s|$ small.

\Remark
The incoming flow map $Z_i$ as in \eqref{q6cc} is a $C^\infty$ diffeomorphism.  Thus, we can regard the $\pi_* H_p(\cdot)$ term in \eqref{q6dd} as defining a $C^\infty$ vector field on $U$,  the \emph{characteristic vector field of $\phi_i$} denoted  by $T_{\phi_i}$. 
The formula for $T_{\phi_i}$ in standard form coordinates is given in \eqref{ezb}.

The eikonal equation \eqref{q5} implies that the graph of $d\phi_0$, 
\begin{align*}
\mathrm{Graph}(d\phi_0):=\{(m,d\phi_0(m)) \ | \  m\in\partial M\cap U\}\subset T^*\partial M,
\end{align*}
satisfies (see \eqref{m11bb})
\begin{align}\label{q8}
\mathrm{Graph}(d\phi_0)\subset H\cup G.
\end{align}
 The next assumption guarantees the existence of a well-defined \emph{illuminable region} of $\partial M$ which is separated from the \emph{shadow region} of $\partial M$ by a smooth  $(n-1)$-dimensional hypersurface $G_{\phi_i}\subset \partial M$.  It also implies the existence of a smooth  $n$-dimensional hypersurface in $\mathring{M}$, the \emph{shadow boundary} $\mathrm{SB}_+$, which separates the illuminable region of $M$ from the shadow region of $M$; see Definition \ref{SB}.
\begin{ass}\label{A2}

For  an open ball  $U$ as in Assumption \ref{A00} taken smaller if necessary and $\usigma$ as in Assumption \ref{A1},   there exists an open set $V\subset T^*\partial M$ containing  $\usigma$ such that $\pi V=\partial M\cap U$ and 
the set $$G_{\phi_i}:=\pi\left(G\cap \mathrm{Graph}(d\phi_0)\cap V\right)$$ is a $C^1$ codimension-two submanifold of   $U$ definable as 
\begin{align}\label{q8a}
G_{\phi_i}=\{m\in U \ | \ \beta(m)=0,  \zeta(m)=0\},
\end{align}
for some $\zeta\in C^1(U)$ such that  $H_p\zeta(\urho)\neq 0$.\footnote{Below we sometimes shrink $U$ without comment.}   Moreover, every point 
$\sigma\in G\cap \mathrm{Graph}(d\phi_0)\cap V$ lies in $\mathcal{G}_d$.
We refer to $G_{\phi_i}\subset \partial M\cap U$ as the \emph{grazing set determined by $\phi_i$.} 
\end{ass}

\Remarks 
1. Since $H_p\zeta(\urho)\neq 0$ we have $(d\beta\wedge d\zeta)(0)\neq 0$ and thus $d\beta\wedge d\zeta\neq 0$ on $U$ after shrinking $U$ if necessary.

\noindent
2. The glancing set $G$ has dimension $2n-1$ and $\mathrm{Graph}(d\phi_0$) has dimension $n$.  By \eqref{q8} the intersection $G\cap \mathrm{Graph}(d\phi_0)$ is not tranversal.    Nevertheless, Assumption \ref{A2} implies that $G\cap \mathrm{Graph}(d\phi_0)\cap V$ is a $(n-1)$-dimensional $C^1$ submanifold of $T^*\partial M$.    An argument of \cite{cheverry1996} shows that 
if $\usigma\in G^2_d\setminus G^3$, then the conditions in Assumption \ref{A2} automatically hold with $\zeta\in C^\infty$ and 
$$G\cap \mathrm{Graph}(d\phi_0)\cap V \subset G^2_d\setminus G^3.$$

\noindent
3. Assumption \ref{A2} generally takes some effort to verify.  In \S \ref{va2} we verify it in a number of examples involving diffractive points of any finite or infinite order.  In some of these examples $\zeta$ is actually $C^\infty$,  but in others it may be no better than $C^1$.

By Assumption \ref{A2} the grazing set $G_{\phi_i}$  is a $C^1$  hypersurface in $\partial M\cap U$.  
A forward characteristic $C_i(s)$ of $\phi_i$ passing through a point of $G_{\phi_i}$ at $s=0$ remains in $M$ for $|s|$ small.  
For $\zeta$ as in  \eqref{q8a} consider the open subregions of $\partial M\cap U$ given by $I_\pm:=\{\pm\zeta>0\}$. 
We show in step \textbf{2} of the  proof of Proposition \ref{m20}  that Assumption \ref{A2} implies that every point $m$ in one of these subregions, say $I_-$,   has the property that if a forward characteristic $C_i(s)$  satisfies $C_i(0)\in I_-$, then $C_i(s)$ \emph{leaves} $M$ as $s$ increases.  In that case every point in $I_+$ has the opposite property:   if a forward characteristic $C_i(s)$  satisfies $C_i(0)\in I_+$, then $C_i(s)$ \emph{enters} $M$ as $s$ increases.   Replacing $\zeta$ by $-\zeta$ if necessary, we can always suppose $I_-$ is the set where forward characteristics leave $M$.

With this preparation we can state:
\begin{defn}[Illuminable and shadow  regions of $\partial M\cap U$]\label{q10zz}
The illuminable region of $\partial M\cap U$ is $I_-\cup G_{\phi_i}$, where $I_-$ is the set where forward characteristics of $\phi_i$ leave $M$ as $s$ increases.   The shadow region of $\partial M\cap U$ is $I_+$, the set where nongrazing forward characteristics of $\phi_i$ enter $M$ as $s$ increases.
\end{defn}
Observe that the definition of the these regions depends on both the choice of $\phi_i$ and the choice of time function $\alpha$. 
Whether or not a part of the illuminable region is actually illuminated in a given problem  \eqref{q4} depends on the size and position of the $m$-support of $U_1$.

By Assumption \ref{A2} the characteristics of $\phi_i$, that is, integral curves of the vector field $T_{\phi_i}$ as in Remark after Definition \ref{q6d}, 
are transverse to the surface $\zeta=0$.  Thus, since the grazing set $G_{\phi_i}$ is a $(n-1)$-dimensional $C^1$ hypersurface in $\zeta=0$,  the flowout of $G_{\phi_i}$ by the characteristics of $\phi_i$ is a $n$-dimensional $C^1$ submanifold of  $\mathbb{R}^{n+1}.$

\begin{defn} \label{SB} 
1. Denote the flowout of $G_{\phi_i}$ using characteristics of $\phi_i$ by $\mathrm{SB}$.  We have 
\begin{align*}
\begin{split}
&\mathrm{SB}=\mathrm{SB}_+\cup \mathrm{SB}_-, \text{ where }\mathrm{SB}_\pm:=\{\mathrm{exp}(sT_{\phi_i})(m)\in U \ | \ m\in G_{\phi_i}, \;\pm s\geq 0\}.
\end{split}
\end{align*}
2. The $n$-dimensional $C^1$ surface $\mathrm{SB}_+$ is called the \emph{shadow boundary}.  
\end{defn}

\subsection{The reflected phase \texorpdfstring{$\phi_r$}{TEXT}.} \label{reflectedphase}

The reflected phase is also constructed by the method of characteristics, this time with initial data on $I_-\cup G_{\phi_i}\subset \partial M\cap U$.  For any $m_0\in I_-\cup G_{\phi_i}$ there is a forward null bicharacteristic associated to $\phi_i$ that either exits or grazes $\partial T^*M$ at some point $(m_0,\nu_i(m_0))$.  For $m_0\in I_-$ let $(m_0,\nu_r(m_0))$ denote the other point in $(i^*)^{-1}\left(i^*(m_0,\nu_i(m_0))\right)\cap p^{-1}(0)$.  For $m_0\in G_{\phi_i}$ set $(m_0,\nu_r(m_0))=(m_0,\nu_i(m_0))$.

With $\nu_r=\nu_r(m_0)$ denote by
 $\gamma_r(s;(m_0,\nu_r))$ the null bicharacteristic of $p$ 
such that 
\begin{align}\label{q11}
\dot\gamma_r(s;(m_0,\nu_r))=H_p\left(\gamma_r(s;(m_0,\nu_r))\right),\;\;     \gamma_r(0;(m_0,\nu_r))=(m_0,\nu_r).
\end{align}
Writing $\gamma_r(s;(m_0,\nu_r))=(m_r(s;(m_0,\nu_r)),\nu_r(s;(m_0,\nu_r)))$, we can now define the reflected flow map.

\begin{defn}\label{q12}
For some $s_0>0$  the \emph{reflected flow map} is the map
\begin{align}\label{q12z}
Z_r:[0,s_0)\times (I_-\cup G_{\phi_i})\to M, \text{ where } Z_r(s,m_0)=m_r(s;(m_0,\nu_r)).
\end{align}
\end{defn}
\noindent The bicharacteristic equations \eqref{q11}  have a solution that is $C^\infty$ in $(s,m_0)$, so the map $Z_r$ is $C^\infty$.

To construct the reflected phase by the method of characteristics we need to invert the map $Z_r$ in \eqref{q12z} on its range, but it is not clear that an inverse exists.  Indeed, when $m_0\in G_{\phi_i}$, the vector field $H_p$ is not transverse to $\partial T^*M$ at $(m_0,\nu_r(m_0))$, and this is manifested in the fact that as $s\to 0$ and $m_0\to G_{\phi_i}$, the Jacobian determinant of 
$Z_r$  approaches $0$.    In \cite{cheverry1996} this determinant was shown to vanish to first order, see \eqref{jac},  in the case $\usigma\in G^2_d\setminus G^3$, and one observes higher order vanishing when $\usigma$ is of higher order diffractive type; see \S\S \ref{2d}--\ref{nd}.
Because of this vanishing, it is not clear in general that the map $Z_r$ in \eqref{q12} is injective even on small domains of the form $[0,s_0)\times (I_-\cup G_{\phi_i})$.   This leads to the next assumption.

\begin{ass}\label{A3}
The reflected flow map $Z_r:[0,s_0)\times (I_-\cup G_{\phi_i})\to M$ is an injective map onto its range, which we denote by $\cJ_r$.   Moreover, the restriction 
$Z_r:[0,s_0)\times I_-\to M$ is a local $C^\infty$ diffeomorphism onto its range, which we denote by $\mathring{\cJ_r}$.\footnote{Note that $\cJ_r$ is not the same as the set $J_r$ defined in the Introduction, which depends on $\cU\supset \mathrm{supp}_{x,y,t}W_1$.}
\end{ass}

\Remarks
1. Assumption \ref{A3} implies that $Z_r:[0,s_0)\times I_-\to M$ is a $C^\infty$ diffeomorphism onto $\mathring{\cJ_r}$, and that 
 $Z_r:[0,s_0)\times (I_-\cup G_{\phi_i})$ is a homeomorphism onto $\cJ_r$.\footnote{For the simple argument showing this,  see step \textbf{5} in the proof of Proposition \ref{i3}.}

\noindent
2. The vector field $H_p$ is transverse to $\partial T^*M$ at points $(m_0,\nu_r(m_0))$ when $m_0\in I_-$, but this implies only that $Z_r$ is a local diffeomorphism on some neighborhood of $(0,m_0)$ whose size may shrink as $m_0\to G_{\phi_i}$.  

\noindent
3. The shadow boundary $\mathrm{SB}_+$ (Definition \ref{SB}) can also be characteristized as the flowout under $Z_r$ of the grazing set 
$G_{\phi_i}$.  This is because $(m_0,\nu_r(m_0))=(m_0,\nu_i(m_0))$ in \eqref{q11} when $m_0\in G_{\phi_i}$.

\noindent
4. Like Assumption \ref{A2}, Assumption \ref{A3} usually takes some effort to verify.  In \S\S \ref{2d}--\ref{nd} we verify it in a number of examples involving points of higher order diffractive type.   In Proposition \ref{i3} we prove that Assumption \ref{A3} always holds when $\usigma\in G^2_d\setminus G^3$
and $\phi_i$ is \emph{any} characteristic phase, possibly nonlinear, such that $\usigma=(0,d\phi_i(0))$.\footnote{In \cite[Lemma 2]{cheverry1996} a partial proof of Proposition \ref{i3} was given.   The Lemma proved injectivity of the map obtained by truncating the Taylor expansion of $Z_r$ at order two.}

The method of characteristics yields a solution of the eikonal equation, the reflected phase $\phi_r$,  such that 
\begin{subequations}\label{q13}
    \begin{align}
        \phi_r(m_r(s;(m_0,\nu_r))) & =\phi_i(m_0),   \ \nu_r=\nu_r(m_0),\\
        d\phi_r(m_r(s;(m_0,\nu_r))) & =\nu_r(s;(m_0,\nu_r)).
    \end{align}
\end{subequations}
As in the construction of $\phi_i$, the construction of $\phi_r$ requires us to invert the associated flow map.   
For $m\in \cJ_r$  Assumption \ref{A3} gives us  
$(s,m_0)=Z_r^{-1}(m)$.   Writing
\begin{align*}
\tilde{\nu}_r(s,m_0):=\nu_r(s;(m_0,\nu_r)),
\end{align*}
by \eqref{q13} we thus obtain
\begin{align}\label{q14}
\phi_r(m)=\phi_i(m_0) \text{ and }d\phi_r(m)=\tilde{\nu}_r\circ Z_r^{-1}(m).
\end{align}
This shows that 
\begin{align}\label{q15}
\phi_r\in C^\infty(\mathring{\cJ_r}), \text{ but we just have }\phi_r\in C^1(\cJ_r). 
\end{align}
A computation given in \cite{cheverry1996} shows that $\phi_r$ generally fails to be in
$C^2(\cJ_r)$ even when $\usigma\in G^2_d\setminus G^3$.   By \eqref{q14} the singularity in $\phi_r$ is due to the singularity of $Z_r^{-1}$ on the set $Z_r(\{0\}_s\times G_{\phi_i})$.

By Remark 1 after Assumption \ref{A3} 
and with $\gamma_r$ as in \eqref{q11}, we can regard \\
$\pi_*H_p\left(\gamma_r(s;(m_0,\nu_r))\right)$ as defining a $C^\infty$ vector field on $\mathring{\cJ_r}$,  denoted $T_{\phi_r}$, which extends to a continuous vector field on $\cJ_r$.

\begin{defn}\label{q16}
1. We call the curve $s\to Z_r(s,m_0)$ a \emph{characteristic of $\phi_r$} and the curve $s\to \gamma_r(s;(m_0,\nu_r))$ 
a \emph{null bicharacteristic associated to $\phi_r$}.   

\noindent
2.  We call $T_{\phi_r}$, which is defined on $\cJ_r$,  the \emph{characteristic vector field of $\phi_r$}.   
\end{defn}

\subsection{Main theorem}

We proceed to state our main result for the continuation problem 
\begin{subnumcases}{\label{q4z}}
    P(m,\partial_m) u^\eps=f(m,u^\eps,\partial_m u^\eps) & near $m=0$ in $M$, \label{q4za}\\
    u^\eps=0 & on $\partial M$, \label{q4zb}\\
    u^\eps=v^\eps\sim_{H^1} u^1(m)+U_1(m,\phi_i(m)/\eps) & in $\{m\in M\ | -T\leq \alpha(m) \leq -T+\delta\}$ \nonumber\\
    \ & for some $T>0$. \label{q4zc}
\end{subnumcases}
Suppose $U_{\det}\subset M\cap U$ with $0\in  U_{\det}$ is a domain of determinacy for the continuation problem in $M$ determined by $P(m,\partial_m)$ and the Dirichlet boundary condition \eqref{q4zb}.   We  set\footnote{Definition \ref{meaning} gives the meaning of $\sim_{H^1}$ in \eqref{q4z} (resp. \eqref{mtz}), with the obvious change that $\Omega_T$ should now be replaced by $U_{{\det},[-T,-T+\delta]}$ (resp. $U_{{\det},[-T,-T]}$). } 
\begin{align*}
\begin{split}
U_{\det,[T_1,T_2]}=U_{\det}\cap \{m \ | \ T_1\leq \alpha(m)\leq T_2\}, \
U_{\det,T_3}=U_{\det}\cap \{m \ | \ \alpha(m)=T_3\}.
\end{split}
\end{align*}

\begin{theo}\label{mt2}
Consider the problem \eqref{q4z} under the structural Assumptions \ref{A0} on $P(m,\partial_m)$ and \ref{A0z} on $f(m,p,q)$, 
Assumption \ref{A00} on the incoming phase $\phi_i$, Assumption \ref{A1} on $\usigma\in \cG_d$, Assumption \ref{A2} on the grazing set $G_{\phi_i}$, and Assumption \ref{A3} on the reflected flow map $Z_r$.   Suppose that both $u^1$ and $U_1$ have $m$-support strictly away from $\partial M$. 

Let $U_{\det}\subset M\cap U$ with $0\in  U_{\det}$ be a domain of determinacy for the continuation problem in $M$ determined by $P(m,\partial_m)$ and the Dirichlet boundary condition \eqref{q4zb}.   Then for some small enough $T>0$ the following statements hold.  
If $U_1(m,\theta)|_{\{m | \alpha=-T\}}$ has small $m$-support near $\mathrm{SB}_-$ such that 
\begin{align*}
 \mathrm{supp}_m\;U_1(m,\theta)|_{\{m | \alpha=-T\}}\subset \mathring{U}_{\det,-T},
\end{align*}
then\footnote{As noted in the Introduction this assumption on the $m$-support is no real restriction, since our purpose is to focus on what happens near the particular grazing point $0\in \partial M$.} 
\begin{align}\label{mtz}
u^\eps(m)|_{U_{\det,[-T,T]}}\sim_{H^1} u(m)+\eps U_r(m,\phi_r/\eps)+\eps U_i(m,\phi_i/\eps).
\end{align}
Here $U_k(m,\theta_k)$ for $k=r,i$  is the unique mean zero periodic primitive in $\theta_k$ of $W_k(m,\theta_k)$,
and the functions 
$$u\in H^1(U_{\det,[-T,T]}),\; W_r\in L^2(U_{\det,[-T,T]}\times \mathbb{T}),\;  W_i\in  L^2(U_{\det,[-T,T]}\times \mathbb{T})$$
are constructed to satisfy the profile equations \eqref{e6}--\eqref{e8}.  In particular, $W_i$ has $m$-support in  the set $K_i$ which is the forward flowout in $U_{\det,[-T,T]}$ of $\mathrm{supp}_m U_1\cap \{\alpha=-T\}$ under $T_{\phi_i}$, and $W_r$ has $m$-support in the forward flowout in $U_{\det,[-T,T]}$ of $K_i\cap\partial M$ under $T_{\phi_r}$.
\end{theo}

The sets $K_k$, $k=r,i$ may be quite irregular, but they are contained  in sets $J_k$, $k=r,i$ respectively, with piecewise $C^1$ boundaries,  which are as described in the Introduction.

\Remark 
An immediate consequence of Theorem \ref{mt2} and Definition \ref{meaning} is that the shadow region adjacent to $\mathrm{SB}_+$ contains no high frequency oscillations detectable in the $H^1$ norm; recall Corollary \ref{inshadow}.

\section{Standard-form coordinates}\label{sfchoice}

In this section we choose spacetime coordinates that put the principal symbol of $P$ in a form  that will facilitate later computations.

Let $(x,y,t)(z)$ be any  $C^\infty$ coordinates near $z=0\in \mathbb{R}^{n+1}$  for which $(x,y,t)(0)=(0,0,0)$ and such that $x=\beta$ and $t=\alpha$ for $\alpha$, $\beta$ as in Assumption \ref{A0}.    Write $(\lambda,\eta,\tau)$ for the dual coordinates.   Then $p$ takes the form
\begin{align}\label{m3}
p(x,y,t,\lambda,\eta,\tau)=\chi(x,y,t)\left[\lambda^2+b(x,y,t,\eta,\tau)\lambda+c(x,y,t,\eta,\tau)\right],
\end{align}
where 
\begin{align*}
\begin{split}
\chi(0,0,0)>0, \
c(0,0,0,0,\pm 1)<0,
\end{split}
\end{align*}
and $b$, $c$ are real homogeneous polynomials of degrees respectively one and two in $(\eta,\tau)$.   Next we change variables to $(x',y',t')=\psi_1(x,y,t)$ to remove the ``mixed" $b\lambda$ term in \eqref{m3}.   For this one can choose $\psi_1$ so that $\psi_1(0,y,t)=(0,y,t)$ and $x'=x$.  If we write
\begin{align*}
b(x,y,t,\eta,\tau)\lambda=\sum_{j=1}^{n-1}b_j(x,y,t)\eta_j\lambda +b_n(x,y,t)\tau \lambda,
\end{align*}
direct computation shows that we may take $\psi_1$ to be given by 
\begin{align}\label{m5b}
x'=x;\ y_k'=y_k+e_k(x,y,t),\ 1\leq k\leq n-1; \ t'=t+e_n(x,y,t),
\end{align}
where the $C^\infty$ functions $e_k$, $1\leq k\leq n$,  are chosen to satisfy the decoupled transport equations
\begin{align*}
\begin{split}
2\partial_x e_k+\sum_{j=1, j\neq k}^{n-1}b_j(\partial_{y_j}e_k)+b_k(1+\partial_{y_k}e_k)+b_n\partial_te_k  & =0, \ 1\leq k\leq n-1,\\
2\partial_x e_n+\sum_{j=1}^{n-1}b_j\partial_{y_j}e_n+b_n(1+\partial_t e_n) & =0,\\
e_k|_{x=0} & =0, \ 1\leq k\leq n.
\end{split}
\end{align*}

For  a new positive function $\chi$ the principal symbol $p$ now takes the form
\begin{align*}
p(x',y',t',\lambda',\eta',\tau')=\chi(x',y',t')\left[\lambda'^2+q(x',y',t',\eta',\tau')\right] \text{ near }(0,0,0).
\end{align*}
It is not clear that the surfaces  $t'=0$ are spacelike for $P$, so we make another change of variables $(x'',y'',t'')=\psi_2(x',y',t')$ to  insure that one of our coordinates is a time variable.     Let 
\begin{align*}
\psi_2(x',y',t'):=\begin{pmatrix}1&0\\0&\mathcal{A}\end{pmatrix}\begin{pmatrix}x'\\y'\\t'\end{pmatrix},
\end{align*}
where  $\mathcal{A}$ is an orthogonal $n\times n$ matrix  chosen to diagonalize the quadratic form 
\begin{align}\label{m6}
q(0,0,0,\eta,\tau)=\begin{pmatrix}\eta&\tau\end{pmatrix}Q\begin{pmatrix}\eta\\\tau\end{pmatrix}; \text{ that is }\mathcal{A}Q\mathcal{A}^t=\diag(q_1,q_2,\dots,q_n).  
\end{align}
The strict hyperbolicity of $p$ and the fact that $x'=0$ is timelike imply that the symmetric matrix $Q$ has signature $(n-1,1)$. We can choose $\mathcal{A}$ so that $q_n$ is the single negative eigenvalue of $Q$.   In the $(x'',y'',t'',\lambda'',\eta'',\tau'')$ coordinates we therefore have 
\begin{align}\label{m7}
q(0,0,0,\eta'',\tau'')=\sum^{n-1}_{k=1}q_k\eta_k''^2+q_n\tau''^2, 
\end{align}
so the surface $t''=0$ is spacelike for $P$ at $(0,0,0)$.    For new functions $\chi$, $q$ the principal symbol of $P$ now takes the form
\begin{align}\label{m8}
p(x'',y'',t'',\lambda'',\eta'',\tau'')=\chi(x'',y'',t'')\left[\lambda''^2+q(x'',y'',t'',\eta'',\tau'')\right], \quad \chi>0,
\end{align}
and $P$ is strictly hyperbolic with respect to $t''$ on a neighborhood of $(0,0,0)$.  In these coordinates the basepoint 
$\usigma$ in Assumption \ref{A1} has the form $(0,0,\ueta,\utau)$, and $\urho$ as in \eqref{q6a} has the form $(0,0,0,0,\ueta,\utau)$.   Replacing $\mathcal{A}$ by $-\mathcal{A}$ if necessary in \eqref{m6}, we can arrange so that 
\begin{align*}
\utau<0\text{ and thus by \eqref{m7} }H_p t''(\urho)>0. 
\end{align*} 
This establishes \eqref{q6a} and the nondegeneracy of $\usigma\in \mathcal{G}_d$; the coordinate $t''$ is the ``modified $\alpha$" that appears in \eqref{q6a}.

\Remark 
This argument shows that the nondegeneracy of $\usigma\in \mathcal{G}_d$ is an automatic consequence of strict hyperbolicity and the fact that the boundary is timelike.

Henceforth, we drop the double primes in \eqref{m8}.  We are free to replace $f$ by $\chi^{-1}f$ in \eqref{q4}, so we take $\chi=1$ from now on.     This gives the following form of the principal symbol of $P$:
\begin{align}\label{m10a}
p(x,y,t,\lambda,\eta,\tau)=\lambda^2+q(x,y,t,\eta,\tau).
\end{align}

\begin{defn}[Standard form of $p$]\label{sform}
We refer to $p$ as in \eqref{m10a},  where $t$ is a global time coordinate and $q(x,y,t,\cdot,\cdot)$ has signature $(n-1,1)$, as a \emph{standard form of $p$}. 
\end{defn}

Sometimes we also need to work with  systems of coordinates $(x,z,\lambda,\eta)$  with $z$ and  $\eta$  in  $\mathbb{R}^n$  in which $p$  takes the form
\begin{align}\label{m10b}
p(x,z,\lambda,\eta)=\lambda^2+q(x,z,\eta),
\end{align}
where $x=0$ defines $\partial M$ but possibly none of the $z_i$ is a suitable time coordinate.\footnote{In \eqref{m10a} $\eta\in\mathbb{R}^{n-1}$.}  In that case we call \eqref{m10a} an \emph{almost standard form of $p$}.

\subsection{Reduction to a problem on a large domain of determinacy \texorpdfstring{$\Omega_T$}{TEXT}.}\label{reduce}

We can modify the coefficients of $P$ outside the neighborhood  $U\ni(0,0,0)$ as in Assumption \ref{A00} on which $\phi_i$ is defined to obtain an operator $P$ with $C^\infty$ coefficients constant outside a compact set that is strictly hyperbolic with respect to $t$ on $\mathbb{R}^{n+1}$,  with $\chi>0$ on $\mathbb{R}^{n+1}$ and with $x=0$ everywhere timelike for $P$.  Similarly, we can modify $f(x,y,t,p,q)$ for $(x,y,t)$ outside $U$ to obtain a smooth function that is uniformly Lipschitzean in $(p,q)$ for $(x,y,t)\in \mathbb{R}^{n+1}$.
Our analysis will be local near $(0,0,0)$, but this extension of $P$  allows us to work on a domain of the form 
\begin{align*}
\Omega_T:=\{(x,y,t)\in\mathbb{R}^{n+1} \ | \ x\geq 0, -T\leq t\leq T\},   \text{ for some }T>0.
\end{align*}

To choose $T$ we first fix an $\mathbb{R}^{n+1}$-open set  $U'\subset U$ such that $U_{\det}:=U'\cap M$ is a domain of determinacy  for the boundary problem \eqref{q4}.   
We then choose $T>0$ small enough so that all forward broken characteristics starting at points  $m\in \{t=-T\}\cap U_{\det}$ reach 
$\{t=T\}$ before leaving $U_{\det}$. Here a forward broken characteristic is either just a forward characteristic of $\phi_i$ that does not leave $M$, or consists of a forward characteristic of $\phi_i$ up to the point of exiting $M$ together with the associated reflected characteristic of $\phi_r$.   
With such a choice of $T$ the set $\Omega_T$ is not only a domain of determinacy for the extended problem    corresponding to \eqref{q4}:
\begin{align*}
    \begin{cases}
        Pu^\eps=f(x,y,t,u^\eps,\nabla_{x,y,t} u^\eps) & \text{in } \Omega_T,  \\
    u^\eps(0,y,t)=0 & \text{on } \Omega_T\cap\{x=0\}, \\
    u^\eps\sim_{H^1} u^1(x,y,t)+\eps U_1(x,y,t,\phi_i/\eps) & \text{on } \Omega_{[-T,-T+\delta]},
    \end{cases}
\end{align*}
where $\delta>0$  is small;
$\Omega_T$ also has the property that $u^\eps|_{U_{\det}\cap \Omega_T}$ is completely determined by the restriction of $f$, $u^1$, and $U_1$ to $U_{\det}\cap \Omega_T$.    Moreover, the sets $J_i$ and $J_r$ defined in the Introduction satisfy
\begin{align*}
J_i\cup J_r\subset U_{\det}\cap \Omega_T.
\end{align*}
This reduction allows us to use the extended problem to study the original problem of Theorem \ref{mt2} on a neighborhood of $0\in M$.

\subsection{Some properties of \texorpdfstring{$q$}{TEXT} and \texorpdfstring{$\phi_i$}{TEXT} in these coordinates. }\label{coords}

In this section we use coordinates to establish some of the claims made in \S \ref{DandA}. 

In coordinates $(x,y,t,\lambda,\eta,\tau)$ that put $p$ in standard form \eqref{m10a} the map 
$i^*:\partial T^*M\to T^*\partial M$ is 
\begin{align*}
i^*(x,y,t,\lambda,\eta,\tau)=(y,t,\eta,\tau),
\end{align*}
 and 
the elliptic, hyperbolic, and glancing regions of 
$T^*\partial M$ are\footnote{We write points in $\partial M$ sometimes as $(0,y,t)$, sometimes as $(y,t)$.} 
\begin{align*}
\begin{split}
&E=\{(y,t,\eta,\tau) \ | \ q(0,y,t,\eta,\tau)>0\},\\
&H=\{(y,t,\eta,\tau) \ | \ q(0,y,t,\eta,\tau)<0\},\\
&G=\{(y,t,\eta,\tau) \ | \ q(0,y,t,\eta,\tau)=0\text{ and }(\eta,\tau)\neq (0,0)\}.
\end{split}
\end{align*}
The eikonal equation takes the form
\begin{align}\label{m11a}
(\partial_x\phi_i)^2+q(x,y,t,\partial_y\phi_i,\partial_t\phi_i)=0.
\end{align}
Evaluating \eqref{m11a} at $x=0$ we obtain
\begin{align*}
q(0,y,t,\partial_{y,t}\phi_i(0,y,t))=-\partial_x\phi_i(0,y,t)^2\leq 0,
\end{align*}
which implies \eqref{q8}:
\begin{align}\label{m11bb}
\mathrm{Graph}(d\phi_0)=\{(y,t,\partial_{y,t}\phi_i(0,y,t)) \ | \ (0,y,t)\in U\}\subset H\cup G.
\end{align}
for $U$ as in Assumption \ref{A2}.  The grazing set determined by $\phi_i$ is thus the set
\begin{align*}
G_{\phi_i}=\{(0,y,t)\in U \ | \ \partial_x\phi_i(0,y,t)=0\}.
\end{align*}
In particular, $\pi\usigma=(0,0,0)\in G_{\phi_i}.$

When $\usigma\in G^2_d\setminus G^3$, it was shown in \cite{cheverry1996} that one can always take the function $\partial_x\phi_i(0,y,t)$ as a 
coordinate function.   To see this note first that since 
\begin{align}\label{hp}
H_p=p_\lambda\partial_x+p_\eta\partial_y+p_\tau\partial_t-p_x\partial_\lambda-p_y\partial_\eta-p_t\partial_\tau,
\end{align}
the conditions defining $G^{2k}_d\setminus G^{2k+1}$ when $k=1$,
\begin{align*}
p(\urho)=0, H_px(\urho)=0, H_p^2x(\urho)>0,
\end{align*}
imply
\begin{subequations}\label{m11e}
    \begin{align}
        q(0,0,0,\ueta,\utau) & =0, \label{m11ea}\\
        q_x(0,0,0,\ueta,\utau) & <0. \label{m11eb}
    \end{align}
\end{subequations}   
Differentiating the eikonal equation  \eqref{m11a} with respect to $x$ yields
\begin{align}\label{m15}
\begin{split}
&2\partial_x\phi_i\partial_{xx}\phi_i+\partial_x q(x,y,t,\partial_{y,t}\phi_i(x,y,t))+\partial_{\eta,\tau}q\cdot \partial_{y,t}\partial_x\phi_i=0.
\end{split}
\end{align}
Evaluating \eqref{m15}  at $(0,0,0)$ we obtain
\begin{align}\label{m16}
q_x(0,0,0,\ueta,\utau)+q_{\eta,\tau}(0,0,0,\ueta,\utau)\cdot \partial_{y,t}\partial_x\phi_i(0,0,0)=0.
\end{align}
With \eqref{m11eb} equation \eqref{m16} implies both
\begin{subequations}\label{m17}
    \begin{align}
        q_{\eta,\tau}(0,0,0,\ueta,\utau) & \neq 0,\text{ and } \label{m17a}\\
        \partial_{y,t}\partial_x\phi_i(0,0,0) & \neq 0. \label{m17b}
\end{align}
\end{subequations}
The property \eqref{m17a} shows again that $\usigma$ is {nondegenerate}, while 
\eqref{m17b} allows us to choose a new system of coordinates $(x,z,\lambda,\eta)$, $z=(z_1,\dots,z_n)$,  such that 
\begin{align}\label{m18}
\partial_x\phi_i(0,z)=z_1.
\end{align}
In these coordinates $p$ has almost standard form \eqref{m10b},  $\urho=(0,0,0,\ueta)$ for some $\ueta\in \mathbb{R}^n$, and  \eqref{m16} takes the form
\begin{align}\label{m19}
q_x(0,0,0,\ueta)+q_{\eta_1}(0,0,0,\ueta)=0.
\end{align}
This argument shows that if $\usigma\in G^2_d\setminus G^3$, then the conditions of Assumption \ref{A2} always hold with 
$\zeta=\partial_x\phi_i(0,z)$; recall Remark 2 after Assumption \ref{A2}.

In the case $\usigma\in G^{2k}_d\setminus G^{2k+1}$ when $k>1$ we have $q_x(0,0,0,\ueta,\utau)=0$, so the above argument does not apply.   When $k>1$ it turns out  that $\partial_x\phi_i(0,y,t)$ can no longer be taken as a coordinate function; see the Remark after Proposition \ref{ra11} and \eqref{ra14z} in particular.  However, we show in Proposition \ref{m20} that Assumption \ref{A2} implies that the zero set of this function, namely $G_{\phi_i}$,  can  be defined by $z_1=0$ in a $C^1$ system of coordinates $(x,z)$.

\begin{prop}\label{m20}
Let $G_{\phi_i}$ be the grazing set defined in Assumption \ref{A2} and let $I_\pm$ be as in Definition \ref{q10zz}.  
Assumption \ref{A2} implies that one can find $C^1$ coordinates $(x,z)$ in $M\cap U$ such that 
\begin{subequations}\label{m21}
    \begin{align}
        & G_{\phi_i}= \{(0,z)\in \partial M\cap U \ | \ \partial_x\phi_i(0,z)=0\} = \{(0,z)\in \partial M\cap U \ | \ z_1=0\}, \label{m21a}\\
        & I_{\pm}= \{(0,z)\in \partial M\cap U\ | \ \pm z_1>0\}, \label{m21b}\\
        & H_pz_1(\urho)\neq  0. \label{m21c}
    \end{align}
\end{subequations}
\end{prop}

\begin{proof}
\textbf{1. }Let $(x,y,t)$ be the standard form coordinates chosen in \S \ref{sfchoice}.    Then \eqref{q8a} implies
\begin{align*}
G_{\phi_i}=\{(0,y,t)\in \partial M\cap U \ | \ \zeta(0,y,t)=0\}.
\end{align*}
Set $\zeta_0(y,t)=\zeta(0,y,t)$.   By Remark 1 after Assumption \ref{A2} 
we have $dx\wedge d\zeta\neq 0$ on $U$,  and this  implies\footnote{Here we regard $\zeta_0$ as a function on all of $U$.}
\begin{align*}
dx\wedge d\zeta_0\neq 0 \text{ on }U.
\end{align*}
Thus, with $x$ as before we may choose $(x,z)$ coordinates on $U$ where $z_1=\zeta_0(y,t)$.   These coordinates are $C^1$ and and $H_p\zeta(\urho)\neq 0\Rightarrow H_pz_1(\urho)\neq 0$.  We now have \eqref{m21a},\eqref{m21c}.

\textbf{2. }The function $\partial_x\phi_i(0,z)$ has a fixed sign in each of the subregions of $\partial M\cap U$ given by $\{(0,z)\in  \partial M\cap U \ | \ \pm z_1>0\}$. To prove \eqref{m21b} we must show that $\partial_x\phi_i(0,z)$ changes sign from one subregion to the other.  

Choose a point $\sigma'=(z',\partial_z\phi_i(0,z'))\in H$ close to $\usigma$, and let $\gamma_i(s)$ be the null bicharacteristic of $p$ such that $\gamma_i(0)=(0,z',\partial_x\phi_i(0,z'),\partial_z\phi_i(0,z'))$.    Since $\usigma\in \mathcal{G}_d$ the null bicharacteristic of $p$ through $\urho$, call it $\gamma(s)$,  is tangent to $\partial T^*M$ at $\gamma(0)=\urho$, but bends and remains in $T^*\mathring{M}$ for $|s|\neq 0$ small.  
We can suppose that $\gamma_i(s)$ leaves $T^*M$ as $s$ increases, that is, $\partial_x\phi_i(0,z')<0$.  By smooth dependence of solutions of ODEs on initial conditions, $\gamma_i(s)$ remains close to $\gamma(s)$ and so reenters $T^*\mathring{M}$.   The curve $\gamma_i(s)$ cannot reenter $T^*M$ at a point $\gamma_i(s'')=(0,z'',\partial_x\phi_i(0,z''),\partial_z\phi_i(0,z''))$ where $\partial_{x}\phi(0,z'')=0$, for in that case Assumption \ref{A2} implies $(z'', \partial_z\phi_i(0,z''))\in \mathcal{G}_d$, so $\gamma_i(s)$ would lie in $T^*\mathring{M}$ for $|s-s''|\neq 0 $ small.  Thus, we must have 
$\partial_{x}\phi(0,z'')>0$, which shows that $\partial_{x}\phi(0,z'')$ changes sign when $z_1$ changes sign.   Replacing $z_1$ by $-z_1$ if necessary, we arrange \eqref{m21b}.
\end{proof}

\section{Eikonal and profile equations}

In this section we formulate and then solve the profile equations for $(u,W_r,W_i)$.   Eventually, we seek
\begin{align*}
u\in H^1(\Omega_T),  W_r\in L^2(\Omega_T\times \mathbb{T}),  W_i\in L^2(\Omega_T\times \mathbb{T})
\end{align*}
for some small enough $T>0$, 
where $W_r$, $W_i$ have $(x,y,t)$-support in the sets $J_r$, $J_i$, respectively, defined in the Introduction.

\subsection{Formal computation of \texorpdfstring{$P(x,y,t,\partial)u^\eps_a$}{TEXT} and \texorpdfstring{$f(x,y,t,u^\eps_a,\nabla u^\eps_a)$}{TEXT}}\label{formalb}
To motivate the eikonal equations for $(\phi_r,\phi_i)$ and the profile equations for $(u,W_r,W_i)$, we 
first do a \emph{formal} computation of $P(x,y,t,\partial_{x,y,z})u^\eps_a$, where $u^\eps_a$ is an approximate solution of the form 
\begin{align*}
u^\eps_a(x,y,t):=u(x,y,t)+\eps U_r\left(x,y,t,\frac{\phi_r(x,y,t)}{\eps}\right)+\eps U_i\left(x,y,t,\frac{\phi_i(x,y,t)}{\eps}\right).
\end{align*}
Here ``formal" means that we pretend all computations involved make sense on  $\Omega_{T}$, and we leave unspecified the norms in which error terms are small.\footnote{To make sense of all these computations we need to work with truncated and regularized profiles.  Second derivatives of the phase $\phi_r$  blow up near the grazing set $G_{\phi_i}$.   The phases are not defined on all of $\Omega_T$.   The profile  $W_r(x,y,t,\theta_r)$ is only in $L^2$, so evaluation at $\theta_r=\phi_r/\eps$ is not well-defined.}    Rigorous computations similar to these will be shown later to hold for truncated and regularized profiles.  

We use standard form coordinates $(x,y,t)$  in which the second-order operator $P$ has the form
\begin{align*}
    P(x,y,t,\partial)
    =p(x,y,t,\partial)+B_1(x,y,t,\partial)+B_0(x,y,t)
\end{align*}
where $B_j$ is of order $j$, and we set
\begin{align}\label{b6zz}
    P_1(x,y,t,\partial) =p(x,y,t,\partial)+B_1(x,y,t,\partial).
\end{align}

We obtain
\begin{equation}\label{b6b}
    \begin{split}
        & P(x,y,t,\partial)u_a^{\epsilon}(x,y,t) \\
        & =  \epsilon^{-1} \sum_{k=i,r} p(x,y,t,\nabla \phi_k(x,y,t))\partial_{\theta}^2 U_k \left(e,y,t,\frac{\phi_k(x,y,t)}{\epsilon}\right) \\
        & \quad + \epsilon^0 \left[ P(x,y,t,\partial)u + \sum_{k=i,r} (T_{\phi_k}(x,y,t,\partial) W_k(x,y,t,\theta_k) )|_{\theta_k=\frac{\phi_k}{\epsilon}} \right. \\
        & \quad + \left. \sum_{k=i,r} (P_1(x,y,t,\partial)\phi_k) W_k\left( x,y,t,\frac{\phi_k}{\epsilon} \right) \right] + O(\epsilon).
    \end{split}
\end{equation}

Expanding $f(x,y,t,u^\eps_a,\nabla u^\eps_a)$ we obtain:
\begin{align*}
\begin{split}
&f\left(x,y,t,u+\eps U_r+\eps U_i, \nabla\left(u(x,y,t)+\eps U_r\left(x,y,t,\frac{\phi_r}{\eps}\right)+\eps U_i\left(x,y,t,\frac{\phi_i}{\eps}\right)\right)\right)\\
&= f\left(x,y,t,u,\nabla u+ W_r(x,y,t,\theta_r)\nabla \phi_r+W_i(x,y,t,\theta_i)\nabla \phi_i\right)|_{\theta_r=\frac{\phi_r}{\eps},\theta_i=\frac{\phi_i}{\eps}}+O(\eps).
\end{split}
\end{align*}

The goal is to make $Pu^\eps_a-f(x,y,t,u^\eps_a,\nabla u^\eps_a)$ small.  
Clearly, the eikonal equations satisfied by $\phi_i$ and $\phi_r$ make the term of order $\eps^{-1}$ vanish.
The profile equations discussed in the next section are designed to make small the term of order $\eps^0$. 

\subsection{Profile equations}\label{pe}
To write the profile equations we first decompose the nonlinear term\footnote{Here we suppress the dependence of $\underline{f}$, $f^*_r$,  $f^*_i$, and $f^*_{\mathrm{nc}}$ on $(u,W_r,W_i)$ in the notation.}
\begin{align}\label{e5}
\begin{split}
&f(x,y,t,u,\nabla u+W_r\nabla\phi_r+ W_i\nabla\phi_i)\\
&=\underline{f}(x,y,t)+f^*_r(x,y,t,\theta_r)+f^*_i(x,y,t,\theta_i)+f^*_{\mathrm{nc}}(x,y,t,\theta_r,\theta_i),
\end{split}
\end{align}
where $\underline{f}$, $f^*_r$, $f^*_i$ denote respectively the mean of $f(x,y,t,u,\nabla u+W_r\nabla\phi_r+ W_i\nabla\phi_i)$ with respect to $(\theta_r,\theta_i)$, the mean with respect to $\theta_i$ minus $\underline{f}$, and the mean with respect to $\theta_r$ minus $\underline{f}$.  
The term $f^*_{\mathrm{nc}}$ carries the noncharacteristic oscillations.  The coupled profile equations for $u, W_r, W_i$  are:\footnote{Here  we write $\underline{f}=\underline{f}(u,W_r,W_i)$ and do similarly for $f^*_r$, $f^*_i$.}

\begin{align}
    & \begin{cases}\label{e6}
        Pu=\underline{f}(u,W_r,W_i) &\text{in }\Omega_T,\\
        u(0,y,t)=0 & \text{on } \Omega_T\cap\{x=0\}, \\
        u=u^1(x,y,t) & \text{on } \Omega_{[-T,-T+\delta]};
    \end{cases}\\
    & \begin{cases}\label{e7}
        T_{\phi_r}W_r+(P_1\phi_r)W_r=f^*_r(u,W_r,W_i) & \text{in }\mathring{J_r}\times \mathbb{T},\\
        W_r(0,y,t,\theta)=-W_i(0,y,t,\theta) & \text{on } (J_r\cap\{x=0\})\times \mathbb T, \\
        W_r=0 & \text{on }(\Omega_T\setminus J_r)\times \mathbb{T};
    \end{cases}\\
    & \begin{cases}\label{e8}
        T_{\phi_i}W_i+(P_1\phi_i)W_i=f^*_i(u,W_r,W_i) & \text{in }\mathring{J_i}\times \mathbb{T},\\
        W_i|_{t=-T}=W_1(x,y,-T,\theta):=g(x,y,\theta) & \text{on } (J_i\cap\{x=0\})\times \mathbb T,\\
        W_i=0 & \text{on }(\Omega_T\setminus J_i)\times \mathbb{T}.
    \end{cases}
\end{align}
The estimates of \S \ref{ee} and Picard iteration can be used to  construct profiles $u(x,y,t)\in H^1(\Omega_T)$,  and $W_r(x,y,t,\theta_r), W_i(x,y,t,\theta_i)\in L^2(\Omega_T\times\mathbb{T})$ satisfying \eqref{e6}--\eqref{e8}.    The iteration scheme is 
\begin{align}
    & \begin{cases}\label{e6z}
        Pu^{n+1}=\underline{f}(u^n,W^n_r,W^n_i) & \text{in }\Omega_T,\\
        u^{n+1}(0,y,t)=0 & \text{on } \Omega_T\cap\{x=0\},\\
        u^{n+1}=u^1(x,y,t) & \text{on } \Omega_{[-T,-T+\delta]};
    \end{cases} \\
    & \begin{cases}\label{e7z}
        T_{\phi_r}W^{n+1}_r+(P_1\phi_r)W^{n+1}_r=f^*_r(u^n,W^n_r,W^n_i) & \text{in }\mathring{J_r}\times \mathbb{T},\\
        W^{n+1}_r(0,y,t,\theta)=-W^{n+1}_i(0,y,t,\theta) & \text{on } (J_r\cap\{x=0\})\times \mathbb T,\\
        W^{n+1}_r=0 & \text{on }(\Omega_T\setminus J_r)\times \mathbb{T};
    \end{cases} \\
    & \begin{cases}\label{e8z}
        T_{\phi_i}W^{n+1}_i+(P_1\phi_i)W^{n+1}_i=f^*_i(u^n,W^n_r,W^n_i) & \text{in }\mathring{J_i}\times \mathbb{T},\\
        W^{n+1}_i|_{t=-T}=g(x,y,\theta) & \text{on } (J_i\cap\{x=0\})\times \mathbb T,\\
        W^{n+1}_i=0 & \text{on }(\Omega_T\setminus J_i)\times \mathbb{T}.
    \end{cases}
\end{align}
We initiate the iteration by taking $u^0$ and $W^0_r$ equal to zero on $\Omega_T$ and by taking $W^0_i\in L^2(\Omega_T\times \mathbb{T})$ equal to a function supported in $J_i$ that is an extension of $W_1$.    We then construct 
iterates in the order: $u^1,W^1_i,W^1_r,u^2,W^2_i,W^2_r,\dots$, taking care not to confuse the first iterate with the initial datum $u^1$ in \eqref{e6z}.  For each $n$ the functions $W_r^n$, $f^*_r(u^n,W^n_r,W^n_i)$ are supported in $J_r$, while the functions 
$W_i^n$, $f^*_i(u^n,W^n_r,W^n_i)$ are supported in $J_i$.

\Remarks 
1. The equation $T_{\phi_r}W_r+(P_1\phi_r)W_r=f^*_r$, for example, holds in the sense of distributions on $\mathring{J_r}$.
The individual terms on the left side of this equation are not expected to lie in $L^2(\Omega_T\times \mathbb{T})$.   We do \emph{not} claim that this equation holds on $\Omega_T$, even though $W_r$ is defined on $\Omega_T$.  Observe that $T_{\phi_r}$ and $P_1\phi_r$ are  only defined where $\phi_r$ is defined, namely on $J_r$.    In the error analysis we will see that a truncated and regularized version of $W_r$ does satisfy a nearby problem on all of $\Omega_T$. 

\noindent
2. The initial condition for $W_i$ taken at $t=-T$ in \eqref{e8} is consistent with the initial condition taken on $\Omega_{[-T,-T+\delta]}$ in the problem \eqref{e0a}.  That is, the function $U_i$ on $\Omega_{[-T,-T+\delta]}$ obtained from $W_i$ by solving \eqref{e6}--\eqref{e8} and then restricting $W_i$ to $\Omega_{[-T,-T+\delta]}$ can be taken as $U_1$ in \eqref{e0a}.

\noindent
3. In this problem waves associated to incoming and reflected phases $\phi_i$, $\phi_r$ interact in the region $J_r\cap J_i$.   We show  that away from $\mathrm{SB}_+$ the gradients $\nabla\phi_r$ and  $\nabla\phi_i$ are linearly independent at each $(x,y,t)$ and that these phases are \emph{nonresonant}:  for $(x,y,t)\in (J_r\cap J_i)\setminus \mathrm{SB}_+$, we have
\begin{align*}
p(x,y,t,\nabla(k_r\phi_r+k_i\phi_i)(x,y,t))\neq 0 \text{ for }(k_r,k_i)\in \mathbb{Z}^2\text{ such that }k_r\neq 0, k_i\neq 0.
\end{align*}
Thus, no new characteristic phases are produced by nonlinear interactions; see Proposition \ref{noresonance}.  The profile equations reflect this fact.

\section{Solution of the profile equations}\label{ee}

In this section we solve the profile equations in two steps.  First we prove energy estimates for the linear problem that must be solved to construct the $n$th iterate of the scheme \eqref{e6z}--\eqref{e8z}.  Having constructed the iterates, we then use the same energy estimates to show that the iterates converge to a solution of \eqref{e6}--\eqref{e8}.   
  
The linear problem that must be solved to construct the $n$-th iterate $(u^ n,W^n_r,W^n_i)$ consists of the three coupled subproblems\footnote{Really only \eqref{e7y} and \eqref{e8y} are coupled.}

\begin{align}
    & \begin{cases}\label{e6y}
        Pu=\underline{f} & \text{in }\Omega_T,\\
        u(0,y,t)=0 & \text{on } \Omega_T\cap\{x=0\},\\
        u =u^1(x,y,t) & \text{on } \Omega_{[-T,-T+\delta]};
    \end{cases} \\
    & \begin{cases}\label{e7y}
        T_{\phi_r}W_r+(P_1\phi_r)W_r=F_r & \text{in }\mathring{J_r}\times \mathbb{T},\\
        W_r(0,y,t,\theta)=-W_i(0,y,t,\theta) & \text{on } (J_r\cap\{x=0\})\times \mathbb T,\\
        W_r=0 & \text{on }(\Omega_T\setminus J_r)\times \mathbb{T};
    \end{cases} \\
    & \begin{cases}\label{e8y}
        T_{\phi_i}W_i+(P_1\phi_i)W_i=F_i & \text{in }\mathring{J_i}\times \mathbb{T},\\
        W_i|_{t=-T}=g(x,y,\theta_i) & \text{on } (J_i\cap\{x=0\})\times \mathbb T, \\
        W_i=0 & \text{on }(\Omega_T\setminus J_i)\times \mathbb{T}.
    \end{cases}
\end{align}
Here we suppose that 
\begin{align*}
\begin{gathered}
\underline{f}\in L^2(\Omega_T), \ u^1\in H^1(\Omega_{[-T,-T+\delta]}),\
F_r, F_i \in L^2(\Omega_T\times \mathbb{T}),\ g\in L^2(\{t=-T\}),\\
F_r \text{ has support in }J_r; \  F_i, g \text{ have support in }J_i,\text{ resp. }J_i\cap \{t=-T\}.
\end{gathered}
\end{align*}

\subsection{Linear energy estimates: formal arguments}\label{formal}

For $t_0\in [-T,T]$ we expect $W_r$ on $J_r\cap\{t=t_0\}$ to be determined by the data $F_r$ and $W_i(0,y,t)$ of problem \eqref{e7y} in $J_{r,t_0}:= J_r\cap \{t\leq t_0\}$.\footnote{The arguments below will make it clear that the trace on $t=t_0$ as well as traces on $x=0$ make sense.}   The boundary of $J_{r,t_0}$ consists of two flat pieces, one in $\{t=t_0\}$ and one in $\{x=0\}$, and a curved piece foliated by integal curves of $T_{\phi_i}$.

We will do an energy estimate for $W_r$ on $J_{r,t_0}$ starting from the transport equation:
\begin{align*}
     \begin{cases}
         T_{\phi_r} W_r+(P_1\phi_r)W_r=F_r &\text{on }J_{r,t_0}\times \mathbb{T},\\
        W_r=-W_i & \text{on }x=0,
     \end{cases}
\end{align*}
which at least formally implies
\begin{align}\label{d3a}
(T_{\phi_r} W_r,W_r)+((P_1\phi_r)W_r,W_r)=(F_r,W_r).
\end{align}
Here $(\cdot,\cdot)$ is the real $L^2$ pairing on $J_{r,t_0}\times\mathbb{T}$, and below we let $\langle\cdot,\cdot\rangle_{t_0}$ be the $L^2$ pairing on $t=t_0$ and let $(\cdot,\cdot)_0$ be the $L^2$ pairing on $x=0$. 

\Remark 
If $W_r\in L^2(\Omega_T\times \mathbb{T})$ neither term on the left of \eqref{d3a} may have a well-defined finite value. Our plan is first to carry out the energy estimates \emph{formally}.   We then explain how to use the estimates rigorously to obtain solutions to \eqref{e6y}--\eqref{e8y} via an approximation argument;  the estimates will clearly apply to the smooth functions that appear in that argument.  Finally, we will use the estimates again to show that the Picard iterates converge to a solution of \eqref{e6}--\eqref{e8}.

It will be convenient in this section to rewrite $(x,y,t,\lambda,\eta,\tau)$, where $y=(y_1,\dots,y_{n-1})$ and $\eta=(\eta_1,\dots,\eta_{n-1})$ as $(x,y,\lambda,\eta)$, where now $y$ and $\eta$ have $n$ components with $y_n=t$, $\eta_n=\tau$.  The principal symbol $p$ and the operator $T_{\phi_r}$ (recall \eqref{ezb}) may now be written
\begin{align*}
\begin{gathered}
p(x,y,\lambda,\eta) =\lambda^2+q(x,y,\eta)=\lambda^2+\sum_{j,k=1}^n q^{jk}(x,y)\eta_j\eta_k,  \text{ where }q^{jk}=q^{kj},\\
T_{\phi_r} =2\phi_{r,x}\partial_x+2\sum_{j,k=1}^nq^{jk}\phi_{r,y_k}\partial_{y_j}.
\end{gathered}
\end{align*}

First we compute {$(T_{\phi_r}W_r,W_r)$}. 
We have by the Gauss--Green theorem\footnote{Here we use Gauss--Green in the form: $\int_D u_{x_i}v dx=-\int_D uv_{x_i}dx+\int_{\partial D}uv\nu_i dS$, where $\nu$ is the outward unit normal to $\partial D$.}
\begin{align}\label{d4}
\begin{split}
{\tfrac12(T_{\phi_r}W_r,W_r)}
= & 
{-\tfrac12(W_r, T_{\phi_r} W_r)}
-((p(x,y,\partial)\phi_r)W_r,W_r)+(O(1)W_r,W_r)\\
&+\left\langle\left(\sum_{k=1}^nq^{nk}\phi_{r,y_k}\right)W_r,W_r\right\rangle_{t_0}-(\phi_{r,x} W_r,W_r)_0,\\
\end{split}
\end{align}
where $O(1)$ is the bounded function $-\sum_{j,k=1}^{n}\phi_{r,y_k}\partial_{y_j}q^{jk}$.  The boundary integral on the curved part of $J_{r,t_0}$ vanishes since {$T_{\phi_r}$} is tangent to the boundary on that part.
Hence
\begin{align*}
\begin{split}
{(T_{\phi_r} W_r,W_r)}
= &-((p(x,y,\partial)\phi_r)W_r,W_r)+(O(1)W_r,W_r) \\
&+\left\langle\left(\sum_{k=1}^nq^{nk}\phi_{r,y_k}\right)W_r,W_r\right\rangle_{t_0}-(\phi_{r,x} W_r,W_r)_0.
\end{split}
\end{align*}
Observing \emph{cancellation} of the $((p(x,y,\partial)\phi_r)W_r,W_r)$ term in \eqref{d3a}, we see that  \eqref{d3a} becomes
\begin{align}\label{d5}
\begin{split}
(F_r,W_r) = &(T_{\phi_r} W_r,W_r)+((P_1\phi_r)W_r,W_r) \\
= & ((B_1\phi_r)W_r,W_r)+(O(1)W_r,W_r)\\
& + \left\langle\left(\sum_{k=1}^nq^{nk}\phi_{r,y_k}\right)W_r,W_r\right\rangle_{t_0}-(\phi_{r,x} W_r,W_r)_0\\
= & (O(1)W_r,W_r)+\left\langle\left(\sum_{k=1}^nq^{nk}\phi_{r,y_k}\right)W_r,W_r\right\rangle_{t_0}-(\phi_{r,x} W_r,W_r)_0.
\end{split}
\end{align}
Using $W_i=-W_r$ and $\partial_x\phi_i=-\partial_x\phi_r$ on $x=0$ we obtain from this the energy estimate
\begin{align*}
\begin{split}
&\left|\left\langle\left(\sum_{k=1}^nq^{nk}\phi_{r,y_k}\right)W_r,W_r\right\rangle_{t_0}\right|\leq |(F_r,W_r)|+C(W_r,W_r)+|((\partial_x\phi_i)W_i,W_i)_0|.
\end{split}
\end{align*}
Since $J_r$ is contained in a small neighborhood of $0$, it follows from \eqref{m7} that \\$\sum_{k=1}^nq^{nk}\phi_{r,y_k}\neq 0$, so 
\begin{align}\label{d6a}
\langle W_r,W_r\rangle_{t_0}\lesssim |(F_r,W_r)|+(W_r,W_r)+|((\partial_x\phi_i)W_i,W_i)_0|.
\end{align}
Gronwall's inequality then implies\footnote{If $y$ and  $\phi$ are nonnegative and continuous and satisfy 
$y(t)\leq C[\alpha+\int^t_{-T}(y(s)+\phi(s))ds]$ for some $C,\alpha>0$, then $y(t)\leq C[\alpha e^{Ct}+\int^t_{-T}e^{C(t-s)}\phi(s)ds]$; see \cite{chazarain1982}.}
\begin{align}\label{d6}
\langle W_r,W_r\rangle_{t_0}\lesssim (F_r,F_r)+|((\partial_x\phi_i)W_i,W_i)_0|.
\end{align}

Next consider $W_i$ in \eqref{e8y}.  For any $t_0\in[-T,T]$ we expect $W_i$ on $J_i\cap\{t=t_0\}$ to be determined by the data $F_i$ and $g$ of problem \eqref{e8y} in the set $J_{i,t_0}\subset J_i$, which we define as the backward flowout under $T_{\phi_i}$ in $\Omega_T$ of $J_i\cap\{t=t_0\}$.  The boundary of $J_{r,t_0}$ consists of two flat pieces, one in $\{t=t_0\}$ and one in $\{t=-T\}$, and a curved piece foliated by integal curves of $T_{\phi_i}$. 
Starting  from the transport equation
\begin{align*}
    \begin{cases}
        T_{\phi_i} W_i+(P_1\phi_i)W_i=F_i & \text{on }J_{i,t_0}\times \mathbb{T},\\
        W_i=g & \text{on }t=-T,
    \end{cases}
\end{align*}
and using similar notation for inner products, we apply essentially the same argument as above to obtain in place of \eqref{d6a}:
\begin{align}\label{d8}
\langle W_i,W_i\rangle_{t_0}\lesssim |(F_i,W_i)|+(W_i,W_i)+\langle g,g\rangle_{-T},
\end{align}
so Gronwall gives
\begin{align}\label{d9}
\langle W_i,W_i\rangle_{t_0}\lesssim (F_i,F_i)+\langle g,g\rangle_{-T}.
\end{align}

To control the trace term on the right in \eqref{d6} we first define $V=J_i\cap \{x=0\}$ as in \S \ref{intro}, and then define 
$J_{i,V}\subset J_i$ to be the backward flowout under $T_{\phi_i}$ in $\Omega_T$ of $V$.  The boundary of $J_{i,V}$ consists of two flat pieces, one in $\{x=0\}$ and one in $\{t=-T\}$, and a curved piece foliated by integal curves of $T_{\phi_i}$.  Starting  from the transport equation
\begin{align*}
    \begin{cases}
        T_{\phi_i} W_i+(P_1\phi_i)W_i=F_i & \text{on }J_{i,V}\times \mathbb{T},\\
        W_i=g & \text{on }t=-T,
    \end{cases}
\end{align*}
we estimate $W^i$ on $J_{i,V}$ by an argument parallel to the one that gave \eqref{d5}.    
In place of \eqref{d8} we obtain
\begin{align*}
|((\partial_x\phi_i)W_i,W_i)_0|\lesssim |(F_i,W_i)|+(W_i,W_i)+\langle g,g\rangle_{-T}.
\end{align*}
With \eqref{d9} this gives
\begin{align*}
|((\partial_x\phi_i)W_i,W_i)_0|\lesssim (F_i,F_i)+\langle g,g\rangle_{-T}.
\end{align*}

Summarizing, we have the following three estimates for any $t_0\in [-T,T]$:
\begin{align}\label{d12}
&\langle W_r,W_r\rangle_{t_0}\lesssim (F_r,F_r)+|((\partial_x\phi_i)W_i,W_i)_0| && \text{ on }J_{r,t_0}\times \mathbb{T},\nonumber\\
&\langle W_i,W_i\rangle_{t_0}\lesssim (F_i,F_i)+\langle g,g\rangle_{-T}&&\text{ on }J_{i,t_0}\times \mathbb{T},\\
&|((\partial_x\phi_i)W_i,W_i)_0|\lesssim (F_i,F_i)+\langle g,g\rangle_{-T}&&\text{ on }J_{i,V}\times \mathbb{T}. \nonumber
\end{align}
Since $W_r$ and $W_i$ are zero outside $J_r\times \mathbb{T}$ and $J_i\times \mathbb{T}$ respectively, we can combine these estimates to obtain for $t_0\in [-T,T]$:
\begin{align}\label{d13}
\begin{split}
& \langle W_r,W_r\rangle_t+\langle W_i,W_i\rangle_t+|((\partial_x\phi_i)W_i,W_i)_0| \\
& \lesssim (F_r,F_r)+(F_i,F_i)+\langle g,g\rangle_{-T}\text{ on }\Omega_T\times \mathbb{T}.
\end{split}
\end{align}
This estimate easily implies
\begin{align}\label{d17}
\|(W_r,W_i)\|_{L^2(\Omega_T\times \mathbb{T})}\leq C(T)\left(|(F_r,F_i)|+\langle g,g\rangle_{-T}\right)\text{ on }\Omega_T\times \mathbb{T},
\end{align}
where $C(T)\to0$ as $T\to 0$.

We also have the following classical Kreiss estimate for the problem \eqref{e6y}:\footnote{See Kreiss \cite{kreiss1970cpam} or Chazarain-Piriou \cite[Chapter 7]{chazarain1982}.}
\begin{align}\label{kreiss2}
\|u\|_{H^1(\Omega_T)}\leq C(T)\|\underline{f}\|_{L^2(\Omega_T)}+C\|u^1\|_{H^1(\Omega_{[-T,-T+\delta]})},
\end{align}
where $C(T)\to 0$ as $T\to 0$.    In the next section we use these estimates to rigorously solve the coupled linear problems \eqref{e6y}--\eqref{e8y}.

\subsection{Linear energy estimates: rigorous arguments}\label{rigorous}

Consider again the coupled linear problems \eqref{e6y}--\eqref{e8y}.
For $k\in\mathbb{N}$ choose a sequence $F^k_r\in C^\infty_c(\mathring{J}_r\times \mathbb{T})$, supported strictly away from the shadow boundary $\mathrm{SB}_+$, such that $F^k_r\to F_r$ in $L^2(\Omega_T\times \mathbb{T})$ as $k\to \infty$.  Similarly, choose 
 a sequence $F^k_i\in C^\infty_c(\mathring{J}_i\times \mathbb{T})$, supported strictly away from $\mathrm{SB}=\mathrm{SB}_+\cup \mathrm{SB}_-$, such that $F^k_i\to F_i$ in $L^2(\Omega_T\times \mathbb{T})$ as $k\to \infty$.    Finally, choose a sequence  $g^k\in C^\infty_c\left((\mathring{J}_i\cap\{t=-T\})\times \mathbb{T}\right)$ supported strictly away from $\mathrm{SB}_-\cap \{t=-T\}$, such that $g^k\to g$ in $L^2(\{t=-T\})$ as $k\to \infty$.   Next for each $k$ construct a $C^\infty$ solution $(W^k_r,W^k_i)$ to the coupled problems 
\begin{align*}
    & \begin{cases}
        T_{\phi_r}W^k_r+(P_1\phi_r)W^k_r=F^k_r & \text{in }\mathring{J_r}\times \mathbb{T},\\
        W^k_r(0,y,t,\theta)=-W^k_i(0,y,t,\theta) & \text{on } (J_r\cap\{ x=0 \})\times \mathbb T,\\
        W^k_r=0 & \text{on }(\Omega_T\setminus J_r)\times \mathbb{T};
    \end{cases} \\
    & \begin{cases}
        T_{\phi_i}W^k_i+(P_1\phi_i)W_i^k=F^k_i & \text{in }\mathring{J_i}\times \mathbb{T},\\
        W^k_i|_{t=-T}=g^k(x,y,\theta_i) & \text{on } (J_i\cap\{x=0\})\times \mathbb T,\\
        W^k_i=0 & \text{on }(\Omega_T\setminus J_i)\times \mathbb{T}.
    \end{cases}
\end{align*}
Both $W^k_i$, which is constructed first, and $W^k_r$ are easily constructed by integration along characteristics.   Since both are smooth and supported away from $\mathrm{SB}$, all the steps in the formal derivation of  the estimate \eqref{d13} apply rigorously to $W^k_i$ and $W^k_r$,  and we obtain
\begin{align}\label{d13z}
\begin{split}
& \langle W^k_r,W^k_r\rangle_t+\langle W^k_i,W^k_i\rangle_t+|((\partial_x\phi_i)W^k_i,W^k_i)_0| \\
& \lesssim (F^k_r,F^k_r)+(F^k_i,F^k_i)+\langle g^k,g^k\rangle_{-T}\text{ on }\Omega_T\times \mathbb{T}.
\end{split}
\end{align}

Passing to the limit as $k\to \infty$, we obtain a (unique) solution
$$(W_r, W_i)\in C\left([-T,T];L^2(\mathbb{R}^n_+\times \mathbb{T})\right)\times C\left([-T,T];L^2(\mathbb{R}^n_+\times \mathbb{T})\right)$$ to \eqref{e7y}--\eqref{e8y} that satisfies the estimate \eqref{d13}. 
The existence and continuity with respect to $x_0$ small of
\begin{align}\label{d18a}
((\partial_x\phi_i)W_i,W_i)_{x_0}\text{ and }((\partial_x\phi_r)W_r,W_r)_{x_0},
\end{align}
where the pairing is now taken in $L^2(y,t,\theta)$ for $x=x_0$ fixed, follows similarly.\footnote{Recall \eqref{d5}, which treats the case $x_0=0$.}

\Remark 
Here, of course, we have used the fact that the \emph{cancellation} of the bad term 
\begin{align}\label{d16}
((p(x,y,t,\partial)\phi_r)W^k_r,W^k_r)
\end{align}
in \eqref{d5} allows us to obtain an estimate \eqref{d13z} where the constant (implicit in $\lesssim$) is independent of $k$.   The term \eqref{d16} generally blows up as $k\to \infty$ because of the singularity in $\phi_r$.

A unique solution $u\in L^2(\Omega_T)$ to the problem \eqref{e6y} satisfying the estimate \eqref{kreiss2} is provided by \cite{kreiss1970cpam}.   This proves
\begin{prop}\label{d13y}
The coupled linear problems \eqref{e6y}--\eqref{e8y} have a  solution $(u,W_r,W_i)$
in $H^1(\Omega_T)\times L^2(\Omega_T\times \mathbb{T})\times L^2(\Omega_T\times \mathbb{T})$ which satisfies the estimates \eqref{d13}--\eqref{kreiss2}.
The functions $W_r$ and $W_i$ are supported in $J_r$ and $J_i$ respectively.  Both $W_r$ and $W_i$ lie in $C\left([-T,T];L^2(\mathbb{R}^n_+\times \mathbb{T})\right).$  Moreover, the inner products \eqref{d18a} are continuous in $x_0$ for $x_0$ small.
\end{prop}

\subsection{Convergence of the Picard iterates.}
 
 Now we apply Proposition \ref{d13y} to the problems \eqref{e6z}--\eqref{e8z} for the $(n+1)$-st iterate $(u^{n+1},W^{n+1}_r,W^{n+1}_i).$ 
  Assumption \ref{A0z} on the nonlinear function $f(x,y,t,\cdot,\cdot)$ implies 
\begin{align*}
\|\underline{f}(u^n,W^n_r,W^n_i)\|_{L^2(\Omega_T)}\lesssim \|u^n\|_{L^2(\Omega_T)}+\|(W^n_r,W^n_i)\|_{{L^2(\Omega_T\times \mathbb{T}) \times L^2(\Omega_T\times \mathbb T)}},
\end{align*}
with  similar estimates for $f^*_r(u^n,W^n_r,W^n_i)$ and $f^*_i(u^n,W^n_r,W^n_i)$.    A standard argument using the estimates 
\eqref{d17} and  \eqref{kreiss2}  shows that for some $T>0$ the iterates $(u^{n+1},W^{n+1}_r,W^{n+1}_i)$ converge to a limit
 $(u,W_r,W_i)\in H^1(\Omega_T)\times L^2(\Omega_T\times \mathbb{T})\times L^2(\Omega_T\times \mathbb{T})$.   Having fixed $T$ small enough, another application of estimate \eqref{d13} yields 
 $$(W_r,W_i)\in C\left([-T,T];L^2(\mathbb{R}^n_+\times \mathbb{T})\times L^2(\mathbb{R}^n_+\times \mathbb{T})\right).$$
 \color{red}
 \color{black}
The existence and continuity with respect to $x_0$ small of
\begin{align}\label{d18}
((\partial_x\phi_i)W_i,W_i)_{x_0}\text{ and }((\partial_x\phi_r)W_r,W_r)_{x_0},
\end{align}
where the pairing is now taken in $L^2(y,t,\theta)$ for $x=x_0$ fixed, follows similarly.  Thus, we may conclude that the limit of the iterates satisfies \eqref{e6}--\eqref{e8}.   This proves
\begin{prop}\label{d13w}
There exists a $T>0$ such that the  nonlinear profile equations \eqref{e6}--\eqref{e8} have a  solution $(u,W_r,W_i)$
in $H^1(\Omega_T)\times L^2(\Omega_T\times \mathbb{T})\times L^2(\Omega_T\times \mathbb{T})$.
The functions $W_r$ and $W_i$ are supported in $J_r$ and $J_i$ respectively.  Both $W_r$ and $W_i$ lie in $C\left([-T,T];L^2(\mathbb{R}^n_+\times \mathbb{T})\right)$ and the inner products  \eqref{d18} 
are continuous in $x_0$ for $x_0$ small.
\end{prop}

\section{Truncation and regularization }\label{TR}

This section is largely inspired by ideas from \cite{cheverry1996} and \cite{dumas2002}.  For the error analysis we need to employ a more careful truncation and regularization process than the one used in \S \ref{rigorous}.   In particular, we want the truncator to have the commutation property \eqref{g1}, so we should ``truncate along the flow".

We first truncate $W_r$, $W_i$ near $\mathrm{SB}_+$ and $\mathrm{SB}$, respectively, in a way that preserves the boundary condition. Using a clever idea of \cite{dumas2002}, we regularize first in the tangential variables $(y,t,\theta)$,  then use the profile equations to deduce extra regularity in $x$,  and finally regularize in the normal variable $x$ in a way that preserves the boundary condition.   This procedure is more transparent in its effect on traces than the one in \cite{cheverry1996}.  Moreover, it does not depend on an explicit calculation of the singularity of the flow map $Z^r$ at the glancing set, so it applies more readily to problems involving higher order grazing.

\subsection{Truncation}
~

\noindent
{\bf Notations.}
1. As in \eqref{e3} and \eqref{e4} we sometimes write $(x,y,t)={Z_r}(s,y',t')$, where $s$ is a flow parameter and the primes indicate that $(y',t')$ specifies an  \emph{initial point} on $x=0$ for the flow.  The primes are helpful here, but in other contexts we usually drop them.

\noindent
2. Let $(x,y,t)=\Phi(x,z)=(x,\Phi_2(z))$ be the $C^1$ diffeomorphism that relates the  standard form $(x,y,t)$ coordinates and the $(x,z)$ coordinates of Proposition \ref{m20},  in which the grazing set $G_{\phi_i}$ near $0$  is the subset of $x=0$ defined by $z_1=0$.    Denote by $\cD^r_{pre}$ the preimage of $\cD^r$ as in \eqref{e4} under the map $(s,z)\mapsto (s,y,t)=(s,\Phi_2(z))$.   

\noindent
3. Let $\Xi^r:L^2(J_r\times \mathbb{T})\to L^2(\cD^r_{pre}\times \mathbb{T},j(s,z)dsdzd\theta)$ be the pullback map given by\footnote{Here $j(s,z)$ is the $C^1$ Jacobian of the map $(s,z)\mapsto Z_r(s,\Phi_2(z))$.   Assumption \ref{A3} implies that $\Xi^r$ is well-defined.}
\begin{align*}
(\Xi^r f)(s,z,\theta):=f(Z_r(s,\Phi_2(z)),\theta).
\end{align*}

\begin{figure}[t]
    \begin{center}
        \includegraphics[scale=0.8]{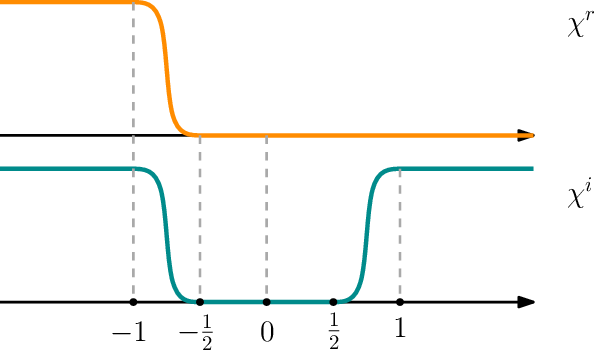}
        \caption{Cutoff functions used in the truncation process.}
        \label{truncation}
    \end{center}
\end{figure}

Suppose that $u\in H^1(\Omega_T)$,  $W_r, W_i \in L^2(\Omega_T)$ is the solution to the profile equations \eqref{e6}--\eqref{e8} provided by Proposition \ref{d13w}.  Let $\chi^r\geq 0$ be a $C^\infty$, decreasing cutoff function such that 
$\chi^r=1$ on $(-\infty,-1]$ and $\chi^r=0$ on $[-1/2,\infty)$.
We truncate $W_r(x,y,t,\theta)$ along $\mathrm{SB}_+$ by defining for $\mu>0$
\begin{align*}
\begin{split}
&W_{r,\mu}(x,y,t,\theta)=\chi^r_{\mu}(x,y,t)W_r(x,y,t,\theta),\text{ where }\chi^r_{\mu}:=(\Xi^r)^{-1}\chi^r(z_1/\mu)\text{ on }J_r.
\end{split}
\end{align*}
We smoothly extend $\chi^r_\mu$ to be zero in the shadow region and to be one on the remaining part of $\Omega_T$.  
Since $j(s,z)$ is $C^1$ even near $s=z_1=0$, we have
\begin{align}\label{e11z}
 \|W_{r,\mu}(x,y,t,\theta)-W_r\|_{L^2(\Omega_T\times \mathbb{T})}=o_\mu(1).
\end{align}

Next we define $W_{i,\mu}$ using the nonsingular flow map ${Z_i}$.   We let $J_{i,e} \supset J_i$ be the extension of $J_i$ defined by 
\begin{align}\label{e12}
J_{i,e}=\{{Z_i}(s,x,y) \ | \ 0\leq s\leq s_e(x,y), (x,y,-T)\in U\}:={Z_i}(\cD^i_e),
\end{align}
where $s_e(x,y)$  is the value of $s$ for which the $t$-component of ${Z_i}(s,x,y)$ is $T$.\footnote{Unlike the range of  ${Z_r}$, the range of ${Z_i}$ can be taken to be a full neighborhood of $0$ in $\mathbb{R}^{n+1},$ and we do that now.   Working with $s_e(x',y')$ and $J_{i,e}$ allows us to avoid difficulties arising from the case by case definition of $s(x',y')$ in \eqref{e2a}. }  
Denote by $\cD^i_{pre}$ the preimage of $\cD^i_e$ as in \eqref{e12} under the map $(s,x,z)\mapsto (s,x,y)=(s,\Phi_d(x,z))$, where $\Phi_d$ is defined by
\begin{align*}
(x,y)=\Phi_d(x,z)\Leftrightarrow (x,y,-T)=\Phi(x,z).
\end{align*}
Let $\Xi^i:L^2(J_{i,e}\times \mathbb{T})\to L^2(\cD^i_{pre}\times \mathbb{T})$ be the pull-back map given by
\begin{align*}
(\Xi^i f)(s,x,z,\theta):=f({Z_i}(s,\Phi_d(x,z)),\theta).
\end{align*}

Let $\chi^i\geq 0$ be a $C^\infty$ cutoff function such that $\chi^i=1$ on $\{t\leq -1\text{ or }t\geq 1\}$, $\chi^i=0$ on $\{-1/2\leq t\leq 1/2\}$, and $\chi^i=\chi^r$ on $[-1,0]$.    We can then truncate $W_i(x,y,t,\theta)$ along $\mathrm{SB}=\mathrm{SB}_+\cup \mathrm{SB}_-$ by 
\begin{align*}
W_{i,\mu}(x,y,t,\theta)=\chi^i_\mu(x,y,t)W_i(x,y,t,\theta),
\end{align*}
where  we have set $\chi^i_\mu(x,y,t)= (\Xi^i)^{-1}\left(\chi^i\left(\frac{z_1}{\mu}\right)\right)$ on $J_i$, and   we smoothly extend $\chi^i_\mu$ to the rest of $\Omega_T$.

Observe that we have the commutation property 
\begin{align}\label{g1}
 [T_{\phi_r},\chi^r_\mu]=[T_{\phi_i},\chi^i_\mu]=0\text{ on }\Omega_T.
\end{align}

\Remarks 
1. The truncations and extensions defined above imply that \eqref{g1} makes sense on $\Omega_T$, even though $T_{\phi_r}$ and $T_{\phi_i}$ are just defined on $J_r$ and $J_i$ respectively.  In the future we will often omit remarks of this nature.

\noindent
2. Recall that the illuminated region of the boundary in $(x,z)$ coordinates is $z_1\leq 0$, and we chose $\chi^i=\chi^r$ on $[-1,0]$.   Then from the definition of the reflected flow and the fact that 
$\chi^i_\mu$ is constant on integral curves of $T_{\phi_i}$, it follows that 
$\chi^r_\mu=\chi^i_\mu\text{ on }x=0$, so the boundary condition is preserved by truncation:
\begin{align*}
W_{r,\mu}+W_{i,\mu}=0\text{ on }x=0.
\end{align*}

\subsection{Regularization}

For $\rho_1>0$ let $\delta_{\rho_1}(y,t,\theta)$ be a smooth approximate identity supported in $|(y,t,\theta)|\leq \rho_1$.  Define \emph{tangential}  regularizations for $k=r,i$ by\footnote{Tangential regularization preserves the boundary condition.    Here \eqref{g4b} means that for fixed $\mu$, the quantity on the left $\to 0$ as $\rho_1\to 0$.} 
\begin{subequations}\label{g4}
    \begin{align}
    & W_{k,\mu,\rho_1}=R^{\rho_1}W_{k,\mu}:=\delta_{\rho_1}*W_{k,\mu}, \text{ and thus} \label{g4a}\\
    & \|W_{k,\mu,\rho_1}- W_{k,\mu}\|_{L^2(\Omega_T\times \mathbb{T})}=o_{\rho_1}(1). \label{g4b}
    \end{align}
\end{subequations}
Using \eqref{g1}, we compute
\begin{align*}
\begin{split}
T_{\phi_k}W_{k,\mu,\rho_1}
= & T_{\phi_k}R^{\rho_1}W_{k,\mu}=R^{\rho_1}T_{\phi_k}W_{k,\mu}+[T_{\phi_k},R^{\rho_1}]W_{k,\mu}\\
= & (T_{\phi_k}W_k)_{\mu,\rho_1}+[T_{\phi_k},R^{\rho_1}]W_{k,\mu}.
\end{split}
\end{align*}
Using a similar computation of $(P_1\phi_k)W_{k,\mu,\rho_1}$ together with the profile equations \eqref{e7}--\eqref{e8}, we  obtain
\begin{align}\label{g6}
\begin{split}
&T_{\phi_k}W_{k,\mu,\rho_1}+(P_1\phi_k)W_{k,\mu,\rho_1}\\
& =  f^*_k(u,W_r,W_i)_{\mu,\rho_1}+[T_{\phi_k},R^{\rho_1}]W_{k,\mu}+[P_1(\phi_k),R^{\rho_1}]W_{k,\mu}\\
& =  f^*_k(u,W_r,W_i)_{\mu}+o_{\rho_1}(1)\text{ in }L^2(\Omega_T\times \mathbb{T}).
\end{split}
\end{align}
Here we use Friedrich's lemma to treat the first commutator and write
\begin{align*}
[P_1\phi_k,R^{\rho_1}]W_{k,\mu}=(I-R^{\rho_1})(P_1\phi_k)W_{k,\mu}+(P_1\phi_k)(W_{k,\mu,\rho_1}-W_{k,\mu})
\end{align*}
for the second. \footnote{The function $P_1\phi_k$ and the coefficients of $T_{\phi_k}$ are smooth on the support of $W_{k,\mu}$.}

Before regularizing in $x$ we set
\begin{align*}
V_{1,\mu,\rho_1}=W_{i,\mu,\rho_1}-W_{r,\mu,\rho_1},\; V_{2,\mu,\rho_1}=W_{i,\mu,\rho_1}+W_{r,\mu,\rho_1},\; V_{\mu,\rho_1}=\begin{pmatrix}V_{1,\mu,\rho_1}\\V_{2,\mu,\rho_1}\end{pmatrix},
\end{align*}
and define for $x_0>0$ small:
\begin{align*}
\begin{split}
&\Omega_{T,x_0}:=\Omega_T\cap \{0\leq x\leq x_0\}\\
&\Omega^e_{T,x_0}=\{(x,y,t)\in\mathbb{R}^n \ | \ t\in[-T,T], -\infty\leq x\leq x_0\}.
\end{split}
\end{align*}
We can rewrite the equations \eqref{g6} on $\Omega_{T,x_0}$ and the boundary condition as 
\begin{align}\label{g8}
 \begin{split}
 \partial_x V_{\mu,\rho_1} & =A_1\partial_yV_{\mu,\rho_1}+A_2\partial_tV_{\mu,\rho_1}+BV_{\mu,\rho_1}+C\in L^2(\Omega_{T,x_0}\times \mathbb{T})\\
 V_{2,\mu,\rho_1} & =0\text{ on }x=0,
 \end{split}
 \end{align}
where the matrices $A_j$ and $B$ can be taken to be smooth on $\Omega_{T,x_0}\times \mathbb{T}$.   Here we use the fact that for $k=r,i$ the coefficients of $\partial_x$ in $T_{\phi_k}$, namely $\partial_x\phi_k$, are nonvanishing near $x=0$ away from the grazing set,  while $V_{\mu,\rho_1}$ vanishes near the grazing set due to truncation.

 The equations \eqref{g8} imply that 
$V_{\mu,\rho_1}\in H^1(\Omega_{T,x_0}\times \mathbb{T})$ and that the zero extension of $V_{2,\mu,\rho_1}$ lies in 
$H^1(\Omega^e_{T,x_0}\times\mathbb{T})$.  After extending $V_{1,\mu,\rho_1}$ as an element of $H^1(\Omega^e_{T,x_0}\times \mathbb{T})$, for $\rho_2>0$ we define regularizations of these extensions by
\begin{align}\label{g9}
V_{k,\mu,\rho_1,\rho_2}:=\delta_{\rho_2}*V_{k,\mu,\rho_1}, \;k=r,i,
\end{align}
where $\delta_{\rho_2}(x)$ is an approximate identity supported in $ 0\leq x\leq 1$.  Hence the boundary condition $V_{2,\mu,\rho_1,\rho_2}=0$ on $x=0$ is preserved.\footnote{This argument involving the $V_{k,\mu,\rho_1,\rho_2}$ is close to an argument in \cite{dumas2002}.}    

Let $\rho:=(\rho_1,\rho_2)$.  By standard properties of approximate identities we have
\begin{align*}
\|V_{k,\mu,\rho}-V_{k,\mu,\rho_1}\|_{H^1(\Omega_{T,x_0}\times \mathbb{T})}\to 0 \text{ as }\rho_2\to 0.
\end{align*}
Now define $W_{k,\mu,\rho}$ in the obvious way from the $V_{k,\mu,\rho}$.   The above properties imply for $k=r,i$:
\begin{align}\label{g11}
\begin{split}
&W_{k,\mu,\rho}\to W_{k,\mu,\rho_1} \text{ in } H^1(\Omega_{T,x_0}\times \mathbb{T})\text{ as }\rho_2\to 0; \text{ hence}\\
&T_{\phi_k}W_{k,\mu,\rho}+(P_1\phi_k)W_{k,\mu,\rho} \\
& \to T_{\phi_k}W_{k,\mu,\rho_1}+(P_1\phi_k)W_{k,\mu,\rho_1}\text{ in }L^2(\Omega_{T,x_0}\times\mathbb{T})\text{ as }\rho_2\to 0.
\end{split}
\end{align}

Using \eqref{g6} and \eqref{g11}, we obtain
\begin{align}\label{g11a}
\begin{split}
&T_{\phi_k}W_{k,\mu,\rho}+(P_1\phi_k)W_{k,\mu,\rho}\\ 
& = f^*_k(u,W_r,W_i)+o_\mu(1)+o_{\rho_1}(1)+o_{\rho_2}(1) \text{in }L^2(\Omega_{T,x_0}\times \mathbb{T}).
 \end{split}
 \end{align}
We can extend \eqref{g11a} to hold on $L^2(\Omega_{T}\times \mathbb{T})$ by observing that for $x\geq x_0/2$ and $\rho_2$ small the convolution \eqref{g9} evaluated at $x$ depends on $V_{k,\mu,\rho_1}(x')$ only for $|x-x'|\leq \rho_2$; so it is unaffected by the extensions into $x<0$ that were taken.   A repetition of the computation  \eqref{g6} in $x\geq x_0$ with tangential convolution replaced by convolution in all variables  yields the claimed extension of \eqref{g11a}.

Summarizing we have
\begin{subnumcases}{\label{g12}}
    T_{\phi_k}W_{k,\mu,\rho}+(P_1\phi_k)W_{k,\mu,\rho} & \nonumber  \\
    \quad = f^*_k(u,W_r,W_i)+o_\mu(1)+o_{\rho_1}(1)+o_{\rho_2}(1)
    & in $L^2(\Omega_T\times \mathbb{T})$, \label{g12a}\\
    W_{r,\mu,\rho}+W_{i,\mu,\rho}=0 & on $x=0$, \label{g12b}\\
    W_{i,\mu,\rho}|_{[-T,-T+\delta]} \nonumber \\
    \quad =W_1|_{[-T,-T+\delta]}+o_{\mu}(1) +o_{\rho_1}(1)+o_{\rho_2}(1) 
    & in $L^2(\Omega_{[-T,-T+\delta]}\times\mathbb{T})$. \label{g12c}
\end{subnumcases}

\Remark 
Here \eqref{g12a} tells us, for example,  that for fixed $\mu$ and $\rho_1$, the quantity $o_{\rho_2}(1)\to 0$ in $L^2(\Omega_T\times \mathbb{T})$, where
\begin{align*}
o_{\rho_2}(1)=\left(T_{\phi_k}W_{k,\mu,\rho}+(P_1\phi_k)W_{k,\mu,\rho}\right)-\left(T_{\phi_k}W_{k,\mu,\rho_1}+(P_1\phi_k)W_{k,\mu,\rho_1}\right).
\end{align*}
The order of fixing parameters -- $\mu,\rho_1,\rho_2$ --  is important.

\section{Error analysis}\label{ea}

In this section we complete the proof of Theorem \ref{mt2}. We begin by stating a couple of useful and rather well-known lemmas, which sometimes allow us to work with functions of $(x,y,t,\theta)$ rather than $(x,y,t,\eps)$.

\begin{lem}[{\cite[Proposition 3.3]{jmr1996cpam}}]\label{jmr1a}
Let $\omega$ be a relatively compact open subset of $\mathbb{R}^{n+1}_{x,y,t}$, and suppose $\phi\in C^1(\overline{\omega})$ is such that $\nabla_{x,y,t}\phi$ is never $0$ on $\overline{\omega}$.
Then if  $a(x,y,t,\theta)\in L^2(\omega;H^1(\mathbb{T}))$, we have 
\begin{align*}
\varlimsup_{\eps\to 0}\|a(x,y,t,\phi/\eps)\|_{L^2(\omega)}\leq (2\pi)^{-1/2}\|a(x,y,t,\theta)\|_{L^2(\omega\times \mathbb{T})}.
\end{align*}
\end{lem}

 We also need the following extension of Lemma \ref{jmr1a}, whose proof is similar.  

\begin{lem}\label{jmr2}
Let $\omega$ be a relatively compact open subset of $\mathbb{R}^{n+1}_{x,y,t}$, and suppose $\phi_i\in C^1(\overline{\omega})$ are such that $\nabla_{x,y,t}\phi_1$ and  $\nabla_{x,y,t}\phi_2$ are linearly independent at each $(x,y,t)\in \overline{\omega}$. 
If  $a(x,y,t,\theta_1,\theta_2)\in L^2(\omega;H^2(\mathbb{T}^2))$, we have 
\begin{align*}
\varlimsup_{\eps\to 0}\|a(x,y,t,\phi_1/\eps,\phi_2/\eps)\|_{L^2(\omega)}\leq (2\pi)^{-1}\|a(x,y,t,\theta_1,\theta_2)\|_{L^2(\omega\times \mathbb{T}^2)}.
\end{align*}
\end{lem}

The error estimate in \S \ref{erest} uses a classical estimate for the following linear boundary problem on $\Omega_T$:
\begin{align*}
\begin{cases}
P(x,y,t,\partial) u=f & \text{in }\Omega_T,\\
u(0,y,t)=b(y,t) & \text{on } b\Omega_T,\\
u=u^1(x,y,t) & \text{on } \Omega_{[-T,-T+\delta]}.
\end{cases}
\end{align*}
We have\footnote{See Kreiss \cite{kreiss1970cpam} or Chazarain-Piriou \cite[Chapter 7]{chazarain1982}.}
\begin{align}\label{kreiss}
\|u\|_{H^1(\Omega_T)}\leq C(T)\left(\|f\|_{L^2(\Omega_T)}+\langle b\rangle_{H^1(b\Omega_T)}\right)+C\|u^1\|_{H^1(\Omega_{[-T,-T+\delta]})},
\end{align}
where $C(T)\to 0$ as $T\to 0$.   Here $b\Omega_T:=\{(y,t)\ | \ (0,y,t)\in \Omega_T\}$ and $\langle\cdot\rangle$ indicates a norm on $b\Omega_T$.

\begin{prop}\label{noresonance}
    The incoming phases and the reflected phases are nonresonant, in the following sense:
    \begin{enumerate}[label={\arabic*.}]
        \item For any $(x,y,t)\in (J_i\cap J_r)\setminus \mathrm{SB}_+$, the two vectors $\nabla \phi_i(x,y,t)$, $\nabla \phi_r(x,y,t)$ are linearly independent;

        \item For any $k_i, k_r\in \RR$, $k_ik_r\neq 0$, the function $\phi:=k_i\phi_i+k_r\phi_r$ is nowhere characteristic on  $(J_i\cap J_r)\setminus \mathrm{SB}_+$, meaning that
        \[ p(x,y,t, k_i\nabla \phi_i(x,y,t)+k_r \nabla \phi_r(x,y,t))\neq 0\;\;\forall (x,y,t)\in (J_i\cap J_r)\setminus \mathrm{SB}_+. \]
    \end{enumerate}
\end{prop}
The proof presented here is modified from \cite[Lemma 1.2]{dumas2002}.
\begin{proof}[Proof of Proposition \ref{noresonance}]
    1. Suppose the contrary, then there exists $(x,y,t)\in ( J_i\cap J_r)\setminus \mathrm{SB}_+$, and $k_i,k_r\in \RR$, $k_ik_r\neq 0$, such that $k_i\nabla \phi_i(x,y,t)+k_r\nabla \phi_r(x,y,t)=0$. Then $\nabla \phi_r(x,y,t) = a\nabla \phi_i(x,y,t)$ with $a:=-\frac{k_i}{k_r}$. Let $\gamma_i(s):=(m_i(s), \nu_i(s))$ be the null bicharacteristic of $p$ satisfying $(m_i(0), \nu_i(0))=(x,y,t;\nabla \phi_i(x,y,t))$. Let $\widetilde{\gamma}(s):=(m_i(as), a\nu_i(as))$. Then, since $p$ is homogeneous of order $2$ in $\nu$, on can check that $\widetilde{\gamma}$ satisfies 
    \[ \dot{\widetilde{\gamma}}(s) = H_p(\widetilde{\gamma}(s)), \ \widetilde{\gamma}(0) = (m_i(0),a\nu_i(0)) = (x,y,t;\nabla \phi_r(x,y,t)). \]
    This implies that $\widetilde{\gamma}=\gamma_r:=(m_r,\nu_r)$, where $\gamma_r$ is the null bicharacteristic passing through $(x,y,t;\nabla \phi_r(x,y,t))$ at $s=0$. Therefore
    \[ m_r(s)=m_i(as), \ \nu_r(s)=a\nu_i(as). \]
    In particular, there exists $s_0\in \RR$ such that $m_r(s_0)=m_i(as_0)=:m_0\in \{x=0\}\setminus G_{\phi_i}$ and $\nu_r(s_0)=a\nu_i(as_0)$. This is impossible by the choice of $(m_0,\nu_r(m_0))$ in \S \ref{reflectedphase}.

    2. Suppose the contrary.  Relabeling $\phi$, $\phi_i$, $\phi_r$ as $\phi_{\ell}$, $\ell=1,2,3$, and after replacing $\phi_{\ell}$ with $-\phi_{\ell}$ if necessary, we can assume that  there exist $k_{\ell}>0$ such that for some $(x,y,t)\in (J_i\cap J_r)\setminus \mathrm{SB}_+$:
    \[ k_1\nabla\phi_1(x,y,t)+k_2\nabla\phi_2(x,y,t)+k_3\nabla\phi_3(x,y,t)=0. \]
    We denote $X_{\ell}:=\nabla \phi_{\ell}(x,y,t)\in \RR^{n+1}\setminus \{0\}$, and let $\mathcal P$ be the quadratic form $p(x,y,t,\cdot,\cdot)$ on $\RR^{n+1}$. Then 
    \[ \sum_{1\leq \ell\leq 3}k_{\ell}X_{\ell}=0, \ \mathcal P(X_{\ell},X_{\ell})=0. \]
    Since $\mathcal P$ has signature $(n,1)$, after changing of coordinates by a linear transformation, we can assume $\mathcal P$ takes the form 
    \[ \mathcal P(X,X)=\sum_{1\leq j\leq n}c_j (X^j)^2-c_{n+1}(X^{n+1})^2, \ X=(X^1,\cdots, X^{n+1}) \]
    with $c_j>0$, $1\leq j\leq n+1$. Since all $k_{\ell}$ are positive, without loss of generality we can assume $X_1^{n+1}, X_2^{n+1}>0$. Then 
    \[ \mathcal P(X_3,X_3)=0 \ \Rightarrow \ \mathcal P(k_1X_1+k_2X_2, k_1X_1+k_2X_2)=0 \ \Rightarrow \ \mathcal P(X_1, X_2)=0. \]
    On the other hand,
    \[\begin{split} 
    \mathcal P(X_1, X_2) 
    = & \sum_{1\leq j\leq n}c_j X_1^j X_2^j - c_{n+1}X_1^{n+1}X_2^{n+1} \\
    = & \sum_{1\leq j\leq n}c_j X_1^j X_2^j - \sqrt{\sum_{1\leq j\leq n} c_j(X_1^j)^2 } \sqrt{ \sum_{1\leq j\leq n} c_j (X_2^j)^2 }\leq 0
    \end{split}\]
    by the Cauchy-Schwarz inequality, with equality holding if and only $X_1$, $X_2$ are colinear. Since $k_3>0$, this implies $X_{\ell}$, $\ell=1,2,3$ are colinear. But this contradicts part 1 of the proposition.
\end{proof}

\subsection{The TR approximate solution \texorpdfstring{$m^l_{\mu,\rho,M,\eps}$}{TEXT}.}\label{approx}

We now define the  truncated and regularized (TR) approximate solution
\begin{align}\label{f1}
\begin{split}
m^l_{\mu,\rho,M,\eps}(x,y,t):=u^l_{\rho}(x,y,t) & +\eps U^l_{r,\mu,\rho}\left(x,y,t,\frac{\phi_r}{\eps}\right)+\eps U^l_{i,\mu,\rho}\left(x,y,t,\frac{\phi_i}{\eps}\right) \\
& + \eps^2 U^{M}_{\mathrm{nc}}\left(x,y,t,\frac{\phi_r}{\eps},\frac{\phi_i}{\eps}\right).
\end{split}
\end{align}
Here the superscript $l$ indicates that $(u^l,W_r^l,W^l_i)$ is the solution to the same profile equations \eqref{e6}--\eqref{e8} as $(u,W_r,W_l)$, except that the initial  data $W_1(x,y,t,\theta_i)$ in \eqref{e8} is replaced by a trigonometric polynomial  $W^l_1$ as in Definition \ref{meaning}.\footnote{Because the problem is nonlinear, note that  $W_r^l$ and $W^l_r$ are not necessarily  trigonometric polynomials.}

\Remark 
The sublinearity of $f(x,y,t,\cdot,\cdot)$ in its last two arguments along with the Kreiss estimate \eqref{kreiss} and the estimates of \S \ref{ee} imply that 
\begin{align}\label{f1y}
\|u-u^l\|_{H^1(\Omega_T)}+\|W_r-W^l_r\|_{L^2(\Omega_T\times \mathbb{T})}+\|W_i-W^l_i\|_{L^2(\Omega_T\times \mathbb{T})}\lesssim \delta_l.  
\end{align}

In \eqref{f1} we have set  $\rho=(\rho_0,\rho_1,\rho_2)$, where $\rho_i$, $i=1,2$ are as before, and $\rho_0>0$ is a regularization parameter for $u^l$.  The TR objects $W^l_{k,\mu,\rho}$ are defined as in \S \ref{TR}, and $U^l_{k,\mu,\rho}$ is the unique periodic $\theta_k$-primitive with mean zero of   $W^l_{k,\mu,\rho}$, $k=r,i$.   The term $\eps^2 U^{M}_{\mathrm{nc}}$ is a corrector designed to solve away \emph{most of} a term similar to  $f^*_{\mathrm{nc}}$ as in \eqref{e5}.      
We will describe $u^l_\rho$ and $U^{M}_{\mathrm{nc}}$ after introducing some notation.

\noindent
{\bf Notations.}
Here are some abuses of notation that we often commit below.
\begin{align*}
\begin{split}
Pu & =P(x,y,t,\partial)u,\\
f(m^l_{\mu,\rho,M,\eps})&:=f(x,y,t,m^l_{\mu,\rho,M,\eps}, \nabla m^l_{\mu,\rho,M,\eps}),\\
\underline{f}(u^l,W^l_r,W^l_i)&:=\underline{f}(x,y,t,u^l,\nabla u^l+W^l_{r}\nabla\phi_r+W^l_{i}\nabla\phi_i),\\
\underline{f}(u^l_\rho,W^l_{r,\mu,\rho},W^l_{i,\mu,\rho})&:=\underline{f}(x,y,t,u^l_\rho,\nabla u^l_\rho+W^l_{r,\mu,\rho}\nabla\phi_r+W^l_{i,\mu,\rho}\nabla\phi_i),\\
f^*_r(u^l,W^l_r,W^l_i)&:=f^*_r(x,y,t,u^l,\nabla u^l+W^l_{r}\nabla\phi_r+W^l_{i}\nabla\phi_i),\\
f^*_i(u^l_\rho,W^l_{r,\mu,\rho},W^l_{i,\mu,\rho})&:=f^*_i(x,y,t,u^l_\rho,\nabla u^l_\rho+W^l_{r,\mu,\rho}\nabla\phi_r+W^l_{i,\mu,\rho}\nabla\phi_i),\\
& \text{ etc...}
\end{split}
\end{align*}
We also recall that we use $\underline{f}$, $f^*_r$, $f^*_i$ denote respectively the mean of $f(\cdot)$ with respect to $(\theta_r,\theta_i)$, the mean with respect to $\theta_i$ minus $\underline{f}$, and the mean with respect to $\theta_r$ minus $\underline{f}$.    Finally,  
\begin{align*}
f^*_{\mathrm{nc}}:=f(\cdot)-(\underline{f}+f^*_r+f^*_i).
\end{align*}
We often rely on the context to make it clear whether $\theta_r$, $\theta_i$ are evaluated at $\phi_r/\eps$, $\phi_i/\eps$ or not.

To define $u^l_\rho$ recall that $u^l$ satisfies
\begin{align*}
\begin{cases}
Pu^l=\underline{f}(u^l,W^l_r,W^l_i):=\underline{F} & \text{in }\Omega_T,\\
u^l(0,y,t)=0 & \text{on } \Omega_T\cap\{x=0\},\\
u^l=u^1 & \text{on }\Omega_{[-T,-T+\delta]}.
\end{cases}
\end{align*}
Choose $C^\infty$ functions $\underline{F}_{\rho_0}\to \underline{F}$ in $L^2(\Omega_T)$ and $u^1_{\rho_0}\to u^1$ in $H^1(\Omega_{[-T,-T+\delta]})$ as $\rho_0\to 0$.\footnote{These functions are easily chosen to satisfy compatibility conditions to infinite order at the corner.} Define $u^l_\rho$ as the $C^\infty$ solution of
\begin{align}\label{h2}
\begin{cases}
Pu^l_\rho=\underline{F}_{\rho_0} & \text{in } \Omega_T,\\
u^l_\rho(0,y,t)=0 & \text{on } \Omega_T\cap\{x=0\},\\
u^l_{\rho}=u^1_{\rho_0} &\text{on }\Omega_{[-T,-T+\delta]}.
\end{cases}
\end{align}
The estimate \eqref{kreiss} implies\footnote{We need this regularization of $u^l$ later to make sense of the trace of $U^{M}_{\mathrm{nc}}$ on $x=0$.}
\begin{align}\label{h3}
u^l_\rho\to u^l\text{ in }H^1(\Omega_T)\text{ as }\rho_0\to 0.
\end{align}
Moreover, the definition of $u^l_\rho$ implies
\begin{align*}
Pu^l_\rho=\underline{f}(u^l,W^l_r,W^l_i)+o_{\rho_0}(1)\text{ in }L^2(\Omega_T).
\end{align*}

Next we define the corrector $U^M_{\mathrm{nc}}$.  
Using Lemma \ref{jmr2}, we may write
\begin{align}\label{f1a}
\begin{split}
& f(x,y,t,m^l_{\mu,\rho,M,\eps}, \nabla m^l_{\mu,\rho,M,\eps}) \\
& =f(x,y,t,u^l_\rho,\nabla u^l_\rho+W^l_{r,\mu,\rho}\nabla\phi_r+W^l_{i,\mu,\rho}\nabla\phi_i)+o_\eps(1) \text{ in }L^2(\Omega_T),
\end{split}
\end{align}
where, similar to \eqref{e5},\footnote{In both \eqref{f1a} and \eqref{f1aa} we set $\theta_r=\phi_r/\eps$, $\theta_i=\phi_i/\eps$.}
\begin{align}\label{f1aa}
\begin{split}
&f(x,y,t,u^l_\rho,\nabla u^l_\rho+W^l_{r,\mu,\rho}\nabla\phi_r+W^l_{i,\mu,\rho}\nabla\phi_i)\\
& =  \underline{f}(u^l_\rho,W^l_{r,\mu,\rho},W^l_{i,\mu,\rho})+f^*_r(u^l_\rho,W^l_{r,\mu,\rho},W^l_{i,\mu,\rho})+f^*_i(u^l_\rho,W^l_{r,\mu,\rho},W^l_{i,\mu,\rho}) \\
& \quad + f^*_{\mathrm{nc}}(u^l_\rho,W^l_{r,\mu,\rho},W^l_{i,\mu,\rho}).
\end{split}
\end{align}
The absence of resonances (Proposition \ref{noresonance}) implies that the term $f^*_{\mathrm{nc}}$ has only noncharacteristic oscillations.  Thus, it has a (real) Fourier series of the form
\begin{align}\label{f2}
f^*_{\mathrm{nc}}(u^l_\rho,W^l_{r,\mu,\rho},W^l_{i,\mu,\rho})(x,y,t)=\sum_{\alpha\in \mathbb{Z}^{2,*}}f_\alpha(x,y,t)e^{i\alpha \phi/\eps},
\end{align}
where $\alpha\phi:=\alpha_r\phi_r+\alpha_i\phi_i$ and 
\begin{align*}
\mathbb{Z}^{2,*}:=\{\alpha=(\alpha_r,\alpha_i)\in\mathbb{Z}^2 \ | \ \alpha_r\neq 0, \alpha_i\neq 0\}.
\end{align*}

Given $\mu>0$ and $\rho=(\rho_0,\rho_1,\rho_2)$,  we can truncate the series \eqref{f2}, preserving its reality,  and set
\begin{align}\label{f3}
f^{*,M}_{\mathrm{nc}}(u^l_\rho,W^l_{r,\mu,\rho},W^l_{i,\mu,\rho}):=\sum_{\alpha\in \mathbb{Z}^{2,*}, |\alpha|\leq M}f_{\alpha}(x,y,t)e^{i\alpha \phi/\eps},
\end{align}
where we choose $M=M(\mu,\rho)$ large enough so that\footnote{The functions in \eqref{f4} are evaluated at $(x,y,t,\theta_r,\theta_i)$, while the one in \eqref{f3} is evaluated at $(x,y,t)$.  }
\begin{align}\label{f4}
\|f^*_{\mathrm{nc}}(u^l_\rho,W^l_{r,\mu,\rho},W^l_{i,\mu,\rho})-f^{*,M}_{\mathrm{nc}}(u^l_\rho,W^l_{r,\mu,\rho},W^l_{i,\mu,\rho})\|_{L^2(\Omega_T\times \mathbb{T})} < \rho_1.
\end{align}
We construct $U^{M}_{\mathrm{nc}}$ in \eqref{f1} to have the form
\begin{align}\label{f4a}
U^{M}_{\mathrm{nc}}=\sum_{\alpha\in \mathbb{Z}^{2,*}, |\alpha|\leq M}U_{\alpha}(x,y,t)e^{i\alpha \phi/\eps},
\end{align}
where the coefficients $U_{\alpha}$ are chosen as follows.   Observe that 
\begin{align*}
P(x,y,t,\partial)(\eps^2U^{M}_{\mathrm{nc}})=\sum_{\alpha\in \mathbb{Z}^{2,*}, |\alpha|\leq M} \left(-p(x,y,t,d(\alpha\phi))U_{\alpha} \right)+O(\eps)\text{ in }L^2(\Omega_T).
\end{align*}
Thus, we can use $U^{M}_{\mathrm{nc}}$ to solve away $f^{*,M}_{\mathrm{nc}}$ if we set 
\begin{align}\label{f6}
U_{\alpha}:=-p^{-1}(x,y,t,d(\alpha\phi))f_{\alpha}\text{ for }\alpha\in \mathbb{Z}^{2,*}, |\alpha|\leq M.
\end{align}
To see that $U_\alpha$ is well-defined on $\Omega_T$, we use the fact that $f^*_{\mathrm{nc}}(u^l_\rho,W^l_{r,\mu,\rho},W^l_{i,\mu,\rho})$ has  $(x,y,t)$-support in a compact set $K\subset J_r\cap J_i$ strictly away from $\mathrm{SB}$; so  $p(x,y,t,d(\alpha\phi))$ is smooth and nonzero for all $(x,y,t)\in K$ and all $\alpha\in \mathbb{Z}^{2,*}$.
This completes the definition of $m^l_{\mu,\rho,M,\eps}$ in \eqref{f1}.\footnote{Observe that the series \eqref{f4a} is real since the series \eqref{f3} is real.}  

 With this choice of $U_\alpha$ we have
\begin{align}\label{f7a}
\begin{gathered}
P(x,y,t,\partial)(\eps^2U^{M}_{\mathrm{nc}})(x,y,t)=f^{*,M}_{\mathrm{nc}}(u^l_\rho,W^l_{r,\mu,\rho},W^l_{i,\mu,\rho})+O(\eps), \\
\text{where } \|f^*_{\mathrm{nc}}(u^l_\rho,W^l_{r,\mu,\rho},W^l_{i,\mu,\rho})-f^{*,M}_{\mathrm{nc}}(u^l_\rho,W^l_{r,\mu,\rho},W^l_{i,\mu,\rho})\|_{L^2(\Omega_T\times \mathbb{T})} < \rho_1.
\end{gathered}
\end{align}

The main step in the error analysis is the proof of the following lemma.

\begin{lem}\label{mainlem}
Let $u,W_r,W_i$ be the functions constructed in Proposition \ref{d13w}.   There exists $T>0$ such that the following statements hold.    For any sequence of positive numbers $\delta_l\to 0$ there exist   sequences of positive numbers $\mu_l$, $\rho_{0,l}$, $\rho_{1,l}$, $\rho_{2,l}$, and $\eps_l$ such that the exact solution $u^\eps$ of \eqref{e0a} satisfies\footnote{Here $\rho_l:=(\rho_{0,l}, \rho_{1,l},\rho_{2,l}).$}
\begin{subequations}\label{f1z}
    \begin{align}
        \|W_k-W^l_{k,\mu_l,\rho_l}\|_{L^2(\Omega_T\times\mathbb{T})}\leq \delta_l\text{ for }k=r,i; \label{f1za}
    \end{align}
and for all  $\eps\in (0,\eps_l]$, 
    \begin{align}
        \left\|u^\eps-\left(u(x,y,t)+\eps U^l_{r,\mu_l,\rho_l}\left(x,y,t,\phi_r/\eps\right)+\eps U^l_{i,\mu_l,\rho_l}\left(x,y,t,\phi_i/\eps\right)\right)\right\|_{H^1(\Omega_T)}\lesssim \delta_l. \label{f1zb}
    \end{align}
\end{subequations}
\end{lem}

The first result \eqref{f1za}  is immediate from the estimate \eqref{f1y} and the TR estimates \eqref{e11z}, \eqref{g4}, \eqref{g11}. 
The second result \eqref{f1zb} is proved in \S\S \ref{erest}--\ref{conclusion}.

\subsection{Estimate of the error term \texorpdfstring{$d^l_{\mu,\rho,M,\eps}=u^\eps-m^l_{\mu,\rho,M,\eps}$}{TEXT}}\label{erest}

The problem satisfied by 
$$d^l_{\mu,\rho,M,\eps}(x,y,t):=u^\eps(x,y,t)-m^l_{\mu,\rho,M,\eps}(x,y,t)$$ 
is\footnote{Here use the fact that $U^l_{r,\mu,\rho}$ and $U^M_{\mathrm{nc}}$ vanish outside $J_r$ and hence in  $\Omega_{[-T,-T+\delta]}$; also $u^l=u^1$ on that set.}
\begin{align}\label{fa7}
    \begin{cases}
        Pd^l_{\mu,\rho,M,\eps}=f(u^\eps)-Pm^l_{\mu,\rho,M,\eps} & \text{in }\Omega_T, \\
        \begin{aligned}
            d^l_{\mu,\rho,M,\eps}(0,y,t)
            = & -m^l_{\mu,\rho,M,\eps}(0,y,t) \\
            = & -\left[\eps U^l_{r,\mu,\rho}+\eps U^l_{i,\mu,\rho}+\eps^2 U^M_{\mathrm{nc}}\right]|_{x=0},
        \end{aligned} & \text{on } \Omega_T\cap\{x=0\},\\
        \begin{aligned}
            d^l_{\mu,\rho,M,\eps}
            = & u^\eps-\left(u^1_\rho+\eps U^l_{r,\mu,\rho}+\eps U^l_{i,\mu,\rho}+\eps^2 U^M_{\mathrm{nc}}\right) \\
            = &\left(u^\eps-(u^1+\eps U^l_1)\right)+\left[(u^1+\eps U^l_1)-(u^1_\rho+\eps U^l_{i,\mu,\rho})\right]
        \end{aligned} & \text{on } \Omega_{[-T,-T+\delta]}.
    \end{cases}
\end{align}

Next write 
\begin{align}\label{f7}
Pd^l_{\mu,\rho,M,\eps}=[f(u^\eps)-f(m^l_{\mu,\rho,M,\eps})]-[Pm^l_{\mu,\rho,M,\eps}-f(m^l_{\mu,\rho,M,\eps})]:=A+B.
\end{align}
When estimating  $d^l_{\mu,\rho,M,\eps}$ using \eqref{kreiss}, the term $A$ can be absorbed into the left side by taking $T$ small enough.  
We decompose $B$ as follows.

First choose $c(\mu)>0$ small enough so that the support of $1- \chi^r_{c(\mu)}$ is disjoint from the union of the supports of $\chi^r_\mu$ and $\chi^i_\mu$, and so that $\lim_{\mu\to 0}c(\mu)=0$.  Then write
\begin{align}\label{f8}
\begin{split}
&Pm^l_{\mu,\rho,M,\eps}-f(m^l_{\mu,\rho,M,\eps})\\
& = (1- \chi^r_{c(\mu)})[Pm^l_{\mu,\rho,M,\eps}-f(m^l_{\mu,\rho,M,\eps})]+ \chi^r_{c(\mu)}[Pm^l_{\mu,\rho,M,\eps}-f(m^l_{\mu,\rho,M,\eps})]\\
&:= B_1(l,\mu,\rho,M,\eps)+B_2(l,\mu,\rho,M,\eps).
\end{split}
\end{align}
Here $B_2$ is supported away from $\mathrm{SB}_+$.  The functions $W^l_{r,\mu,\rho}(x,y,t,\frac{\phi_r}{\eps})$, $W^l_{i,\mu,\rho}(x,y,t,\frac{\phi_i}{\eps})$     and  $(P_1\phi_r)W^l_{r,\mu,\rho}$,  $(P_1\phi_i)W^l_{i,\mu,\rho}$ are all $C^\infty$ on $\Omega_{T}$.  
To make $B_2$ small, we will use the profile equations.   To make $B_1$ small,  we use the profile equations to show it supported in a small neighborhood of $\mathrm{SB}_+$,  call it $\cJ_\mu$, whose measure satisfies $|\cJ_\mu|=o_\mu(1)$.

The next two lemmas treat 
$B_1$.

\begin{lem}\label{f9}
For $l$, $\rho$, $\mu$ fixed we have
\begin{align}\label{f10}
\begin{split}
&\limsup_{\eps\to 0}\;\left \|(1- \chi^r_{c(\mu)})[Pm^l_{\mu,\rho,M,\eps}-f(m^l_{\mu,\rho,M,\eps})]\right\|_{L^2(\Omega_{T})} \\
&\leq  \left\|(1-\chi^r_{c(\mu)})[Pu^l(x,y,t)-f(x,y,t,u^l,\nabla u^l)]\right\|_{L^2(\Omega_{T})}.
\end{split}
\end{align}
\end{lem}

\begin{proof}
Using \eqref{f1a} and the disjointness of supports described above, 
we have 
\begin{align*}
\begin{split}
&(1-\chi^r_{c(\mu)})f(m^l_{\mu,\rho,M,\eps})=(1- \chi^r_{c(\mu)})f(x,y,t,u^l,\nabla u^l)+o_\eps(1)\text{ in }L^2(\Omega_{T}).
\end{split}
\end{align*}
Along with a similar analysis of  $(1-\chi^r_{c(\mu)})Pm^l_{\mu,\rho,M,\eps}$ using the computation \eqref{b6b}, this  gives \eqref{f10}.
\end{proof}

\begin{lem}\label{f13}
We have 
$$\left\|(1-\chi^r_{c(\mu)})[Pu^l(x,y,t)-f(x,y,t,u^l,\nabla u^l)]\right\|_{L^2(\Omega_{T})}=o_\mu(1).$$
\end{lem}

An argument similar to the following proof occurs in \cite[\S 9]{cheverry1996}.
\begin{proof}
Let $\cJ:=J_r\cup J_i$.   Since both $W^l_r$ and $W^l_i$ are zero on $\Omega_T\setminus \cJ$, we have 
\begin{align*}
\underline{f}(u^l,W^l_r,W^l_i)=f(x,y,t,u^l,\nabla u^l)\text{ on }\Omega_T\setminus \cJ.
\end{align*}
Thus, the profile equations satisfied by $(u^l,W^l_r,W^l_i)$ imply
\begin{align*}
0=Pu^l-\underline{f}(u^l,W^l_r,W^l_i)=Pu^l-f(x,y,t,u^l,\nabla u^l)\text{ on } \Omega_T\setminus \cJ.
\end{align*}
Hence $(1-\chi^r_{c(\mu)})[Pu^l-f(x,y,t,u^l,\nabla u^l)]$ is supported in a small neighborhood of $\mathrm{SB}_+$,  call it $\cJ_\mu$, whose measure satisfies $|\cJ_\mu|=o_\mu(1)$.  
This implies the lemma since both $Pu^l$ and $f(x,y,t,u^l,\nabla u^l)$ are in  $L^2(\Omega_{T})$.\footnote{Use the profile equations to see that $Pu^l\in L^2(\Omega_{T})$.}
\end{proof}

Next we estimate $B_2$ in \eqref{f8}.
Using the fact that  formal computations like those in  \S \ref{formalb} are valid when $u^\eps_a$ is replaced by $m^l_{\mu,\rho,M,\eps}$, with  \eqref{h2} and \eqref{f7a}   we compute
\begin{align}\label{f16}
\begin{split}
Pm^l_{\mu,\rho,M,\eps}= & P(u^l_\rho+\eps U^l_{r,\mu,\rho}+\eps U^l_{i,\mu,\rho}+\eps^2 U^{M}_{\mathrm{nc}})\\
= & \underline{f}(u^l,W^l_r,W^l_i)_{\rho_0}+\left[T_{\phi_r} W^l_{r,\mu,\rho}+(P_1\phi_r)W^l_{r,\mu,\rho}\right] \\
& +\left[T_{\phi_i} W^l_{i,\mu,\rho}+(P_1\phi_i)W^l_{i,\mu,\rho}\right]+ f^{*,M}_{\mathrm{nc}}(u^l_\rho,W^l_{r,\mu,\rho},W^l_{i,\mu,\rho})+ o_\eps(1).
\end{split}
\end{align}
Recall from \eqref{f1a} and \eqref{f1aa} that 
\begin{align*}
\begin{split}
f(m^l_{\mu,\rho,M,\eps})
= & \underline{f}(u^l_\rho,W^l_{r,\mu,\rho},W^l_{i,\mu,\rho})+f^*_r(u^l_\rho,W^l_{r,\mu,\rho},W^l_{i,\mu,\rho}) \\
& +f^*_i(u^l_\rho,W^l_{r,\mu,\rho},W^l_{i,\mu,\rho})+ f^*_{\mathrm{nc}}(u^l,W^l_{r,\mu,\rho},W^l_{i,\mu,\rho})+o_\eps(1).  
\end{split}
\end{align*}
Thus, with \eqref{f16} we obtain\footnote{In \eqref{f17} ${f}^*_k(u^l,W^l_r,W^l_i)_\mu:=\chi^k_\mu {f}^*_k(u^l,W^l_r,W^l_i) $, $k=r,i$.}
\begin{align}\label{f17}
    \begin{split}
        & \chi^r_{c(\mu)}\left[Pm^l_{\mu,\rho,M,\eps}-f(m^l_{\mu,\rho,M,\eps})\right](x,y,t) \\
        & = \chi^r_{c(\mu)}\left[\underline{f}(u^l,W^l_r,W^l_i)_{\rho_0}-\underline{f}(u^l_\rho,W^l_{r,\mu,\rho},W^l_{i,\mu,\rho})\right]\\
        & \quad + \chi^r_{c(\mu)}\left[\left(T_{\phi_r} W^l_{r,\mu,\rho}(P_1\phi_r)W^l_{r,\mu,\rho}\right)-{f}^*_r(u^l,W^l_r,W^l_i)_\mu\right] \\
        & \quad + \chi^r_{c(\mu)}\left[{f}^*_r(u^l,W^l_r,W^l_i)_\mu-{f}^*_r(u^l_\rho,W^l_{r,\mu,\rho},W^l_{i,\mu,\rho})\right] \\
        & \quad + \chi^r_{c(\mu)}\left[\left(T_{\phi_i} W^l_{i,\mu,\rho}+(P_1\phi_i)W^l_{i,\mu,\rho}\right)-{f}^*_i(u^l,W^l_r,W^l_i)_\mu\right]  \\
        & \quad + \chi^r_{c(\mu)}\left[{f}^*_i(u^l,W^l_r,W^l_i)_\mu-{f}^*_i(u^l_\rho,W^l_{r,\mu,\rho},W^l_{i,\mu,\rho})\right] \\
        & \quad + \chi^r_{c(\mu)}\left[f^{*,M}_{\mathrm{nc}}(u^l_\rho,W^l_{r,\mu,\rho},W^l_{i,\mu,\rho})-f^*_{\mathrm{nc}}(u^l_\rho,W^l_{r,\mu,\rho},W^l_{i,\mu,\rho})\right]+o_\eps(1).
    \end{split}
\end{align}
We expect each of the  differences appearing in \eqref{f17} to be ``small" in $L^2(\Omega_{T})$.

\Remark 
More precisely, given $\delta>0$, we expect that if $\mu$ is first fixed small enough, then $\rho_0=\rho_0(\mu)$ can be fixed small enough, then $\rho_1=\rho_1(\mu,\rho_0)$ can be fixed small enough, then $\rho_2=\rho_2(\mu,\rho_0,\rho_1)$ can be fixed small enough, then $M=M(\mu,\rho)$ can be fixed large enough, and finally $\epsilon_0=\eps_0(\mu,\rho,M)$ can be fixed small enough, so that for $0<\eps<\epsilon_0$,  each of the differences in \eqref{f17} is less than $\delta$ in $L^2(\Omega_{T})$.   If $h$ denotes any one of those differences, this can be expressed more briefly by\footnote{In fact, $\rho_0$ does not really depend on $\mu$.}
\begin{align}\label{f17a}
\varlimsup_{\mu\to 0}\left(\varlimsup_{\rho_0\to 0}\left(\varlimsup_{\rho_1\to 0}\left(\varlimsup_{\rho_2\to 0}\left(\varlimsup_{M\to \infty}\left(\varlimsup_{\eps\to 0}\|h(\mu,\rho,M,\eps)(x,y,t)\|_{L^2(\Omega_{T})}\right)\right)\right)\right)\right)=0.
\end{align}
This \emph{order} of fixing $\mu,\rho_0,\rho_1,\rho_2,M,\eps$ is implicit in the notation $o_\eps(1)$ used, for example, in \eqref{f16}.  There $o_\eps(1)$ denotes a function $r(\mu,\rho,M,\eps)$ such that for $\mu, \rho,M$ fixed we have $$\lim_{\eps\to 0}\|r(\mu,\rho,M,\eps)\|_{L^2(\Omega_{T})}=0.$$

\begin{prop}\label{f19}
The function $h$ given by $\chi^r_{c(\mu)}\left[Pm^l_{\mu,\rho,M,\eps}-f(m^l_{\mu,\rho,M,\eps})\right](x,y,t)$ satisfies \eqref{f17a}.
\end{prop}

\begin{proof}
\textbf{1. } We show that each of the  six differences appearing in \eqref{f17} satisfies \eqref{f17a}.   By \eqref{f7a} and Lemma \ref{jmr2} we have immediately
\begin{align*}
\begin{split}
&\varlimsup_{\eps\to 0}\left\| \chi^r_{c(\mu)}\left[f^{*,M}_{\mathrm{nc}}(u^l_\rho,W^l_{r,\mu,\rho},W^l_{i,\mu,\rho})-f^*_{\mathrm{nc}}(u^l_\rho,W^l_{r,\mu,\rho},W^l_{i,\mu,\rho})\right]\right\|_{L^2(\Omega_T)}  \\
&\lesssim \left\|\chi^r_{c(\mu)}\left[f^{*,M}_{\mathrm{nc}}(u^l_\rho,W^l_{r,\mu,\rho},W^l_{i,\mu,\rho})-f^*_{\mathrm{nc}}(u^l_\rho,W^l_{r,\mu,\rho},W^l_{i,\mu,\rho})\right]\right\|_{L^2(\Omega_T\times \mathbb{T})}=o_{\rho_1}(1).
\end{split}
\end{align*}

\textbf{2. }We have 
\begin{align*}
\begin{split}
&\underline{f}(u^l,W^l_r,W^l_i)_{\rho_0}-\underline{f}(u^l_\rho,W^l_{r,\mu,\rho},W^l_{i,\mu,\rho})\\
&=\left[\underline{f}(u^l,W^l_r,W^l_i)_{\rho_0}-\underline{f}(u^l,W^l_r,W^l_i)\right]+\left[\underline{f}(u^l,W^l_r,W^l_i)-\underline{f}(u^l_\rho,W^l_{r,\mu,\rho},W^l_{i,\mu,\rho})\right].
\end{split}
\end{align*}
The first term on the right is $o_{\rho_0}(1)$, and the sublinearity assumption on $f$ implies
\begin{align*}
\begin{split}
&\varlimsup_{\eps\to 0}\left\|\chi^r_{c(\mu)}\left[\underline{f}(u^l,W^l_r,W^l_i)-\underline{f}(u^l_\rho,W^l_{r,\mu,\rho},W^l_{i,\mu,\rho})\right]\right\|_{L^2(\Omega_T)} \\
&\lesssim \|u^l-u^l_\rho\|_{H^1(\Omega_T)}+\|(W^l_r-W^l_{r,\mu,\rho},W^l_i-W^l_{i,\mu,\rho})\|_{ L^2(\Omega_T)\times L^2(\Omega_T)} \\
& =o_{\rho_0}(1)+o_{\rho_2}(1)+
o_{\rho_1}(1)+o_{\mu}(1).
\end{split}
\end{align*}
Here we use \eqref{h3} to get the  $o_{\rho_0}(1)$ term.  For the remaining terms we used Lemma \ref{jmr1a} followed by \eqref{g11}, \eqref{g4}, and \eqref{e11z}. 

\textbf{3. }Recall from \eqref{g12} that for $k=r,i$:
\begin{align*}
\begin{split}
\left\|\left(T_{\phi_k}W_{k,\mu,\rho}+(P_1\phi_k)W_{k,\mu,\rho}\right)-f^*_k(u,W_r,W_i)\right\|_{L^2(\Omega_T\times \mathbb{T})}
=o_\mu(1)+o_{\rho_1}(1)+o_{\rho_2}(1).
\end{split}
\end{align*}
Thus, Lemma \ref{jmr1a} implies
\begin{align*}
\begin{split}
\varlimsup_{\eps\to 0}\left\| \chi^r_{c(\mu)}\left[\left(T_{\phi_k} W^l_{k,\mu,\rho}+(P_1\phi_k)W^l_{k,\mu,\rho}\right)-{f}^*_k(u^l,W^l_r,W^l_i)_\mu\right]\right\|_{L^2(\Omega_T)} =o_\mu(1)+o_{\rho_1}(1)+o_{\rho_2}(1)
\end{split}
\end{align*}

\textbf{4. }Similarly, applying Lemma \ref{jmr1a} and using the sublinearity of $f$ as in step \textbf{2}  yields for $k=r,i$:
\begin{align*}
\begin{split}
\varlimsup_{\eps\to 0}\left\|\chi^r_{c(\mu)}\left[{f}^*_k(u^l,W^l_r,W^l_i)_\mu-{f}^*_k(u^l_\rho,W^l_{r,\mu,\rho},W^l_{i,\mu,\rho}\right]\right\|_{L^2(\Omega_T)}
=o_\mu(1)+o_{\rho_1}(1)+o_{\rho_2}(1).
\end{split}
\end{align*}
This completes the proof.
\end{proof}

Next we consider the boundary term and the initial data term in the  application of the Kreiss estimate \eqref{kreiss} to the problem \eqref{fa7} satisfied by $d^l_{\mu,\rho,M,\eps}(x,y,t)$.

\begin{prop}\label{f21}
Let  $\delta_l\to 0$ be as in Definition \ref{meaning} {as applied to} the symbol $\sim_{H^1}$ in \eqref{e0ac}.
We have
\begin{subequations}\label{f22}
    \begin{gather}
        \left\langle d^l_{\mu,\rho,M,\eps}\right\rangle_{H^1(b\Omega_T)}=o_\eps(1); \text{ and } \label{f22a}\\
        \varlimsup_{\eps\to 0}\|d^l_{\mu,\rho,M,\eps}\|_{H^1(\Omega_{[-T,-T+\delta]})}\lesssim \delta_l+ o_\mu(1)+o_{\rho_0}(1)+o_{\rho_1}(1)+o_{\rho_2}(1). \label{f22b}
    \end{gather}
\end{subequations}
\end{prop}

\begin{proof}
\textbf{For \eqref{f22a}: }by \eqref{fa7} and \eqref{g12} we have $d^l_{\mu,\rho,M,\eps}(0,y,t)=-\eps^2 U^M_{\mathrm{nc}}|_{x=0}=o_\eps(1) \text{ in }H^1(b\Omega_T)$.  Indeed, \eqref{f3} and \eqref{f6} imply that each term is smooth in the finite sum  \eqref{f4a} that gives $U^M_{\mathrm{nc}}$.\footnote{Here we use the fact that $u^l_\rho$ and  $W^l_{k,\mu,\rho}$, $k=r,i$, are smooth. }

\textbf{For \eqref{f22b}: }by \eqref{fa7} we have
\begin{align*}
d^l_{\mu,\rho,M,\eps}|_{\Omega_{[-T,-T+\delta]}}=\left(u^\eps-(u^1+\eps U^l_1)\right)+\left[(u^1+\eps U^l_1)-(u^1_\rho+\eps U^l_{i,\mu,\rho})\right], 
\end{align*}
hence
\begin{align*}
\begin{split}
\|d^l_{\mu,\rho,M,\eps}\|_{H^1(\Omega_{[-T,-T+\delta]})}\lesssim  \delta_l+\|u^1-u^1_\rho\|_{H^1(\Omega_{[-T,-T+\delta]})} 
+\|\eps U^l_1-\eps U^l_{i,\mu,\rho}\|_{H^1(\Omega_{[-T,-T+\delta]})}
\end{split}
\end{align*}
The conclusion then follows by the choice of $u^1_{\rho_0}$ in \eqref{h2} and, after applying Lemma \ref{jmr1a},  from \eqref{e11z}, \eqref{g4}, \eqref{g11}.
\end{proof}

\subsection{Conclusion of the proof of Theorem \ref{mt2}}\label{conclusion}
Application of the Kreiss estimate \eqref{kreiss} to the error problem \eqref{fa7} yields, after absorption of the term involving $A$ in \eqref{f7}, the estimate
\begin{align*}
\begin{split}
\|d^l_{\mu,\rho,M,\eps}\|_{H^1(\Omega_{T})}
\lesssim & \sum^2_{k=1}\|B_k(l,\mu,\rho,M,\eps)\|_{L^2(\Omega_{T})} 
 +\langle d^l_{\mu,\rho,M,\eps} \rangle_{H^1(b\Omega_T)}+\|d^l_{\mu,\rho,M,\eps}\|_{H^1(\Omega_{[-T,-T+\delta]})},
\end{split}
\end{align*}
where the $B_k$ are defined in \eqref{f8}.    The term $B_1$ is estimated in Lemmas \ref{f9} and \ref{f13}, the term $B_2$ is estimated in Proposition \ref{f19}, and the remaining terms are  estimated in Proposition \ref{f21}. 
Together these estimates show that for the  sequence of numbers $\delta_l\to 0$ in Proposition  \ref{f21}, we have 
\begin{align}\label{f24}
\|d^l_{\mu,\rho,M,\eps}\|_{H^1(\Omega_{T,X})}\lesssim \delta_l+R(l,\mu,\rho,M,\eps),
\end{align}
where for each $l\in\mathbb{N}$ 
\begin{align}\label{f25}
\varlimsup_{\mu\to 0}\left(\varlimsup_{\rho_0\to 0}\left(\varlimsup_{\rho_1\to 0}\left(\varlimsup_{\rho_2\to 0}\left(\varlimsup_{M\to \infty}\left(\varlimsup_{\eps\to 0}\|R(l,\mu,\rho,M,\eps)\|_{L^2(\Omega_{T})}\right)\right)\right)\right)\right)=0.
\end{align}

\begin{proof}[Proof of Lemma \ref{mainlem}]
We proved \eqref{f1za} at the end of \S \ref{approx}.   To prove \eqref{f1zb}, 
for each $l$ we  use \eqref{f24} and \eqref{f25} to choose (or modify) consecutively 
$\mu_l$, $\rho_{0,l}$, $\rho_{1,l}$, $\rho_{2,l}$, $M_l$, and $\eps_l$ such that 
\begin{align*}
\text{for all } \eps\in (0,\eps_l], \ \|d^l_{\mu_l,\rho_l,M_l,\eps}\|_{H^1(\Omega_{T})}\lesssim \delta_l.
\end{align*}
Recalling the definition of $d^l_{\mu,\rho,M,\eps}$ and using 
\begin{align*}
\|u-u^l\|_{H^1(\Omega_{T})}\lesssim \delta_l,\ 
\|u^l-u^l_\rho\|_{H^1(\Omega_T)}=o_{\rho_0}(1), \text{ and }\left\|\eps^2 U^M_{\mathrm{nc}}\right\|_{H^1(\Omega_{T})}=o_\eps(1),
\end{align*}
we obtain \eqref{f1zb} after possibly another modification of $\rho_{0,l}$ and $\eps_l$.  
\end{proof}

To complete the proof of Theorem \ref{mt2}, one then just needs to replace the smooth functions $W^l_{k,\mu_l,\rho_l}$ in \eqref{f1z} by trigonometric polynomial approximations $W^l_{k,\mu_l,\rho_l,N_l}$ such that\footnote{This entails another application of Lemma \ref{jmr1a} and another possible reduction of $\eps_l$. } 
$$\left\|W^l_{k,\mu_l,\rho_l}-W^l_{k,\mu_l,\rho_lN_l}\right\|_{L^2(\Omega_T\times \mathbb{T})}\leq \delta_l. $$

\Remark
Since the profiles $W_r$, $W_i$ have support in $J_r\cup J_i$, Theorem \ref{mt2} implies
\begin{align*}
\|u^\eps - u\|_{H^1(\Omega_{T}\setminus (J_r\cup J_i))}=o_\eps(1).
\end{align*}
In particular,  there are no high frequency oscillations in the shadow that are  detectable in the $H^1$ norm.

\section{Diffraction of plane waves by a convex obstacle} \label{convexobstacle}

In this section we let $P(m,\partial_m)$ be the wave operator on $\mathbb{R}^{n+1}$, 
\begin{align}\label{r1}
\Box=\partial_{x_1}^2+\dots+\partial_{x_n}^2-\partial_t^2,
\end{align}
and  show that Theorem \ref{mt2} applies to describe the diffraction of oscillatory plane waves by a large class of convex obstacles  $\mathcal{O}\subset \mathbb{R}^{n}$ with $C^\infty$ boundary.   We take the spacetime domain to be $M=(\mathbb{R}^n\setminus \cO)\times \mathbb{R}_t$ and use coordinates $(x_1,\ox,t,\xi_1,\oxi,\tau)$ on $T^*M$.
Grazing rays of any finite or infinite order are allowed.   We must show that Assumptions \ref{A2} and \ref{A3} hold for these problems.

Denote points in $\mathbb{R}^n$ by $x=(x_1,\ox)$.
Our analysis is local near a given boundary point, so we make the following definition.

\begin{defi}\label{r2}
Let $\mathcal{O}\subset \mathbb{R}^n$ be an open convex set with $C^\infty$ boundary and suppose $\mathrm P_0\in \partial \cO$.
After rotation and translation of $\cO$ we can suppose $\mathrm P_0=(1,0)$, that the tangent plane to $\partial \cO$ at $\mathrm P_0$ is 
$x_1=1$, and that $\cO$ lies to the \emph{left} of $\mathrm P_0$ near $\mathrm P_0$.    We say that \emph{$\cO$ is strictly convex near $\mathrm P_0$} provided there exists an $\mathbb{R}^n$-open set $\Omega\ni \mathrm P_0$ such that $\partial \cO\cap \Omega$ is the graph $x_1=F(\ox)$ of a function $F(\ox)$ with the following properties. There exists an $\mathbb{R}^{n-1}$-open ball $B(0,r)$ of radius $r>0$ such that $F: B(0,r)\to \mathbb{R}$ and 
\begin{enumerate}[label={{ \arabic*.}}]
    \item $F \in C^\infty(B(0,r))$ and $F(0)=1$;
    \item For all $\ox$, $\ox^*\in B(0,r)$, we have $F(\ox^*)-F(\ox)\leq \langle\nabla F(\ox),\ox^*-\ox\rangle$ with equality holding if and only if $\ox=\ox^*$.
\end{enumerate}
Thus, we have
\begin{align*}
\partial \cO\cap \Omega=\{(F(\ox),\ox) \ | \ \ox\in B(0,r)\}.
\end{align*}
\end{defi}
The second condition in Definition \ref{r2} means that $F$ is strictly concave on $B(0,r)$.  
The conditions 1, 2 in Definition \ref{r2}  imply that  the Hessian of $F$ is negative semi-definite,  that is,  $\nabla^2 F\leq 0$ on $B(0,r)$.\footnote{ In fact, the conditions 1, 2 in Definition \ref{r2} imply $\nabla^2 F<0$ on $B(0,r)$, except possibly on a nowhere dense subset.  See \cite{robertsvarberg1973} for properties of  convex functions.}  Note also that $\nabla F(0)=0$.

\Remark 
If condition 1 in Definition \ref{r2}  holds along with $\nabla^2 F<0$ on $B(0,r)\setminus \{0\}$, then $\cO$ is strictly convex near $\mathrm P_0=(1,0)$.

\noindent 
{\bf Examples. }
For the following functions $F_j:\mathbb{R}^{n-1}\to \mathbb{R}$  the sets $\{(x_1,\ox) \ | \ x_1< F(\ox)\}$ are strictly convex near $(1,0)$:
\begin{subequations}\label{r4}
    \begin{align}
        & F_0(\ox) = 1-|\ox|^{2k} \text{ where } k\in \mathbb N;\\
        & F_1(\ox)=1-(x_2^{2k}+\dots+x_n^{2k}) \text{ where } k\in\mathbb{N};  \label{r4a}\\
        & F_2(\ox)=1-\begin{cases} e^{-|\ox|^{-2}}, \ & \ox\neq 0, \\ 0, \ & \ox=0. \end{cases}   \label{r4b}
    \end{align}
\end{subequations}
Here $F_2$, which vanishes to infinite order at $\ox=0$,   and $F_0$ satisfy $\nabla^2 F<0$ for $\ox\neq 0$ small.  The function $F_1$ does not.

\begin{figure}[t] 
   \centering
   \includegraphics[width=0.49\textwidth]{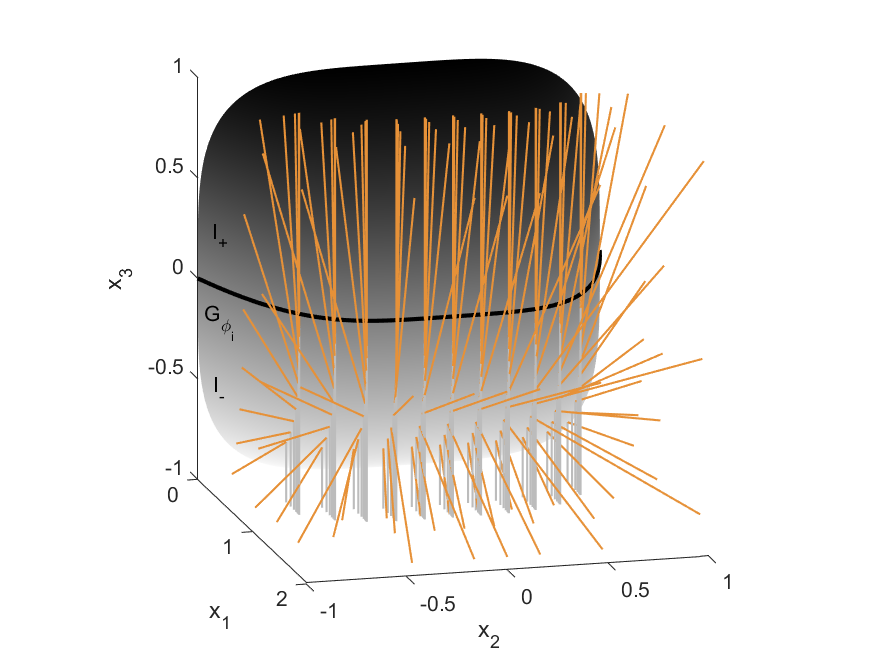} 
   \includegraphics[width=0.49\textwidth]{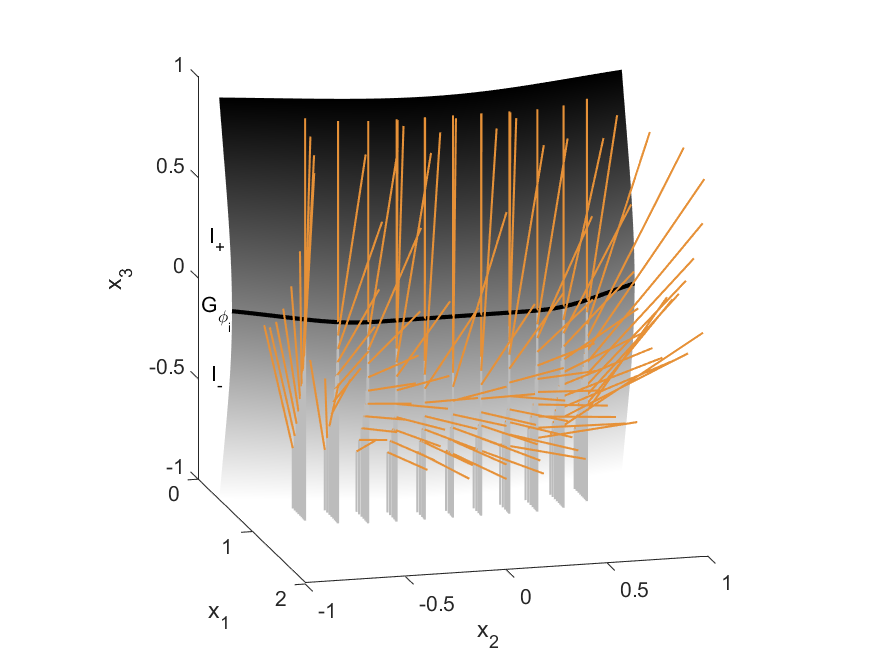} 
   \caption{{\bf Left:} Convex obstacle $\mathcal O_1$ with $F_1$ in \eqref{r4a} and $n=3$, $k=2$. {\bf Right:} Convex obstacle $\mathcal O_2$ with $F_2$ in \eqref{r4b} and $n=3$. In both figures, $I_+$, $G_{\phi_i}$, and $I_-\cup G_{\phi_i}$ are the $x$-projections of the shadow regions, the grazing sets and the illuminable regions respectively. The gray lines are the incoming rays and the yellow lines are the reflected rays.}
   \label{fig:rays}
\end{figure}

Suppose now that $\cO$ is strictly convex near $\mathrm P_0=(1,0)$.   Incoming plane waves correspond to linear incoming phases.  
A linear phase having a forward characteristic that grazes $\partial M$ at $(\mathrm P_0,t_0)=(1,0,t_0)$  must be some positive multiple of\footnote{The point $(1,0,t_0)$ is now playing the role of the distinguished basepoint $``0"\in \partial M$ of \S \ref{DandA}.}  
\begin{align}\label{r5}
\phi_i(x_1,\ox,t)=-t+\langle\otheta, \ox\rangle,  \text{ where }\otheta=(\theta_2,\dots,\theta_n)\in \mathbb S^{n-2}.
\end{align}

In \S \ref{va2} we verify Assumption \ref{A2} for oscillatory incoming plane waves for the following kinds of obstacles:
\begin{enumerate}[label={\arabic*.}]
    \item \emph{any} two-dimensional obstacle that is strictly convex near $\mathrm P_0=(1,0)$; see Proposition \ref{ra11}.

    \item any three dimensional obstacle that is strictly convex near $\mathrm P_0=(1,0)$, provided $F$ as in Definition \ref{r2} also satisfies Assumption \ref{A4}; see Proposition \ref{r11}.

    \item $n$ dimensional obstacles that are strictly convex near $\mathrm P_0=(1,0)$ and have an additional symmetry property -- Assumption \ref{ndA}; see Proposition \ref{rnd}.
\end{enumerate}

In \S\S \ref{2d}--\ref{nd} we show that for strictly convex obstacles, the reflected flow map $Z_r$ resulting from an incoming phase $\phi_i$ in \eqref{r5} satisfies Assumption \ref{A3}.

\subsection{Assumption \ref{A2}}\label{va2}
For an obstacle $\cO$ defined by a function $F$ as in Definition \ref{r2} and incoming phase $\phi_i=-t+\langle \otheta,\ox\rangle$ as in \eqref{r5} the grazing set determined by $\phi_i$, defined in Assumption \ref{A2}, is\footnote{Using the parametrization of $\partial M$ given by $(\ox,t)\mapsto (F(\ox),\ox,t)$, we can write $\phi_0=-t+\langle \otheta,\ox\rangle.$   Thus, $\usigma=(0,t_0,d\phi_0(0,t_0))=(0,t_0,\otheta,-1)=i^*\urho$,  where $\urho=(1,0,t_0,0,\otheta,-1)$.}
\begin{align}\label{r10}
G_{\phi_i}=\{(F(\ox),\ox,t) \ | \ \langle\nabla F(\ox),\otheta\rangle=0, \ \ox\in B(0,r), \ t\in \mathbb{R}\}.
\end{align}
Indeed, the normal vector to $\partial M$ at $(F(\ox),\ox,t)$ is $(1,-\nabla F(\ox),0)$ and the direction of a forward characteristic of $\phi_i$ at $(F(\ox),\ox,t)$ is $(0,\otheta,1)$.   Similarly, the illuminated region (Definition \ref{q10zz}) is $I_-\cup G_{\phi_i}$, where 
\begin{align*}
I_- = \{(F(\ox),\ox,t)\ | \ \langle\nabla F(\ox),\otheta\rangle>0, \ \ox\in B(0,r), \ t\in \mathbb{R}\}.
\end{align*}

\subsubsection{2D obstacles} \label{2do}   

We show now that  Assumption \ref{A2} holds for incoming plane waves when $\cO$ is 
\emph{any} two-dimensional obstacle that is strictly convex near $\mathrm P_0=(1,0)$.

\begin{prop}\label{ra11}
Suppose $\cO\subset \mathbb{R}^2$ is defined by a function $F$ as in Definition \ref{r2}; that is, assume only that $\cO$ is strictly convex near $\mathrm P_0=(1,0)$. 
Let $P=\Box$ be the wave operator \eqref{r1} on $M=(\mathbb{R}^2\setminus \cO)\times \mathbb{R}_t$ and let $\phi_i=-t+\langle\otheta,\ox\rangle$ where $\otheta=\pm 1$.   Assume 
$$\usigma=i^*\urho\in \cG_d:=\cup_{k=1}^\infty \left(G^{2k}_d\setminus G^{2k+1}\right)\cup G^\infty_d,$$ where $\urho=(1,0,t_0,0,\otheta,-1)$.    Then the conditions of Assumption \ref{A2} are satisfied if one takes $\zeta(\ox)=\ox=x_2$.   That is, we have
\begin{align}\label{ra12}
\begin{split}
&G_{\phi_i}=\{(F(x_2),x_2,t) \ | \ x_2=0,\ x_2\in B(0,r), t\in \mathbb{R}\}. 
\end{split}
\end{align}
Moreover, $H_p\zeta(\urho)\neq 0$   and points in $\left(G\cap \mathrm{Graph}(d\phi_0)\right)\setminus\{\usigma\}$ near $\usigma$ belong to $\cG_d$ and have the same order as $\usigma$.    
\end{prop}

\begin{proof}
\textbf{1. }The strict convexity assumption implies that the Taylor expansion of $F$ at $0$ must have the form
\begin{align*}
F(x_2)=1-(\beta_2 x_2^2+\beta_4x_2^4+\dots+\beta_{2k}x_2^{2k})+r(x_2), \text{ where }r(x_2)=O(|x_2|^{2k+1}),
\end{align*}
where the first nonzero coefficient $\beta_{2j}$, if there is one, must be positive.  A computation similar to  \eqref{r6a} shows that 
\begin{subequations}\label{raa13}
    \begin{align}
        \usigma\in G^{2k}_d\setminus G^{2k+1} & \Leftrightarrow \beta_{2j}=0 \text{ for } j=1,\dots,k-1 \text{ and } \beta_{2k}>0;    \label{raa13a}\\
        \usigma\in G^{\infty}_d & \Leftrightarrow \beta_{2j}=0=0 \text{ for all } j.     \label{raa13b}
    \end{align}
\end{subequations}
In case \eqref{raa13b}, $r(x_2)=O(|x_2|^\infty)$ and the condition (b) in Definition \ref{r2} implies $r'(x_2)$ is strictly increasing for $x_2\in B(0,r)$.
\footnote{See \cite[\S 11]{robertsvarberg1973}.}  
Both cases in \eqref{raa13} give $\usigma\in \mathcal G_d$.

From  \eqref{r10} we have
\begin{align}\label{ra14}
G_{\phi_i}=\{(F(x_2),x_2,t) \ | \ F'(x_2)=0, \ \ox\in B(0,r), \ t\in \mathbb{R}\}.
\end{align}
If \eqref{raa13a} holds, then $F'(x_2)=x_2^{2k-1}G(x_2)$ for some $C^\infty$ function $G$ such that $G(0)\neq 0$. If \eqref{raa13b} holds, then again $F'(x_2)=r'(x_2)=0 \Leftrightarrow x_2=0$.  With \eqref{ra14} this gives \eqref{ra12}.

\textbf{2. }We have $H_p=2\xi_1\partial_{x_1}+2\xi_2\partial_{x_2}-2\tau\partial_t$, so
$H_p x_2(\urho)=2\otheta\neq 0$.    Moreover, if $$\sigma\in\left(G\cap \mathrm{Graph}(d\phi_0)\right)\setminus\{\usigma\}$$ lies near 
$\usigma$, we must have $\sigma=i^*\rho$, where $\rho=(F(x_2),x_2,t_1,0,\otheta,-1)$ with $t_1$ near $t_0$ and $x_2$ near $0$.  If $x_2\neq 0$, then with $\beta=x_1-F(x_2)$ we have 
\begin{align}
H_p\beta(\rho)=-2\otheta F'(x_2)\neq 0,
\end{align}
so  $\sigma\notin G$.  If $x_2=0$, then $\sigma\in \cG_d$ has the same order as $\usigma$.
\end{proof}

\Remark
Let $P$ and $\phi_i$ be as in Proposition \ref{ra11} and consider $F(x_2)$ in the case where \eqref{raa13a} holds.
If we first change variables to flatten the boundary by defining 
\begin{align}
(x,z_1,z_2):=(x_1+\beta_{2k}x_2^{2k}-r(x_2)-1,x_2,t),
\end{align}
and then put $p$ into standard form via the second change of variables
\begin{align}
(x'_1,z_1',z_2')=(x,z_1+e_1(x,z_1),z_2),
\end{align}
where $e_1$ is chosen to remove the  ``mixed term"  in $p$ as in \eqref{m5b},
then  direct computation shows 
\begin{align}\label{ra14z}
\partial_{x'}\phi_i(0,z_1',z_2')=z_1'^{2k-1}v(z_1').
\end{align}
Here $v$ is $C^\infty$ and $v(0)\neq 0$.   Thus, we can't expect to use $\partial_{x'}\phi_i(0,z')$ as a smooth coordinate function 
when $k>1$.

\subsubsection{3D obstacles}

In this section, we show that Assumption \ref{A2} is satisfied for incoming plane waves by any three-dimensional obstacle that is strictly convex near $\mathrm P_0=(1,0)$, provided $F$ as in Definition \ref{r2} also satisfies the next assumption.   

\begin{ass}\label{A4}
Let  $\mathcal{O}\subset \mathbb{R}^3$ be an obstacle that is strictly convex near $\mathrm P_0=(1,0)$,  and which is defined by a function  $F$ as in Definition \ref{r2} that satisfies the following additional condition for some $k\in\mathbb{N}:$\footnote{Condition \eqref{r6bb} itself implies that $\cO$ is strictly convex near $P$.} 
\begin{subequations}\label{r6}
    \begin{align}
        F(\ox)= & 1+\sum_{|\alpha|=2k}\frac{\partial^\alpha F(0)}{\alpha !}\ox^\alpha +O(|\ox|^{2k+1}), \nonumber \\ 
        & \text{where }\sum_{|\alpha|=2k}\frac{\partial^\alpha F(0)}{\alpha !}\ox^\alpha<0 \text{ for }\ox\neq 0; \text{ and } \label{r6aa}\\
        \nabla^2 F_{2k}< & 0 \text{ for }\ox\neq 0, \text{ where }F_{2k}:=1+\sum_{|\alpha|=2k}\frac{\partial^\alpha F(0)}{\alpha !}\ox^\alpha.  \label{r6bb}
    \end{align}
\end{subequations}
\end{ass}

In the proof of Proposition \ref{ra11} we saw that the analogue of Assumption \ref{A4} for $\cO\subset \mathbb{R}^2$ holds automatically when $\cO$ is strictly convex near $\mathrm P_0$ and $\usigma\in G^{2k}_d\setminus G^{2k+1}$.   This is no longer true for obstacles $\cO\subset \mathbb{R}^n$ for $n>2$.   A $C^\infty$ function of the form
\begin{align*}
\begin{split}
F(\ox)=1+h_2(\ox)+h_4(\ox)+\dots +h_{2k-2}(\ox)+ h_{2k}(\ox)+O(|\ox|^{2k+1}), 
\end{split}
\end{align*}
where each function $h_{2j}$ is a homogeneous polynomial in $\ox$ of degree $2j$ and 
\begin{subequations}\label{r6y}
    \begin{align}
        & h_{2j}\leq 0,\; \nabla^2 h_{2j}\leq 0, \;h_{2j}(\otheta)=0\; \text{ for }j=1,\dots,k-1, \text{ but } \label{r6ya}\\
        & h_{2k}<0\text{ and }\nabla^2 h_{2k}<0\text{ for }\ox\neq 0, \label{r6yb}
    \end{align}
\end{subequations}
defines an obstacle $\cO$ that is strictly convex near $\mathrm P_0$ and for which $\usigma\in G^{2k}_d\setminus G^{2k+1}$; see the computation \eqref{r6a}.   Below the proof of Proposition \ref{r11}, we remark an extension of Proposition \ref{r11} to certain functions of this type.

The condition \eqref{r6} implies that for \emph{every} $\otheta\in \mathbb S^2$, the point 
$\usigma=i^*\urho$, where $\urho=(1,0,t_0,0,\otheta,-1)$,  lies in $G^{2k}_d\setminus G^{2k+1}$.
To see this we check that the conditions \eqref{q3}  hold with $\beta(y,t):=x_1-F(\ox)$.   
The forward null bicharacteristic associated to $\phi_i$ such that $\gamma(0)=\urho=(1,0,t_0,0,\otheta,-1)$ is
\begin{align*}
\gamma(s)=(1,2s\otheta,t_0+2s,0,\otheta,-1),
\end{align*}
We have
\begin{align}\label{r6a}
\begin{gathered}
\beta(\gamma(s))=1-F(2s\otheta)=(2s)^{2k}\left(-\sum_{|\alpha|=2k}\frac{\partial^\alpha F(0)}{\alpha !}\otheta^\alpha\right)+O(s^{2k+1}), \\
H_p^j\beta(\urho)=\left. \left(\frac{d}{ds}\right)^j \right|_{s=0}\;\beta(\gamma(s))\text{ for all }j,
\end{gathered}
\end{align}
which implies that the conditions \eqref{q3} hold.

\Remarks 
1.  A  computation like \eqref{r6a} shows that for  $F$ as in Example \eqref{r4a} we have $\usigma\in G^{2k}_d\setminus G^{2k+1}$, while for $F$ as in Example \eqref{r4b} we have $\usigma\in G^\infty_d$.

\noindent
2. The following  $C^\infty$ functions $F_j:\mathbb{R}^2\to \mathbb{R}$ satisfy Assumption \ref{A4}:
\begin{align}\label{r8}
\begin{split}
&F_3(\ox)=1-(x_2^{4}+x_2^2x_3^2+x_3^{4})+r(\ox), \text{ where }r(\ox)=O(|\ox|^{5});\\
&F_4(\ox)=1-(x_2^{4}+x_2^2x_3^2+x_3^{4}-x_2x_3^3)+r(\ox), \text{ where }r(\ox)=O(|\ox|^{5});\\
&F_5(\ox)=1-(x_2^{6}+x_2^2x_3^4+x_2^4x_3^2+x_3^{6})+r(\ox), \text{ where }r(\ox)=O(|\ox|^{7}).
\end{split}
\end{align}

\noindent
3. The function $F(\ox)=1-(x_2^{6}+x_2^3x_3^3+x_3^{6})$ satisfies \eqref{r6aa} but fails to satisfy even $\nabla^2 F\leq 0$.

\begin{prop}\label{r11}
Let $\mathcal{O}\subset \mathbb{R}^3$ be an obstacle defined by $F$ as in Assumption \ref{A4}.  
Let $P=\Box$ be the wave operator \eqref{r1} on $M=(\mathbb{R}^3\setminus \cO)\times \mathbb{R}_t$ and let $\phi_i=-t+\langle\otheta,\ox\rangle$ where $\otheta=(\theta_2,\theta_3)\in \mathbb S^{1}$.   Assume $\usigma=i^*\urho\in G^{2k}_d\setminus G^{2k+1}$, $k\in \mathbb{N}$, where $\urho=(1,0,t_0,0,\otheta,-1)$.    Then the conditions of Assumption \ref{A2} are satisfied: there is a function $\zeta$ such that 
\begin{align*}
\begin{gathered}
\zeta\in C^1(B(0,r)),\ \zeta\in C^\infty(B(0,r)\setminus 0), \\
G_{\phi_i}=\{(F(\ox),\ox,t) \ | \ \zeta(\ox)=0,\ \ox\in B(0,r), \ t\in \mathbb{R}\}.
\end{gathered}
\end{align*}
Moreover, $H_p\zeta(\urho)\neq 0$   and every point in $\left(G\cap \mathrm{Graph}(d\phi_0)\right)\setminus\{\usigma\}$ near $\usigma$ lies in $\cG_d$.    When $k=1$, $\zeta$ can be found $C^\infty(B(0,r))$.

\end{prop}

\begin{proof}
\textbf{1. }Write $F=F_{2k}+r$, where 
\begin{subequations}\label{r13a}
    \begin{align}
        & F_{2k}(\ox) =1+\sum_{|\alpha|=2k}\frac{\partial^\alpha F(0)}{\alpha !}\ox^\alpha<0 \text{ for }\ox\neq 0,\  r(\ox)=O(|\ox|^{2k+1}), \label{r13aa}\\
        & \nabla^2 F_{2k}< 0 \text{ for }\ox\neq 0. \label{r13ab}
    \end{align}
\end{subequations}
 With \eqref{r10} in mind, we define grazing functions
\begin{align*}
g_{\otheta}(\ox):=\langle \nabla F(\ox),\otheta\rangle\text{ and }g_{2k,\otheta}(\ox):=\langle \nabla F_{2k}(\ox),\otheta\rangle
\end{align*}
and observe that 
\begin{subequations}
    \begin{align*}
        \nabla g_{\otheta}(0)=0, \nabla g_{2k,\otheta}(0)=0, \
        \nabla g_{2k,\otheta}(\ox)=\nabla^2 f_{2k}(\ox)\otheta\neq 0 \text{ for }\ox\neq 0.
    \end{align*}
\end{subequations}

\textbf{2. }The function $g_{2k,\otheta}$ is a homogeneous polynomial in $\ox=(x_2,x_3)$ of degree $2k-1$.  The homogeneity implies that the real zero set of $g_{2k,\otheta}$ is a union of at most $2k-1$ lines through the origin.     We claim  that \eqref{r13a}  implies there is only one line.   To see this fix $\eps>0$ small and define the level curve
\begin{align*}
C_\eps:=\{\ox \ | \ 1-F_{2k}(\ox)=\eps\}.
\end{align*}
This is a compact strictly convex $C^\infty$ curve enclosing $0$ with positive curvature at all points.\footnote{Compactness follows from 
$1-F_{2k}(\ox)\geq C|\ox|^{2k}$, and the other properties follow from $\nabla^2(1-F_{2k})>0$.}     Now $g_{2k,\otheta}(\ox)=0\Leftrightarrow \nabla F_{2k}(\ox)= a\otheta^\perp$ for some $a\neq 0$,  and the positive curvature of $C_\eps$ implies this can happen only at two points of $C_\eps$. Thus, the zero set of $g_{2k,\otheta}$ must consist of just one line,
whose equation we can write as\footnote{For $F_3$ in \eqref{r8} and $\otheta=(\frac{1}{\sqrt{2}},\frac{1}{\sqrt{2}})$, that line is $x_2+x_3=0.$  For $F_4$ in \eqref{r8} and $\otheta=(1,0)$, the line is $x_3-cx_2=0$, for some $c\in (\frac{5}{2},3)$. }
\begin{align*}
x_3=0, \text{ or }x_2-cx_3=0 \text{ for some }c\in\mathbb{R}. 
\end{align*}
Below we consider the second case; the first is treated similarly.

\textbf{3. } We have
\begin{align}\label{r14a}
g_{\otheta}(\ox)=g_{2k,\otheta}(\ox)+\langle \nabla r(\ox),\otheta \rangle
\end{align}
as well as the factorization 
\begin{align}\label{r14b}
g_{2k,\otheta}(\ox)=(x_2-cx_3)G(\ox),
\end{align}
where $G$ is a real homogeneous polynomial of degree $2k-2$ that is nonvanishing \emph{off} the line $x_2-cx_3=0$.   
Next we show that $G$ is nonvanishing on that line as well, except at $\ox=0$.

\textbf{4. } For any $\ox$ we compute
\begin{align}\label{r14c}
 \langle \nabla g_{2k,\otheta}(\ox),\otheta\rangle=\langle (1,-c),\otheta\rangle G(\ox)+(x_2-cx_3)\langle \nabla G(\ox),\otheta\rangle.
\end{align}
The left side of \eqref{r14c} is $\langle \nabla^2 F_{2k}(\ox)\otheta,\otheta\rangle <0$ for $\ox\neq 0$, so after evaluating \eqref{r14c} at $x_2=cx_3$, we conclude both 
\begin{align}\label{r15}
\langle(1,-c),\otheta\rangle \neq 0 \text{ and }G(\ox)\neq 0 \text{ for }x_2=cx_3\neq 0. 
\end{align}
Thus, $G$ has a fixed sign for $\ox\neq 0$, which we may take as positive.   This implies
\begin{align}\label{r16}
\text{there exists } C>0\text{ such that }G(x)\geq C|\ox|^{2k-2}.
\end{align}

\textbf{5. }Recalling \eqref{r14a} and \eqref{r14b}, we see that 
\begin{align}\label{r17}
g_{\otheta}(\ox)=0\Leftrightarrow \zeta(\ox)=0, \text{ where }\zeta(\ox)=\begin{cases}x_2-cx_3+\frac{\langle \nabla r(\ox),\otheta \rangle}{G(\ox)}, \ & \ox\neq 0, \\ 0,\ & \ox=0,\end{cases}
\end{align}
that is, $\zeta=0$ defines the grazing set $G_{\phi_i}$.
It follows from \eqref{r16} and $\langle \nabla r(\ox),\otheta \rangle=O(|\ox|^{2k})$ that $\zeta$ is $C^1$ but possibly not $C^2$ when $k>1$.   If $k=1$, then $G$ is a positive constant and the function $\zeta$ in \eqref{r17} is $C^\infty$. 

\textbf{6. }We have $H_p=2\xi_1\partial_{x_1}+2\overline\xi\partial_{\ox}-2\tau\partial_t$, so with $\urho=(1,0,t_0,0,\otheta,-1)$ we have
\begin{align*}
H_p\zeta (\urho)=2\langle\otheta, \partial_{\ox}\zeta (0)\rangle=2\langle\otheta,(1,-c)\rangle \neq 0
\end{align*}
by \eqref{r15}.

\textbf{7. }Finally we show that every point $\sigma\in\left(G\cap \mathrm{Graph}(d\phi_0)\right)\setminus\{\usigma\}$ near $\usigma$ satisfies 
$$\sigma\in (G^2_d\setminus G^3)\cup (G^{2k}_d\setminus G^{2k+1})\subset \cG_d.$$   Using the parametrization of $\partial M$ given by $(\ox,t)\mapsto (F(\ox),\ox,t)$, we can write $\phi_0=-t+\langle \otheta,\ox\rangle.$   Thus, such a $\sigma$ has the form 
$$\sigma=(\ox,t,\otheta,-1)=i^*\rho, \text{ where }\rho=(F(\ox),\ox,t_1,0,\otheta,-1)$$ for some $t_1$ near $t_0$ and $\ox$ near $0$ satisfying $g_{\otheta}(\ox)=0$.    
With $\beta(x_1,\ox)=x_1-F(\ox)$, if $\ox\neq 0$  we compute
\begin{align}\label{r19}
H_p\beta (\rho)=-2\langle\nabla F(\ox),\otheta\rangle=0, \ H_p^2\beta(\rho)=-4\langle\nabla^2 F(\ox)\otheta,\otheta\rangle>0.
\end{align}
Thus, $\sigma\in G^2_d\setminus G^3$.   If $\ox=0$, then $\sigma\in \cG_d$ has the same order as $\usigma$.
\end{proof}

\noindent
{\bf Remark} (Extension of Proposition \ref{r11}){\bf.} 
If one takes a more general function $F$ of the form 
\begin{align}\label{improve}
F(\ox)=1+h_2(\ox)+h_4(\ox)+ h_{2k}(\ox)+O(|\ox|^{2k+1}), \;\;\text{ for }k\geq 3
\end{align}
where the conditions \eqref{r6y} hold, we have checked that the conclusions of Proposition \ref{r11} still hold.
Indeed, one can show that the conditions \eqref{r6ya} {imply}
\begin{align*}
\langle \nabla h_2(x),\otheta\rangle =\langle \nabla h_4(x),\otheta\rangle=0\text{ for all }\ox,
\end{align*}
so \eqref{r14a} in step \textbf{3} of the above proof remains true.   The rest of the proof follows as before.

\subsubsection{Obstacles in $\mathbb{R}^n$}\label{ond}
Here we present examples involving obstacles $\cO\subset \mathbb{R}^n$ for any $n$  that satisfy all the assumptions of Theorem \ref{mt}.

\begin{ass}\label{ndA}
Let  $\mathcal{O}\subset \mathbb{R}^n$ be an obstacle that is strictly convex near $\mathrm P_0=(1,0)$,  and which is defined by a function  $F$ as in Definition \ref{r2} that satisfies the following additional condition
\begin{equation}\begin{gathered}\label{r20}
    F(\ox) = 1-h(|\Lambda \ox|^2), \
    h\in C^{\infty}( [0,R);[0,\infty) ), \\
    h(0)=0, \ h^{\prime}|_{(0,R)}>0, \ h^{\prime\prime}|_{[0,R)}\geq 0, \\
    \Lambda \text{ is a positive definite constant matrix. }
\end{gathered}\end{equation}
\end{ass}

\begin{prop}\label{rnd}
Suppose $\cO\subset \mathbb{R}^n$ is defined by a function $F$ as in Assumption \ref{ndA}.  
Let $P=\Box$ be the wave operator \eqref{r1} on $M=(\mathbb{R}^2\setminus \cO)\times \mathbb{R}_t$ and let $\phi_i=-t+\langle\otheta,\ox\rangle$ where $\otheta\in \mathbb S^{n-2}$.   
Then $\usigma:=i^*\urho \in \mathcal G_d$, where $\urho = (1,0,t_0, 0, \otheta, -1)$. 
The conditions of Assumption \ref{A2} are satisfied if one takes $\zeta(\ox)=\langle \otheta, \Lambda \ox \rangle$.   That is, we have
\begin{align*}
\begin{split}
G_{\phi_i}=\{(F(x_2),x_2,t) \ | \ \langle \otheta, \Lambda \ox \rangle=0, \ \ox \in B(0,r), \  t\in \mathbb{R}\}.
\end{split}
\end{align*}
Moreover, $H_p\zeta(\urho)\neq 0$   and every point in $\left(G\cap \mathrm{Graph}(d\phi_0)\right)\setminus\{\usigma\}$ near $\usigma$ lies in $\cG_d$.
\end{prop}

\begin{proof}
    We compute
\begin{subequations}\label{r21}
    \begin{align}
        \nabla F(\ox) & =-2h'(|\Lambda\ox|^2)\Lambda \ox,\ \langle\nabla F(\ox),\otheta\rangle=-2h'(|\ox|^2)\langle \otheta, \Lambda \ox \rangle, \label{r21a}\\
        \nabla^2 F(\ox) & =-2 h'(|\ox|^2)\Lambda - 4 h''(|\Lambda \ox|^2) \ (\Lambda \ox)\otimes (\Lambda \ox). \label{r21b}
    \end{align}
\end{subequations}
From \eqref{r21b} we see that $\nabla^2 F(\ox)<0$ for $\ox\neq 0$.   Thus, $\cO$ is strictly convex near $P_0=(1,0)$, so the results of \S \ref{nd} imply that  Assumption 
\ref{A3} on the forward flow map $Z^r$ holds.

For any $\otheta\in \mathbb S^{n-2}$, let $\usigma=i^*\urho$, where $\urho=(1,0,t_0,0,\otheta,-1)$.    Write the Taylor expansion of $h$ at $s=0$ as
\begin{align*}
h(s)=\sum^k_{j=1}\frac{h^{(j)}(0)}{j!}s^j+O(s^{k+1}),
\end{align*}
and observe that the first nonzero coefficient (if there is one) must be \emph{positive}, since $h''(s)\geq 0$ on $[0,R)$.
A computation similar to  \eqref{r6a} shows that 
\begin{align}\label{r23}
\begin{split}
\usigma\in G^{2k}_d\setminus G^{2k+1} & \Leftrightarrow h^{(j)}(0)=0 \text{ for }j=1,\dots,k-1\text{ and }h^{(k)}(0)>0;\\
\usigma\in G^{\infty}_d & \Leftrightarrow h^{(j)}(0)=0 \text{ for all }j.
\end{split}
\end{align}
Both cases give $\usigma\in \mathcal G_d$.

To verify Assumption \ref{A2} we recall that the grazing set $G_{\phi_i}$ is determined by $\langle\nabla F(\ox),\otheta\rangle=0$, and 
from \eqref{r20} and \eqref{r21a} we see that 
\begin{align*}
\langle\nabla F(\ox),\otheta\rangle=0\Leftrightarrow \zeta(\ox)=0, \text{ where }\zeta(\ox):=\langle \Lambda\ox,\otheta\rangle.
\end{align*}
We have $\zeta\in C^{\infty}$ and 
\begin{align*}
H_p\zeta(\urho)=2\langle \Lambda \otheta,\otheta\rangle>0
\end{align*}
since $\Lambda$ is positive definite.

Finally, a repetition of the computation in step \textbf{7} of the proof of Proposition \ref{r11} shows that points  $\sigma\in G\setminus \{\usigma\}$ must lie in $\cG_d$.    If the $\ox$ coordinate of $\sigma$ is zero, then $\sigma$ has  the same order as $\usigma$; otherwise,  $\sigma\in G^2_d\setminus G^3$.    Thus, Assumption \ref{A2} holds.
\end{proof}

\Remark 
Consider the function $F_1(\ox)=1-(x_2^{2k}+\dots+x_n^{2k})$ of Example \eqref{r4a}.   Now the condition $\nabla^2 F_1<0$ fails, but the obstacle $\cO$ defined by $F_1$ is strictly convex near $\mathrm P_0=(1,0)$.   If we take $\phi_i=-t+\langle\otheta,\ox\rangle$ where $\otheta=(1,0,\dots,0)\in \mathbb S^{n-2}$, then Assumption \ref{A2} is easily seen to hold with $\zeta(\ox)=x_2$.

\subsection{Assumption \ref{A3}:  two-dimensional convex obstacles.} \label{2d}

In this section, we show that Assumption \ref{A3} is satisfied by plane waves when $\mathcal O$ is any two-dimensional obstacle that is strictly convex near $\mathrm P_0=(1,0)$.

We introduce the notation
\begin{equation}
\label{omega}
    \omega:=\{ (s,x_2,t^{\prime}) \ | \ 0\leq s<s_0, \ |x_2|<r, \ F^{\prime}(x_2)\geq 0, \ t^{\prime}\in \RR \} \simeq [0,s_0)\times (I_-\sqcup G_{\phi_i})
\end{equation}
and the ``interior'' of the domain
\begin{equation}
\label{omegao}
    \mathring\omega:=\{ (s,x_2,t^{\prime}) \ | \ s\geq 0, \ |x_2|<r, \ F^{\prime}(x_2)> 0, \ t^{\prime}\in \RR \} \simeq [0,s_0)\times I_-.
\end{equation}

\begin{lemm}
\label{lem: 2d-ref-flow}
Let $\mathcal O$ and $F$ be as in Definition \ref{r2} with $n=2$, $M=(\RR^2\setminus \mathcal O)\times \RR$, and $\phi_i=-t+\langle \otheta, \ox \rangle$ with $\otheta=\pm 1$ be the incoming phase for the wave operator $\Box$. Then through the parametrization \eqref{omega}, the reflected flow map $Z_r$ in Definition \ref{q12}, is given by 
\begin{equation}\begin{gathered}
\label{eq: 2d-ref-flow}
    Z_r: [0,s_0)\times (I_-\cup G_{\phi_i}) \to M, \\ Z_r(s,x_2,t^{\prime}) = \left( F(x_2)+\frac{4 \otheta F^{\prime}(x_2)}{1+F^{\prime}(x_2)^2}s, \ x_2+\frac{2 \otheta (1-F^{\prime}(x_2)^2)}{1+F^{\prime}(x_2)^2}s, \ t^{\prime}+2s \right).
\end{gathered}\end{equation}
\end{lemm}
\begin{proof}
    The wave operator $\Box$ has symbol $p(x,t,\xi,\tau):=|\xi|^2-\tau^2$. The Hamiltonian vector field of $p$ is $H_p=2\xi_1\partial_{x_1}+2\xi_2\partial_{x_2}-2\tau\partial_t$. The incoming bicharacteristics passing $(x_1^0, x_2^0,t^0, d\phi_i(x_1^0, x_2^0,t^0))$ where $\otheta x_2^0<0$, $t^0<0$ are then 
    \begin{equation*}
        \gamma_i(s) := (x_1, x_2, t, \xi_1,\xi_2,\tau)(s) = (x_1^0, x_2^0+2\otheta s, t^0+2s, 0,\otheta,-1), \ s\geq 0.
    \end{equation*}
    Notice that when $x_1^0=1$, $\gamma_i$ hits $\partial T^*M$ tangentially; when $x_1^0<1$, $\gamma_i$ hits $\partial T^*M$ transversally; when $x_1^0>1$, $\gamma_i$ does not hit $\partial T^*M$ near $(1,0)$.

    Suppose $\gamma_i$ hits $\partial T^*M$ at the point $(F(x_2), x_2, t^{\prime}, 0,\otheta,-1)$, that is, 
    $$(x_1(s), x_2(s), t(s)) =(F(x_2),x_2,t^{\prime})\in \partial M$$
    for some $s\geq 0$. Then the initial point of the reflected bicharacteristic is the unique point $(F(x_2),x_2,t^{\prime}, \xi_1^r, \xi_2^r, \tau^r)\in p^{-1}(0)\cap \partial T^*M$ such that 
    \begin{equation*}
        i^*(F(x_2),x_2,t^{\prime},0,\otheta,-1)=i^*(F(x_2),x_2,\tau^{\prime},\xi_1^r, \xi_2^r,\tau^r). 
    \end{equation*}
    Notice that $\mathrm{Ker}(i^*)=N^*(\partial M)$, which is the conormal bundle on $\partial M$. Near $\mathrm P_0=(1,0)$, $\partial M$ is given by $x_1-F(x_2)=0$, hence the normal vectors of $\partial M$ at $(F(x_2),x_2,t^{\prime})$ are parallel to $(1,-F^{\prime}(x_2),0)$. Thus there exists $c\in \RR$ such that 
    \[ (0,\otheta,-1)-(\xi_1^r,\xi_2^r,\tau^r) = c(1,-F^{\prime}(x_2),0), \ |(\xi_1^r, \xi_2^r)|=|\tau^r|. \]
    From here we solve 
    \[ \xi_1^r=\frac{2 \otheta F^{\prime}(x_2)}{1+(F^{\prime}(x_2))^2}, \ \xi_2^r = \otheta \frac{1-F^{\prime}(x_2)^2}{1+F^{\prime}(x_2)^2}, \ \tau^{r}=-1. \]
    The reflected bicharacteristic satisfies 
    \begin{equation}\begin{cases}\label{eq: bichar}
        & \dot{x}_1=2\xi_1, \ \dot{x}_2 = 2\xi_2, \ \dot{t}=-2\tau, \ \dot{\xi}_1=\dot{\xi}_2=\dot{\tau}=0, \\
        & x_1(0)=F(x_2), \ x_2(0)=x_2, \ t(0)=t^{\prime}, \ \xi_1(0)=\xi_1^r, \ \xi_2(0)=\xi_2^r, \ \tau(0)=\tau^r.
    \end{cases}\end{equation}
    Hence we obtain the reflected bicharacteristic passing $(x_1,x_2,t^{\prime}, \xi_1^r, \xi_2^r,\tau^r)$:
    \begin{equation*}
        \gamma_r(s) = (x_1(s),x_2(s), \tau(s), \xi_1(s), \xi_2(s), \tau(s)  )
    \end{equation*}
    where
    \[\begin{gathered} 
    x_1(s) = F(x_2)+\frac{4 \otheta F^{\prime}(x_2)}{1+F^{\prime}(x_2)^2}s, \ x_2(s) = x_2+\frac{2 \otheta (1-F^{\prime}(x_2)^2)}{1+F^{\prime}(x_2)^2}s, \ t(s)=t^{\prime}+2s, \\
    \xi_1^r(s) = \frac{2 \otheta F^{\prime}(x_2)}{1+F^{\prime}(x_2)^2}, \ \xi_2^r(s) = \frac{2 \otheta (1-F^{\prime}(x_2)^2)}{1+F^{\prime}(x_2)^2}, \ \tau^r(s) =-1.
\end{gathered}\]
It remains to project $\gamma_r$ onto the base manifold $M$ to conclude the formula \eqref{eq: 2d-ref-flow}.
\end{proof}

\noindent
{\bf Remark} (Equal angle reflection){\bf .}
    The projections onto the $(x_1,x_2)$-plane of the incoming and reflected bicharacteristic exhibit ``equal angle reflection''. That is 
    \begin{equation}
    \label{ear}
    (0,-\otheta)\cdot n(x_2) = (\xi_1^r,\xi_2^r)\cdot n(x_2) 
    \end{equation}
    where $n(x_2)=(1, -F^{\prime}(x_2))$ is a normal vector to the obstacle $\mathcal O$ at $(F(x_2),x_2)$. 
    Indeed, 
    \[\begin{split}
        \eqref{ear}\Leftrightarrow [(\xi_1^r, \xi_2^r)+(0,\otheta)]\cdot n(x_2)=0 \Leftrightarrow [(\xi_1^r, \xi_2^r)+(0,1)]\cdot [(\xi_1^r,\xi_2^r)-(0,\otheta)]=0.
    \end{split}\]
    The last equality holds as $\otheta=\pm 1$ and $|(\xi_1^r,\xi_2^r)|=1$.

The next proposition justifies Assumption \ref{A3} for strictly convex obstacles in 2D.

\begin{prop}\label{prop: 2d-ref}
    Let $\mathcal O$, $F$, $Z_r$ be as in Lemma \ref{lem: 2d-ref-flow} with $n=2$, and $\omega$, $\mathring\omega$ be as in \eqref{omega}, \eqref{omegao}. Then the map $Z_r: \mathring\omega \to Z_r(\mathring\omega)$ is a $C^{\infty}$ diffeomorphism, which extends to a homeomorphism $Z_r: \omega\to Z_r(\omega)$.
\end{prop}

\begin{proof}
    We first remark that by Proposition \ref{ra11}, the domains $\omega$, $\mathring\omega$ takes the form 
    \[ \omega=\{ (s,x_2,t^{\prime}) \ | \ s\geq 0,\ \otheta x_2\leq 0, \ t^{\prime}\in \RR \}, \ \mathring\omega=\{ (s,x_2,t^{\prime}) \ | \ s\geq 0, \ \otheta x_2< 0, \ t^{\prime}\in \RR \}. \]

    \textbf{1. Injectivity.} To show that $Z_r: \omega\to Z_r(\omega)$ is injective, it suffices to show the injectivity of  
    \begin{equation*}
        z(s,x_2):= \left( F(x_2)+\frac{4 \otheta F^{\prime}(x_2)}{1+F^{\prime}(x_2)^2}s, \ x_2+\frac{2 \otheta (1-F^{\prime}(x_2)^2)}{1+F^{\prime}(x_2)^2}s\right)
    \end{equation*}
    on the $(s,x_2)$-projection of $\omega$.

    Suppose the contrary, then there exist $(s,x_2)$, $(s^*, x_2^*)$ in the $(s,x_2)$-projection of $\omega$ such that 
    \begin{equation}\label{z1z2}
        (s,x_2)\neq (s^*,x_2^*), \ z(s,x_2)=z(s^*,x_2^*)=:(z_1,z_2).
    \end{equation}
    Without loss of generality, we assume $\otheta (x_2^*-x_2)>0$.

    \begin{figure}[t]
        \begin{center}
        \includegraphics[scale=0.5]{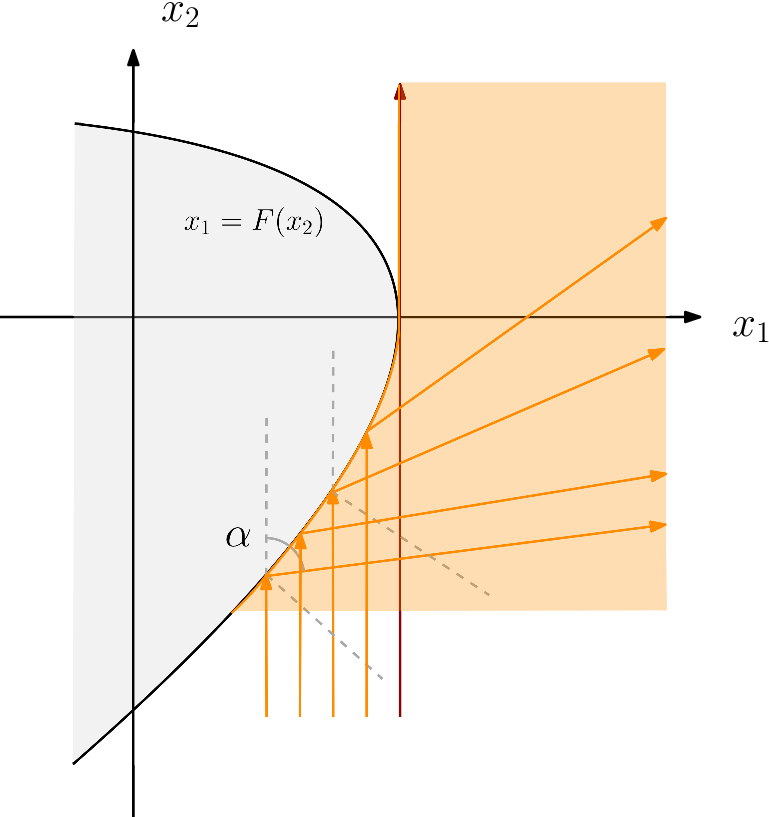}
        \caption{Reflected rays in the proof of Proposition \ref{prop: 2d-ref} when $\otheta=1$. }
        \label{2dreflect}
        \end{center}
    \end{figure}

    Let $\alpha(x_2)$ be the angle between the vectors $(0,\otheta)$ and $(\xi_1^r(x_2),\xi_2^r(x_2))$. Shrink the $x_2$ component of $\omega$ if needed, we can assume that $0\leq \alpha(x_2)<\frac{\pi}{2}$. Then we have 
    \begin{equation}\label{alpha} 
    \sin{\alpha(x_2)} = \frac{ 2 \otheta F^{\prime}(x_2) }{ 1+F^{\prime}(x_2)^2 }, \ \cos{\alpha(x_2)} = \frac{ 1-F^{\prime}(x_2)^2 }{1+F^{\prime}(x_2)^2}. 
    \end{equation}
    We first claim that in $\omega$, the reflected bicharacteristics are {\em defocusing}, that is, $\alpha(x_2^*)<\alpha(x_2)$. Indeed, differentiate the first identity in \eqref{alpha} with respect to $x_2$ and we obtain 
    \[ \alpha^{\prime}(x_2)\cos(\alpha(x_2)) = \frac{2 \otheta F^{\prime\prime}(x_2) (1-F^{\prime}(x_2)^2)}{(1+F^{\prime}(x_2)^2)^2}. \]
    Use the second identity in \eqref{alpha} and we find 
    \[ \alpha^{\prime}(x_2)=\frac{2 \otheta F^{\prime\prime}(x_2)}{1+F^{\prime}(x_2)^2} \  \Rightarrow \ \otheta \alpha^{\prime}(x_2)\leq 0 \text{ in } \omega \]
    which implies that $\alpha(x_2^*)\leq \alpha(x_2)$. Moreover, if $\alpha(x_2^*)=\alpha(x_2)$, then $F^{\prime\prime}=0$ on $[x_2,x_2^*]$ when $\otheta=1$, or on $[x_2^*, x_2]$ when $\otheta=-1$; but neither of the cases is possible since $F$ is strictly concave.

    Now by the second identity in \eqref{z1z2}, we know $(z_1,z_2)$ satisfies
    \[\begin{gathered}
        (z_2-x_2)\tan{\alpha(x_2)} = \otheta ( z_1-F(x_2)), \ (z_2-x_2^*)\tan{\alpha(x_2^*)} = \otheta (z_1-F(x_2^*)).
    \end{gathered}\]
    From this we find 
    \begin{equation}\begin{split}\label{f1f2} 
    F(x_2^*)-F(x_2)
    = & (z_2-x_2) \otheta \tan\alpha(x_2)-(z_2-x_2^*)\otheta \tan\alpha(x_2^*) \\
    = & (\tan\alpha(x_2)-\tan\alpha(x_2^*))\otheta z_2+ \otheta (x_2^*\tan\alpha(x_2^*)-x_2\tan\alpha(x_2)). 
    \end{split}\end{equation}
    We showed $0\leq \alpha(x_2^*)<\alpha(x_2)<\frac{\pi}{2}$, hence $\tan\alpha(x_2)-\tan\alpha(x_2^*)>0$.
    Since $s^*\geq 0$, $\cos\alpha(x_2^*)\geq 0$, we know $\otheta z_2 = \otheta x_2^*+2s^*\cos{\alpha(x_2^*)}\geq \otheta x_2^*$. 
    Using the monotonicity of the right hand side of \eqref{f1f2} in $z_2$, we conclude that
    \begin{equation}\label{tan1}
    F(x_2^*)-F(x_2)\geq (x_2^*-x_2) \otheta \tan\alpha(x_2). 
    \end{equation}
    On the other hand, by \eqref{alpha} we have 
    \[ \tan\alpha(x_2) = \frac{ 2\otheta F^{\prime}(x_2) }{1-F^{\prime}(x_2) } > \otheta F^{\prime}(x_2). \]
    Combining this with the assumption $\otheta (x_2^*-x_2)>0$ and the strict concavity of $F$, we obtain
    \begin{equation}\label{tan2}
    F(x_2^*)-F(x_2)< F^{\prime}(x_2)(x_2^*-x_2) = \otheta F^{\prime}(x_2) \cdot \otheta (x_2^*-x_2) < \tan{\alpha(x_2)} \cdot \otheta (x_2^*-x_2). 
    \end{equation}
    This contradicts \eqref{tan1}. We have now proved the injectivity of $Z_r: \omega\to Z_r(\omega)$.

    \textbf{2. Local diffeomorphism.} To prove $Z_r$ is a local diffeomorphism from $\mathring\omega \to Z_r(\mathring\omega)$, it suffices to show its Jacobian $j$ is nonzero in $\mathring\omega$. A direct computation gives that 
    \begin{equation}\label{tan2a}\begin{split} 
        j(s,x_2,t^{\prime}) 
        = & \begin{vmatrix}
            2\sin\alpha & F^{\prime}+2s\alpha^{\prime}\cos{\alpha} & 0 \\
            2\otheta \cos{\alpha} & 1-2s\otheta \alpha^{\prime} \sin{\alpha} & 0 \\
            2 & 0 & 1
        \end{vmatrix} \\
        = & 2\left( \sin\alpha -\otheta F^{\prime} \cos{\alpha} -2s \otheta \alpha^{\prime} \right) \\
        = & 2 \otheta F^{\prime}(x_2)-\frac{ 8 s F^{\prime\prime}(x_2)  }{ 1+F^{\prime}(x_2)^2 }.
    \end{split}\end{equation}
    By the definition of $\omega$, we have $\otheta F^{\prime}(x_2)>0$. By the concavity of $F$, we have $F^{\prime\prime}\leq 0$.
    Hence when $s\geq 0$, we have 
    \begin{equation*}
        j(s,x_2,t^{\prime})\geq 2\otheta F^{\prime}(x_2)>0.
    \end{equation*}
    This completes the proof.
\end{proof}

\Remark\label{tan2b}
For the functions $F_0$ and $F_1$ in Examples \eqref{r4} with $n=2$ we obtain from \eqref{tan2a} that 
\begin{align}
j(s,x_2,t')\sim |x_2|^{2k-1}+s|x_2|^{2k-2}.
\end{align}
This reduces to the formula of \cite{cheverry1996} when $k=1$. For the function $F_2$ in Examples \eqref{r4} with $n=2$ we obtain
\begin{align}
j(s,x_2,t')\sim e^{-\frac{1}{x^2}}\left(|x_2|^{-3}+s|x_2|^{-6}\right).
\end{align}
Here we have taken $\otheta=1$ and the grazing set is $\{x_2=0\}$.

\subsection{Assumption \ref{A3}: \texorpdfstring{$n$}{TEXT}-dimensional convex obstacles.}\label{nd}

We generalize the results in the previous section to $n$-dimensional convex obstacles. 

We first introduce the parametrizations of $[0,s_0)\times (I_-\sqcup G_{\phi_i})$ and $[0,s_0)\times I_-$:
\begin{equation}\begin{gathered}\label{nomega}
    \omega:=[0,s_0)\times \{ (\ox, t^{\prime}) \ | \ \langle \otheta, \nabla F(\ox) \rangle \geq 0, \ |\ox|<r, \ t^{\prime}\in \RR \}\simeq [0,s_0)\times (I_-\sqcup G_{\phi_i}), \\
    \mathring\omega:=\{ (s,\ox, t^{\prime}) \ | \ 0\leq s<s_0, \ \langle \otheta, \nabla F (\ox) \rangle > 0, \ |\ox|<r, \ t^{\prime}\in \RR \}\simeq [0,s_0)\times I_-.
\end{gathered}\end{equation}

\begin{lemm}
\label{lem: nd-ref-flow}
Let $\mathcal O$ and $F$ be as in Definition \ref{r2}, $M:=(\RR^n\setminus \mathcal O)\times \RR$, and $\phi_i=-t+\langle \otheta,\ox \rangle$ be the incoming phase for the wave operator $P=\Box$. Then through the identification \eqref{nomega}, the reflected flow map $Z_r$ in Definition \ref{q12}, is given by 
\begin{equation}\begin{gathered}
\label{eq: nd-ref-flow}
    Z_r: [0,s_0)\times (I_-\sqcup G_{\phi_i}) \to M, \\
    Z_r(s,\ox, t^{\prime}):=(F(\ox)+2s\xi_1^r(\ox), \ox+2s\oxi^r(\ox), t^{\prime}+2s)
\end{gathered}\end{equation}
with 
\begin{equation}\label{xir}
    \xi_1^r(\ox) := \frac{ 2\langle \otheta, \nabla F(\ox) \rangle }{ 1+|\nabla F(\ox)|^2 }, \ \oxi^r(\ox) := \otheta - 2\frac{ \langle \otheta, \nabla F(\ox) \rangle}{1+|\nabla F(\ox)|^2} \nabla F(\ox).
\end{equation}
\end{lemm}
\begin{proof}
    The proof is similar to the proof of Lemma \ref{lem: 2d-ref-flow}. The wave operator $\Box$ has symbol $p=|\xi|^2-\tau^2$, whose Hamiltonian vector field is $H_p=2\xi\cdot \nabla_x -2\tau\partial_t$. Thus for the incoming phase $\phi_i=-t+\langle \otheta, \ox \rangle$, the incoming bicharacteristics passing $(x_1^0, \overline x^0, t, 0, \otheta,-1)$ is 
    \[ \gamma_i(s) :=(x_1^0, \overline x, t, \xi_1, \oxi, \tau)(s) = (x_1^0, \overline x^0+2\otheta s, \tau+2s, 0, \otheta, -1). \]
    Suppose $\gamma_i(s)$ hits $\partial T^*M$ at $( F(\ox), \ox, t^{\prime}, 0, \otheta,-1 )$. Then the starting point of the reflected bicharacteristic $(F(\ox),\ox, t^{\prime}, \xi_1^r, \oxi^r, \tau^r)$ must satisfy 
    \begin{equation}\label{nd-ref-eq}
        i^*(0,\otheta,-1)=i^*(\xi_1^r, \oxi^r, \tau^r), \ p(F(\ox),\ox, t^{\prime},\xi_1^r, \oxi^r,\tau^r)=0.
    \end{equation}
    Since $\mathrm{Ker}( i^*)=N^*(\partial M)$, and the normal vectors of $\partial M$ at $(F(\ox),\ox,t^{\prime})$ is parallel to $(1,-\nabla F(\ox),0)$, we can rewrite \eqref{nd-ref-eq} as 
    \begin{equation*}
        (0,\otheta,-1)-(\xi_1^r,\oxi^r, \tau^r)=c(1,-\nabla F(\ox),0), \ |(\xi_1^r, \oxi^r)|=|\tau^r|.
    \end{equation*}
    From this we solve 
    \begin{equation*}
    \xi_1^r = \frac{ 2\langle \otheta, \nabla F(\ox) \rangle }{ 1+|\nabla F(\ox)|^2 }, \ \oxi^r = \otheta - 2\frac{ \langle \otheta, \nabla F(\ox) \rangle}{1+|\nabla F(\ox)|^2} \nabla F(\ox), \ \tau^r=-1. 
    \end{equation*}
    A similar computation as \eqref{eq: bichar} gives the reflected bicharacteristics
    \[ \gamma_r = \gamma_r(s,\ox, t^{\prime}) = (F(\ox)+2s\xi_1^r(\ox), \ox+2s\oxi^r(\ox), t^{\prime}+2s, \xi_1^r(\ox), \oxi^r(\ox), -1 ). \]
    Project the bicharacteristics onto $M$ and we obtain the reflected flow map \eqref{eq: nd-ref-flow}.
\end{proof}

\noindent
{\bf Remark} (Law of reflection){\bf .}
The projection onto the $x$-plane of the incoming and reflected bicharacteristics obeys the following law of reflection: at $(F(\ox),\ox)\in \partial\mathcal O$, the direction of the incoming rays $(0,-\otheta)$, the direction of the reflected rays $(\xi_1^r, \oxi^r)$ and the normal vector $(1,-\nabla F(\ox))$ are coplanar, and the normal vector bisects the angle formed by $(0,-\otheta)$ and $(\xi_1^r, \oxi^r)$. The proof is similar ot the proof of \eqref{ear}.

The remaining part of this section is devoted to justifying that Assumption \ref{A3} holds for strictly convex obstacles in $n$ dimensional and plane wave phases. 

\begin{prop}\label{nd-ref}
    Let $\mathcal O$, $F$, $Z_r$ be as in Lemma \ref{lem: nd-ref-flow}, and $\omega$, $\mathring\omega$ be as in \eqref{nomega}. Then the map $Z_r: \mathring\omega \to Z_r(\mathring\omega)$ is a $C^{\infty}$ diffeomorphism, which extends to a homeomorphism $Z_r: \omega\to Z_r(\omega)$.
\end{prop}

\begin{proof}
    \textbf{1. Injectivity.} To show the injectivity of $Z_r$, it suffices to show that the map
    \[  z(s,\ox):= (F(\ox)+2s\xi_1^r(\ox), \ox+2s\oxi^r(\ox)) \]
    is injective on the $(s,\ox)$-projection of $\omega$. 

    Suppose the contrary that there exists $(s, \ox)$, $(s^*,\ox^*)$ in the $(s,\ox)$-projection of $\omega$, such that 
    \begin{equation}\label{yys}
        (s,\ox)\neq (s^*, \ox^*), \ z(s,\ox)=z(s^*, \ox^*).
    \end{equation}
    From \eqref{yys} one can see that $s\neq s^*$, $\ox \neq \ox^*$. We record two observations based on \eqref{yys}:
    \begin{itemize}[leftmargin=40pt]
        \item[\textbf{OB1.}] The set of vectors 
        \[ \left\{ \xi^r(\ox), \ \xi^r(\ox^*), \ (F(\ox^*) - F(\ox), \ox^*-\ox) \right\} \]
        is linearly dependent, where $\xi^r:=(\xi_1^r, \bar{\xi}^r)$;

        \item[\textbf{OB2.}] There holds 
        \begin{equation}\label{focus}
            \langle \xi^r(\ox^*)-\xi^r(\ox), (F(\ox^*) - F(\ox), \ox^*-\ox ) \rangle < 0.
        \end{equation}
    \end{itemize}
    \begin{figure}[t]
    \begin{center}
        \includegraphics[scale=0.5]{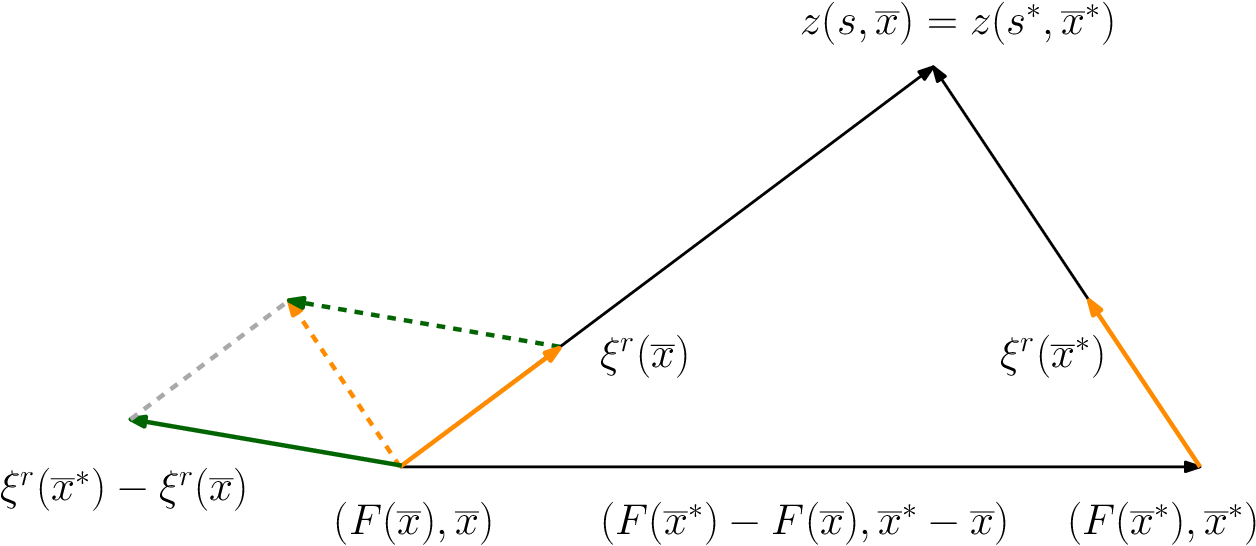}
        \caption{Intersecting reflected rays satisfying \eqref{yys}. }
        \label{intersectrays}
    \end{center}
    \end{figure}
    \noindent
    \begin{proof}[Proof of {\bf OB1}] 
    This is because $z(s,\ox)=z(s^*,\ox^*)$ implies
        \[ (F(\ox),\ox)+2s\xi^r(\ox)=( F(\ox^*), \ox^* )+2s^*\xi^r(\ox^*), \]
        that is,
        \begin{equation}\label{linear} 
        2s\xi^r(\ox) - 2s^*\xi^r(\ox^*) -( F(\ox^*)-F(\ox), \ox^*-\ox )=0.
        \end{equation}
        This justifies {\bf OB1}.
    \end{proof}

    \begin{proof}[Proof of {\bf OB2}] 
    Indeed, using \eqref{linear} and the facts that $|\xi^r(\ox)|=|\xi^r(\ox^*)|=1$, we obtain 
        \[ \langle \xi^r(\ox^*)-\xi^r(\ox), (F(\ox^*) - F(\ox), \ox^*-\ox ) \rangle = -2(s+s^*)\left( 1-\langle \xi^r(\ox), \xi^r(\ox^*) \rangle
        \right)\leq 0. \]
        Moreover, the inner product on the left can be $0$ if and only if $\xi^r(\ox)=\xi^r(\ox^*)$, which is true if and only if $\nabla F(\ox) = \nabla F(\ox^*)$ or $\langle \otheta, \nabla F(\ox)\rangle = \langle \otheta, \nabla F(\ox^*) \rangle=0$. 

        If $\nabla F(\ox)=\nabla F(\ox^*)$ with $\ox\neq \ox^*$, then by the strict concavity of $F$, we have 
        \[ F(\ox)-F(\ox^*)<\langle \nabla F(\ox^*), \ox-\ox^* \rangle = -\langle \nabla F(\ox), \ox^*-\ox \rangle <-( F(\ox^*) - F(\ox) ). \]
        This is impossible. 

        If $\langle \otheta, \nabla F(\ox) \rangle = \langle \otheta, \nabla F(\ox^*) \rangle=0$ with $\ox \neq \ox^*$. Then from \eqref{xir}, we know $\xi^r(\ox)=\xi^r(\ox^*)=(0,\otheta)$. The assumption $z(s,\ox)=z(s^*, \ox^*)$ implies 
        \[ \ox+2s\otheta = \ox^* + 2s^*\otheta \ \Rightarrow \ \ox^*-\ox = 2(s-s^*)\otheta. \]
        Since $\ox\neq \ox^*$, by the strict concavity of $F$, we have  
        \[ F(\ox^*)-F(\ox)<\langle \nabla F(\ox), \ox^*-\ox \rangle = 2(s-s^*)\langle \otheta, \nabla F(\ox) \rangle=0. \]
        Similarly, we have 
        \[ F(\ox) - F(\ox^*)<\langle \nabla F(\ox^*), \ox-\ox^* \rangle = 2(s^*-s)\langle \otheta, \nabla F(\ox^*) \rangle=0. \]
        This is a contradiction. We can now conclude that \eqref{focus} holds.
        \end{proof}

        On the other hand, we claim that for $\ox, \ox^*$ in the $\ox$-projection of $\omega$, there holds 
    \begin{equation}\label{defocus}
        \langle \xi^r(\ox^*)-\xi^r(\ox), (F(\ox^*) - F(\ox), \ox^*-\ox ) \rangle \geq 0.
    \end{equation}
    Indeed, by \eqref{xir} and the concavity of $F$, there holds
    \[\begin{split} 
    & \langle \oxi^r(\ox^*) - \oxi^r(\ox), \ox^*-\ox \rangle \\
    & = \frac{ 2\langle \otheta, \nabla F(\ox) \rangle }{ 1+|\nabla F(\ox)|^2 } \langle \nabla F(\ox), \ox^*-\ox \rangle
    + \frac{ 2\langle \otheta, \nabla F(\ox^*) \rangle }{ 1+|\nabla F(\ox^*)|^2 } \langle \nabla F(\ox^*), \ox-\ox^* \rangle \\
    & \geq  \frac{ 2\langle \otheta, \nabla F(\ox) \rangle }{ 1+|\nabla F(\ox)|^2 } ( F(\ox^*) - F(\ox) )
    + \frac{ 2\langle \otheta, \nabla F(\ox^*) \rangle }{ 1+|\nabla F(\ox^*)|^2 } ( F(\ox) - F(\ox^*) ) \\
    & = \left( \frac{ 2\langle \otheta, \nabla F(\ox) \rangle }{ 1+|\nabla F(\ox)|^2 } 
    - \frac{ 2\langle \otheta, \nabla F(\ox^*) \rangle }{ 1+|\nabla F(\ox^*)|^2 } \right) ( F(\ox^*) - F(\ox) ) \\
    & = - (\xi_1^r(\ox^*) - \xi_1^r(\ox) )( F(\ox^*) - 
    F(\ox )).
    \end{split}\]
    This proves \eqref{defocus}, which contradicts the observation \eqref{focus}.

    {\bf 2. Local diffeomorphism.} We now show that $Z_r: \mathring\omega\to Z_r(\mathring\omega)$ is a local diffeomorphism. For that, we compute the Jacobian $j$ of $Z_r$:
    \begin{equation*}\begin{split}
    j(s,\ox,t^{\prime}) 
    = & \begin{vmatrix}
        2\xi_1^r & \partial_{x_2}F+2s\partial_{x_2}\xi_1^r & \partial_{x_3}F+2s\partial_{x_3}\xi_1^r & \cdots & \partial_{x_n}F+2s\partial_{x_n}\xi_1^r & 0 \\
        2\xi_2^r & 1+2s\partial_{x_2}\xi_2^r & 2s\partial_{x_3}\xi_2^r & \cdots & 2s\partial_{x_n}\xi_2^r & 0 \\
        2\xi_3^r & 2s\partial_{x_2}\xi_3^r & 1+2s\partial_{x_3}\xi_3^r & \cdots & 2s\partial_{x_n}\xi_3^r & 0 \\
        \cdots & \cdots & \cdots & \cdots & \cdots & \cdots \\
        2\xi_n^r & 2s\partial_{x_2}\xi_n^r & 2s\partial_{x_3}\xi_n^r & \cdots & 1+2s\partial_{x_n}\xi_n^r & 0 \\
        2 & 0 & 0 & \cdots & 0 & 1
    \end{vmatrix} \\
    = &  2\begin{vmatrix}
        \xi_1^r & \nabla F + 2s \nabla \xi_1^r \\
        (\oxi^r)^T & I+2s \frac{\partial \oxi^r}{\partial \ox}
    \end{vmatrix}.
\end{split}\end{equation*}
By row reduction, we have 
\begin{equation}\begin{gathered}\label{ja}
    j(s,\ox,t^{\prime}) 
    =  2 \begin{vmatrix}
    \xi_1^r & \nabla F+2s\nabla \xi_1^r \\
    0 & I+2s\frac{ \partial \oxi^r }{\partial \ox}-\frac{1}{\xi_1^r} \oxi^r \otimes (\nabla F+ 2s\nabla \xi_1^r)
    \end{vmatrix} = 2\xi_1^r \det(A), \\
    \text{with } A:= I-\frac{\oxi^r \otimes \nabla F}{\xi_1^r} +2s \left( \frac{\partial\oxi^r}{\partial \ox} -\frac{\oxi^r \otimes \nabla \xi_1^r}{\xi_1^r}  \right).
\end{gathered}\end{equation}
Here for two $n-1$ dimensional row vectors $v_1$, $v_2$, we define their tensor product by $v_1\otimes v_2:=v_1^T v_2$, which is an $(n-1)\times (n-1)$ matrix.

By \eqref{xir}, for $2\leq k, \ell\leq n$, we have 
\[ \xi_k^r = \theta_k-\xi_1^r \partial_{x_k}F \ \Rightarrow \ \partial_{x_{\ell}}\xi_k^r = - \partial_{x_k}F \partial_{x_{\ell}}\xi_1^r - \xi_1^r \partial_{x_k}\partial_{x_{\ell}}F. \]
Therefore, we have
\begin{equation*}\begin{split}
    \frac{\partial\oxi^r}{\partial \ox} 
    = -\nabla F\otimes \nabla \xi_1^r -\xi_1^r \nabla^2 F.
\end{split}\end{equation*}
Hence 
\begin{equation*}\begin{split}
    A= & I-\frac{\oxi^r \otimes \nabla F}{\xi_1^r} -2s \left( \nabla F\otimes \nabla \xi_1^r +\xi_1^r \nabla^2 F + \frac{\oxi^r \otimes \nabla \xi_1^r}{\xi_1^r}  \right) \\
    = & I-\frac{\oxi^r \otimes \nabla F}{\xi_1^r} -2s \left( \xi_1^r \nabla^2 F + \frac{ (\oxi^r+\xi_1^r \nabla F)\otimes \nabla \xi_1^r }{\xi_1^r} \right) \\
    = & I-\frac{\oxi^r \otimes \nabla F}{\xi_1^r} -2s \left( \xi_1^r \nabla^2 F + \frac{\otheta\otimes \nabla \xi_1^r}{\xi_1^r} \right)
\end{split}\end{equation*}
Use the formula for $\xi_1^r$ in \eqref{xir} and we compute for $2\leq \ell\leq n$,
\begin{equation*}
    \partial_{x_{\ell}}(\log \xi_1^r) = \frac{ \sum_{2\leq k\leq n} \theta_k\partial_{x_{\ell}}\partial_{x_k}F }{\langle \otheta,\nabla F  \rangle} -\frac{ 2\sum_{2\leq k\leq n} \partial_{x_k}F\partial_{x_{\ell}}\partial_{x_k}F }{ 1+|\nabla F|^2}.
\end{equation*}
Therefore 
\[ \nabla (\log{\xi_1^r}) = \frac{ \otheta \cdot \nabla^2 F }{ \langle 
\otheta, \nabla F \rangle } - \frac{2\nabla F \cdot \nabla^2 F}{ 1+|\nabla F|^2 } = \frac{1}{\langle \otheta, \nabla F \rangle}\left( 
\otheta - \frac{2\langle \otheta, \nabla F \rangle}{ 1+|\nabla F|^2 } \nabla F \right)\cdot \nabla^2 F = \frac{ \oxi^r \cdot \nabla^2 F }{ \langle \otheta, \nabla F \rangle }. \]
We can now simplify $A$ as
\begin{equation*}
    A=I-\frac{ \oxi^r \otimes \nabla F }{ \xi_1^r } -2s \left( \xi_1^r I + \frac{ \otheta \otimes \oxi^r }{\langle \otheta, \nabla F \rangle}\right)\nabla^2 F.
\end{equation*}
Denote
\begin{equation}\label{bc}
    B := I-\frac{\oxi^r\otimes \nabla F}{\xi_1^r}, \ C:= \xi_1^r I +\frac{ \otheta\otimes \oxi^r }{ \langle \otheta, \nabla F \rangle }.
\end{equation}
Then we can write 
\begin{equation}\label{abc}
    A=B-2sC\nabla^2 F.
\end{equation}

The following lemmata are used to show that $A$ has a positive determinant. 

\begin{lem}\label{magic}
    Let $B$, $C$ be as in \eqref{bc}. Then there holds
    \begin{equation}\label{cbt}
        CB^T=\frac{ (\xi_1^r)^2 I + \oxi^r\otimes \oxi^r }{\xi_1^r}.
    \end{equation}
    In particular, $CB^T$ is positive definite.
\end{lem}
\begin{proof}[Proof of Lemma \ref{magic}]
    We first notice that by the definition of tensors,
    \begin{equation*}
        (\otheta\otimes \oxi^r)( \nabla F \otimes \oxi^r ) = (\otheta^T\oxi^r)((\nabla F)^T \oxi^r) = \otheta^T ( \oxi^r (\nabla F)^T ) \oxi^r = \langle \oxi^r, \nabla F \rangle ( \otheta \otimes \oxi^r ).
    \end{equation*}
    Use \eqref{xir} and the relation $\oxi^r = \otheta - \xi_1^r \nabla F$, and we find 
    \begin{equation}\begin{split}\label{xrnf}
        \langle \oxi^r, \nabla F \rangle 
        & = \langle \otheta - \xi_1^r \nabla F, \nabla F \rangle = \langle \otheta,\nabla F \rangle - \xi_1^r |\nabla F|^2 \\
        & = \frac{1+|\nabla F|^2}{2}\xi_1^r -|\nabla F|^2 \xi_1^r = \frac{ 1-|\nabla F|^2 }{2} \xi_1^r.
    \end{split}\end{equation}
    We now compute the product $CB^T$
    \begin{equation*}\begin{split}
        CB^T = &  \left( \xi_1^r I +\frac{ \otheta\otimes \oxi^r }{\langle \otheta, \nabla F \rangle} \right)\left( I-\frac{ \nabla F \otimes \oxi^r }{\xi_1^r} \right) \\
        = & \xi_1^r I +\frac{ \otheta\otimes \oxi^r }{ \langle \otheta, \nabla F \rangle } - \nabla F \otimes \oxi^r -\frac{ (\otheta\otimes \oxi^r)(\nabla F \otimes \oxi^r) }{\xi_1^r \langle \otheta,\nabla F \rangle } \\
        = & \xi_1^r I +\frac{ \otheta\otimes \oxi^r }{ \langle \otheta, \nabla F \rangle } - \nabla F \otimes \oxi^r -\frac{1-|\nabla F|^2}{2\langle \otheta, \nabla F \rangle}(\otheta \otimes \oxi^r) \\
        = & \xi_1^r I +\frac{ \otheta \otimes \oxi^r }{ \xi_1^r } - \nabla F\otimes \oxi^r \\
        = & \xi_1^r I +\frac{ (\otheta -\xi_1^r \nabla F) \otimes \oxi^r }{ \xi_1^r } \\
        = & \frac{ (\xi_1^r)^2 I + \oxi^r\otimes \oxi^r }{\xi_1^r}.
    \end{split}\end{equation*} 
    One can now see that $CB^T$ is symmetric. Moreover, for any $v\in \RR^{n-1}$, there holds
    \begin{equation*}
        \langle v,CB^T v \rangle = \frac{(\xi_1^r)^2 |v|^2 +|\langle \oxi^r, v \rangle|^2}{\xi_1^r} \geq \xi_1^r |v|^2.
    \end{equation*}
    Since $\xi_1^r>0$ on $\mathring\omega$, we conclude that $CB^T$ is positive definite.
\end{proof}

\begin{lem}\label{detb}
    Let $B$ as in \eqref{bc}. Then there holds
    \begin{equation*}
        \det(B) = \frac{1+|\nabla F|^2}{2}>0.
    \end{equation*}
    In particular, $B$ is invertible.
\end{lem}
\begin{proof}[Proof of Lemma \ref{detb}]
    We prove a slightly more general result. Let $a,b\in \RR^{n-1}$ be two row vectors. Then there holds 
    \begin{equation}\label{atb}
        \det(I+a\otimes b) = 1+\langle a,b \rangle.
    \end{equation}
    We first notice the following identities
    \begin{equation*}
        \begin{pmatrix}
            1 & -b \\ a^T & I
        \end{pmatrix}
        \begin{pmatrix}
            1 & 0 \\ -a^T & I
        \end{pmatrix}
        = \begin{pmatrix}
            1+ba^T & -b \\ 0 & I
        \end{pmatrix}, \ 
        \begin{pmatrix}
            1 & 0 \\ -a^T & I
        \end{pmatrix}
        \begin{pmatrix}
            1 & -b \\ a^T & I
        \end{pmatrix}
        =\begin{pmatrix}
            1 & -b \\ 0 & I+a^T b
        \end{pmatrix}.
    \end{equation*}
    Take determinants in both identities and we obtain 
    \begin{equation*}
        \begin{vmatrix}
            1 & -b \\ a^T & I
        \end{vmatrix}
        = \begin{vmatrix}
            1+ba^T & -b \\ 0 & I
        \end{vmatrix}
        = 1+ba^T, \ 
        \begin{vmatrix}
            1 & -b \\ a^T & I
        \end{vmatrix}
        = \begin{vmatrix}
            1 & -b \\ 0 & I+a^T b
        \end{vmatrix}
        =\det(I+a^T b).
    \end{equation*}
    Combining both identities of the determinants and recalling $ba^T=\langle a,b \rangle$, $a^Tb=a\otimes b$, we conclude that \eqref{atb} holds.

    Now put $a=-\frac{\oxi^r}{\xi_1^r}$, $b=\nabla F$ in \eqref{atb}, and we get 
    \begin{equation*}
        \det(B) = 1-\frac{\langle \oxi^r, \nabla F \rangle}{\xi_1^r} = 1-\frac{1-|\nabla F|^2}{2} = \frac{1+|\nabla F|^2}{2}.
    \end{equation*}
    Here we used \eqref{xrnf}.
\end{proof}

We are now ready to show that $A$ has a positive determinant. Indeed, recalling \eqref{abc}, we have
\begin{equation}\label{adecom}
    A=B\left( I-2sB^{-1}C \nabla^2 F \right) \ \Rightarrow \ \det(A)=\frac{1+|\nabla F|^2}{2}\det\left( I-2s B^{-1}C \nabla^2 F \right).
\end{equation}
Notice that
\[ B^{-1}C = B^{-1} (CB^T) (B^{-1})^T, \]
which implies that $B^{-1}C$ is positive definite since $CB^T$ is positive definite by Lemma \ref{magic}. Hence we can find an invertible matrix $L$ such that $B^{-1}C=LL^T$.
Since $F$ is concave, which implies that $\nabla^2 F$ is negative semi-definite, we know eigenvalues of $\nabla^2 F$ are non-positive. Use the identity
\[ B^{-1}C\nabla^2 F = LL^T(\nabla^2 F) = L \left(L^T(\nabla^2 F) L\right) L^{-1} \]
and we conclude that eigenvalues of $B^{-1}C\nabla^2 F$ are all non-positive. Using \eqref{adecom} and $s\geq 0$, we find that
\begin{equation*}
    \det(A)\geq \frac{1+|\nabla F|^2}{2}.
\end{equation*}
It now remains to recall \eqref{ja} to conclude that 
\begin{equation*}
    j(s,\ox,t^{\prime})=2\xi_1^r \det(A) \geq \xi_1^r (1+|\nabla F|^2) = 2\langle \otheta, \nabla F \rangle >0.
\end{equation*}
This completes the proof.
\end{proof}

\Remarks
1. The proof shows that the statement of Proposition \ref{nd-ref} can be made global, meaning that if $\mathcal O:=\{ (F(\ox), \ox) \ | \ \ox\in \RR^{n-1} \}$ with a strictly concave smooth function $F$ such that $F(0)=1$ and $\ox=0$ is the global maximum of $F$. Then Proposition \ref{nd-ref} holds with the restriction $|\ox|<r$ in \eqref{nomega} removed.

\noindent
2. Formula \eqref{adecom} and the fact that eigenvalues of $B^{-1}C\nabla^2 F$ are nonnegative implies that for fixed $\ox$, $t^{\prime}$, the Jacobian $j(s,\ox, t^{\prime})$ is non-decreasing as $s$ increases.

\subsection{Summary of the examples}

We summarize the examples we discussed in \S\S \ref{va2}--\ref{nd} in the following proposition.

\begin{prop}
    Suppose $\mathcal O\in \RR^n$ is defined by a function $F$ as in Definition \ref{r2}. Let $P=\Box$ be the wave operator \eqref{r1} on $M=(\RR^n\setminus \mathcal O)\times \RR_t$ and let $\phi_i=-t+\langle \otheta, \ox\rangle$ where $\otheta\in \mathbb S^{n-2}$. Set $\usigma=i^*\urho$, where $\urho=(1,0,t_0, 0,\otheta,-1)$ for any $t_0\in \RR$. 
    \begin{enumerate}[label={\arabic*.}]
        \item If $n=2$, then $\usigma\in \mathcal G_d$, in fact, \eqref{raa13} holds, and the conclusions of Theorem \ref{mt2} apply;

        \item If $n=3$ and $F$ satisfies Assumption \ref{A4} for some $k\in \mathbb N$, then $\usigma\in G_d^{2k}\setminus G^{2k+1}$, and the conclusions of Theorem \ref{mt2} apply;

        \item If $n\geq 2$ and $F$ satisfies Assumption \ref{ndA}, then $\usigma\in \mathcal G_d$, in fact, \eqref{r23} holds, and the conclusions of Theorem \ref{mt2} apply;

        \item Additionally, the conclusions of Theorem \ref{mt2} apply also to $3$-dimensional obstacles described by $F$ in \eqref{improve} in the Remark after Proposition \ref{r11}; and $n$-dimensional obstacles described by $F$ in \eqref{r4a} with $\otheta=(1,0)\in \mathbb S^{n-2}$.
    \end{enumerate}
\end{prop}

\appendix

\section{The forward flow map \texorpdfstring{$Z_r$}{TEXT} in the case \texorpdfstring{$\usigma\in G^2_d\setminus G^3$}{TEXT}.}\label{flowmap}

In this section we show that Assumption \ref{A3} is always satisfied when $\usigma\in G^2_d\setminus G^3$.

We work in $C^\infty$ almost standard form coordinates $(x,z,\lambda,\eta)$ for which $\partial_x\phi_i(0,z)=z_1$; recall \eqref{m10b} and \eqref{m18}.   Let  $$p(x,z,\lambda,\eta)=\lambda^2+q(x,z,\eta)$$ be the principal symbol of the main operator.   The bicharacteristic equations used to construct the reflected flow map $(s,y)\to Z_r(s,y)=(x(s,y),z(s,y))$  are 
\begin{align}\label{i1}
\begin{cases}
x_s=2\lambda,& x(0,y)=0,\\
z_s=\partial_\eta q,& z(0,y)=y,\\
\lambda_s=-\partial_x q,& \lambda(0,y)=-y_1\text{ where }y_1\leq 0, \\
\eta_s=-\partial_z q,& \eta(0,y)=\partial_y\phi_i(0,y)\text{ where }\partial_y\phi_i(0,0)=\ueta.
\end{cases}
\end{align}  
Let $\rho=(0,0,0,\underline{\eta})\in G^2_d\setminus G_3$.   
From \eqref{m11e} and \eqref{m19} we have
\begin{align*}
\alpha:=\partial_{\eta_1}q(0,0,\ueta)=-q_x(0,0,\ueta)>0.
\end{align*}

\begin{prop}\label{i3}
Let $\omega$ be the closure of an open neighborhood of $(0,0)$ in 
$\{(s,y)\ | \ s\geq 0,y_1\leq 0\}$,
and set $\mathring{\omega}:=\omega\cap \{y_1<0\}$.
If $\omega$ is small enough, the map $Z_r:\mathring{\omega}\to Z_r(\mathring{\omega})$ is a $C^\infty$ diffeomorphism, which extends to a homeomorphism $Z_r:{\omega}\to Z_r({\omega})$.

\end{prop}

\begin{proof}
\textbf{1. } Integrating the equations \eqref{i1} we obtain
\begin{subequations}\label{i4}
    \begin{align}
        & \begin{aligned}
            x(s,y)= & 2\int^s_0\lambda(t,y)dt=-2y_1s-2\int^s_0\int^t_0\partial_x q(x(r,y),z(r,y),\eta(r,y))drdt\\
            = & \alpha s^2-2y_1s+\eps_3(s,y),
        \end{aligned} \label{i4a}\\
        & z_1(s,y)=y_1+\int^s_0\partial_{\eta_1}q(x(t,y),z(t,y),\eta(t,y))dt=y_1+\alpha s+\eps^1_2(s,y), \label{i4b}\\
        & \begin{aligned}  
            z_j(s,y)= & y_j+\int^s_0\partial_{\eta_j}q(x(t,y),z(t,y),\eta(t,y))dt \\
            = & y_j+\partial_{\eta_j}q(0,0,\ueta)s+\eps^j_2(s,y), \ j=2,\dots n,
        \end{aligned} \label{i4c}\\
        & \lambda(s,y)=-y_1-\int^s_0\partial_x q(x(t,y),z(t,y),\eta(t,y))dt, \label{i4d}\\
        & \eta(s,y)=\partial_z\phi_i(0,y)-\int^s_0\partial_z q(x(t,y),z(t,y),\eta(t,y))dt. \label{i4e}
    \end{align}
\end{subequations}

\textbf{2. Estimate of the error terms.} Let
\begin{align*}
\begin{split}
Q(r,y) & :=-2\partial_x q(x(r,y),z(r,y),\eta(r,y))\text{ and } \\
Q_j(t,y) & :=\partial_{\eta_1}q(x(t,y),z(t,y),\eta(t,y)).
\end{split}
\end{align*}
Then we can rewrite 
\begin{align*}
\begin{split}
&\eps_3(s,y)=\int^s_0\int^t_0[Q(r,y)-Q(0,0)]drdt=\int^s_0\int^t_0[Q_1(r,y)r+Q_2(r,y)y]drdt,\\
&\eps^j_2(s,y)=\int^s_0[Q_j(t,y)-Q_j(0,0)]dt=\int^s_0[Q_{j1}(t,y)t+Q_{j2}(t,y)y]dt,
\end{split}
\end{align*}
for some smooth functions $Q_k$, $Q_{jk}$, $k=1,2$.    Obvious estimates of these integrals yield
\begin{align}\label{i5c}
\begin{gathered}
|\eps_3(s,y)|\lesssim s^3+s^2|y|, \; |\partial_s\eps_3|\lesssim s^2+s|y|,\;|\partial_y\eps_3|\lesssim s^2,\\
|\eps^j_2(s,y)|\lesssim s^2+|y|s,\; |\partial_s \eps^j_2|\lesssim s+|y|,\;|\partial_y \eps^j_2|\lesssim s.
\end{gathered}
\end{align}

\textbf{3. }A direct computation using \eqref{i4} and \eqref{i5c}  shows that the Jacobian determinant, $j(s,y)$, of the map 
$(s,y)\mapsto Z_r(s,y)=(x(s,y),z(s,y))$ satisfies
\begin{align}\label{jac}
j(s,y)=4\alpha s-2y_1+\eps_1(s,y)s,\text{ where }|\eps_1(s,y)|\lesssim |(s,y)|,
\end{align}
and thus $j(s,y)>0$ on $\mathring{\omega}$ if $\omega$ is small enough.    Thus, $Z_r$ is a local diffeomorphism on $\mathring{\omega}$.  

\textbf{4. $Z_r$ is injective on ${\omega}$. }Suppose $(s,y)$ and $(\os,\oy)$ lie ${\omega}$ and $Z_r(s,y)=Z_r(\os,\oy)$.    Using \eqref{i4b}--\eqref{i4d} this may be rephrased as:
\begin{subequations}\label{i6}
    \begin{align}
        & (s-\os)[\alpha(s+\os)-(y_1+\oy_1)]+\eps_3(s,y)-\eps_3(\os,\oy)=(s+\os)(y_1-\oy_1), \label{i6a}\\
        & y_1-\oy_1=\alpha(\os-s)+\eps^1_2(\os,\oy)-\eps^1_2(s,y), \label{i6b}\\
        & y_j-\oy_j=\gamma_j(\os-s)+\eps^j_2(\os,\oy)-\eps^j_2(s,y), \text{ where }\gamma_j:=\partial_{\eta_j}q(0,0,\ueta), \ 2\leq j\leq n. \label{i6c}
    \end{align}
\end{subequations}
We are free to switch $y_1$ and $\oy_1$, so from now on we assume 
\begin{align*}
y_1\leq \oy_1\leq 0.
\end{align*}
Observe that if all the error terms in \eqref{i6} are  set equal to zero, then \eqref{i6a} implies $s\leq \os$, while \eqref{i6b} implies $\os\leq s$.  Thus $s=\os$ and \eqref{i6b}, \eqref{i6c} imply $y=\oy$.\footnote{This observation was made in \cite{cheverry1996}, but the argument was incomplete because it did not treat the error terms.}

To treat the error terms we must estimate the error differences in \eqref{i6}.    We have 
\begin{align}\label{i8}
\begin{split}
& \eps_3(s,y)-\eps_3(\os,\oy)\\
& =  [\eps_3(s,y)-\eps_3(\os,y)]+[\eps_3(\os,y)-\eps_3(\os,\oy]\\
& =  \int^s_{\os}\int^t_0[Q(r,y)-Q(0,0)]drdt+\int^{\os}_{0}\int^t_0[Q(r,y)-Q(r,\oy)]drdt\\
& =  \int^s_{\os}\int^t_0[Q_1(r,y)r+Q_2(r,y)y]drdt+\int^{\os}_{0}\int^t_0[Q(r,y)-Q(r,\oy)]drdt.
\end{split}
\end{align}
From \eqref{i8} we can read off the estimate
\begin{align}\label{i9}
\begin{split}
|\eps_3(s,y)-\eps_3(\os,\oy)|\lesssim & |s^3-\os^3|+|y|(s^2-\os^2|+|y-\oy|\os^2 \\
\lesssim & |s-\os||(s,\os)|^2+|s-\os||(s,\os)||y|+\os^2|y-\oy|.
\end{split}
\end{align}
A similar estimate of the other differences yields
\begin{align}\label{i10}
|\eps^j_2(s,y)-\eps^j_2(\os,\oy)|\lesssim |s-\os||(s,\os)|+|s-\os||y|+\os|y-\oy|, \;j=1,\dots,n.
\end{align}
From \eqref{i6b}, \eqref{i6c} and \eqref{i10} we obtain
\begin{align}\label{i11}
|y-\oy|\lesssim |s-\os|+\os|y-\oy|\Rightarrow |y-\oy|\lesssim |s-\os|
\end{align}
if $\omega$ is small enough, after absorbing $\os|y-\oy|$ into the left side.  Using \eqref{i11} we can rewrite the inequalities \eqref{i9},\eqref{i10} as
\begin{subequations}\label{i12}
    \begin{align}
        & |\eps_3(s,y)-\eps_3(\os,\oy)|\lesssim |s-\os|\;\left(|(s,\os)|^2+|(s,\os)||y|\right)\lesssim |s-\os|\;|(s,\os)|\;|(s,\os,y)|, \label{i12a}\\
        & |\eps^j_2(s,y)-\eps^j_2(\os,\oy)|\lesssim |s-\os|\;|(s,\os,y)| \text{ for } 1\leq j\leq n. \label{i12b}
    \end{align}
\end{subequations}
If $\omega$ is small enough, \eqref{i12b} implies that the right side of \eqref{i6b} has the same sign as $\alpha(\os-s)$, so \eqref{i6b} implies $\os\leq s$.   Similarly, \eqref{i12a} implies that the left side of \eqref{i6b} has the same sign as 
$(s-\os)[\alpha(s+\os)-(y_1+\oy_1)]$.  Thus, \eqref{i6a} implies $s\leq \os$.   This implies $s=\os$, which by \eqref{i11} implies
$y=\oy$.

\textbf{5. }The flow map $Z_r:{\omega}\to Z_r({\omega})$  defined by the bicharacteristic equations \eqref{i1} is clearly continuous.  We have shown that $Z_r$ is a bijection onto its image, when ${\omega}$ is small enough.  The inverse is continuous provided $Z_r$ maps closed subsets of ${\omega}$ to closed sets.  That holds since ${\omega}$ is compact.
\end{proof}

\bibliographystyle{alpha}
\bibliography{GeoOptics}

\end{document}